\documentclass[12pt, twoside, b5 paper]{report}
\usepackage[T1]{fontenc}
\usepackage{titlesec, blindtext, color}
\usepackage{pdfpages}
\definecolor{gray75}{gray}{0.75}
\newcommand{\hsp}{\hspace{20pt}}
\titleformat{\chapter}[hang]{\Huge\bfseries}{\thechapter\hsp\textcolor{gray75}{|}\hsp}{0pt}{\Huge\bfseries}
\usepackage{setspace}
\usepackage[inner=2.5 cm,outer=2.0 cm,top=2.0 cm,bottom=2.0 cm]{geometry}
\usepackage{amsmath,amssymb,amsfonts,amsthm}

\newcounter{item}[section]
\newcounter{kirshr}
\newcounter{kirsha}
\newcounter{kirshb}
\newtheorem{theorem}{Theorem}[section]
\newtheorem{proposition}[theorem]{Proposition}

\newtheorem{example}[theorem]{Example}
\newtheorem{lemma}[theorem]{Lemma}
\newtheorem{corollary}[theorem]{Corollary}
\newtheorem{proviso}[theorem]{Proviso}
\newtheorem{conjecture}[theorem]{Conjecture}

\theoremstyle{definition}
\newtheorem{remark}[theorem]{Remark}

\newtheorem{definition}[theorem]{Definition}

\title{The Power of the Depth of Iteration in Defining Relations by Induction}
\author{By\\Amena Assem Abd-AlQader Mahmoud}

\date{}
\begin{document}
\onehalfspacing
\begin{titlepage}
\begin{center}
\Large \textbf{The Power of the Depth of Iteration in Defining Relations by Induction}\\[1.5cm]

Presented by\\ \large \textbf{Amena Assem Abd-AlQader Mahmoud}\\[1cm]

\normalsize A Thesis Submitted \\ to\\ \textbf{Faculty of Science} \\ In Partial Fulfillment of the\\ Requirements for \\ the Degree of \\ Master of Science \\ (Pure Mathematics)\\ Mathematics Department \\ Faculty of Science \\ Cairo University \\ (2014)
\end{center}
\end{titlepage}
\newpage
\thispagestyle{empty}
\mbox{}
.
\newpage
\thispagestyle{empty}
\mbox{}

\begin{center}
\large \textbf{APPROVAL SHEET FOR SUBMISSION}\\[1cm]
\end{center}
\textbf{Thesis Title}: The Power of the Depth of Iteration in Defining Relations by Induction.\\
\textbf{Name of candidate}: Amena Assem Abd-AlQader Mahmoud.\\[1cm]

This thesis has been approved for submission by the supervisors:\\\\
1- Prof. Ford Georgy\\[.5cm]
Signature:\\[2cm]
2- Prof. Wafik Boulos Lotfallah\\[.5cm]
Signature:\\[3cm]
\begin{flushright}
Prof.Dr. N.H. Sweilam.\\
Chairman of Mathematics Department\\
Faculty of Science - Cairo University
\end{flushright}

\newpage
\thispagestyle{empty}
\mbox{}
.
\newpage
\thispagestyle{empty}
\mbox{}

\begin{center}
\large \textbf{ABSTRACT}\\[1cm]
\end{center}
\textbf{Student Name}: Amena Assem Abd-AlQader Mahmoud.\\
\textbf{Title of the Thesis}: The Power of the Depth of Iteration in Defining Relations by Induction.\\
\textbf{Degree}: M.Sc. (Pure Mathematics).\\[0.1cm]

In this thesis we study inductive definitions over finite structures, particularly,
 the depth of inductive definitions. We also study infinitary finite variable logic which contains fixed-point logic and we introduce a new complexity measure $\textrm{FO}_{\bigvee}[f(n),g(n)]$ which counts the number, $f(n)$, of $\vee$-symbols, and the number, $g(n)$, of variables, in first-order formulas needed to express a given property. We prove that for $f(n)\geq \log{n}$, $\textrm{NSPACE}[f(n)] \subseteq \textrm{FO}_{\bigvee}[f(n)+\left(\frac{f(n)}{\log{n}}\right)^2,\frac{f(n)}{\log{n}}]$, and that for any $f(n),g(n)$, $\textrm{FO}_{\bigvee}[f(n),g(n)]\subseteq \textrm{DSPACE}[f(n)g(n)\log{n}]$. Also we study the expressive power of quantifier rank and number of variables and we prove that there is a property of words expressible with two variables and quantifier rank $2^n+2$ but not expressible with quantifier rank $n$ with any number of variables.\\[0.1cm]

\textbf{Keywords}: Fixed-Point; Depth; Finite Variable Logics.\\[0.2cm]

\textbf{Supervisors}:\hspace{6cm} \textbf{Signature:} \\
Prof. Ford Georgy \\\\
Prof. Wafik Boulos Lotfallah \\[0.2cm]
\begin{flushright}
Prof.Dr. N.H. Sweilam.\\
Chairman of Mathematics Department\\
Faculty of Science - Cairo University
\end{flushright}
\pagenumbering{roman}
\newpage
\thispagestyle{empty}
\mbox{}
.
\newpage
\thispagestyle{empty}
\mbox{}

\newpage
.
\thispagestyle{empty}
\mbox{}
.
\newpage
\tableofcontents
\newpage

\newpage
\thispagestyle{empty}
\mbox{}
.

\newpage

\pagenumbering{arabic}
\setcounter {chapter}{-1}
\chapter{Introduction}
\section{Background}
\begin{par}
First-order logic is weak, it is incapable of making inductive definitions. There are extensions of first-order logic in which inductive definitions are possible, namely, fixed-point extensions. For example, first-order logic cannot express the path relation, which is the transitive closure of the edge relation, in graphs. Actually it is unable to express transitive closure in general as noted by Aho and Ullman in 1979 (see \cite{Aho} and \cite{Gurevich}, in fact they noted this for relational calculus, which, from the point of view of expressive power, is exactly first-order logic). They then, Aho and Ullman, suggested extending the relational calculus by adding the least fixed-point operator.
\end{par}
 \begin{par}
 There are two fields where fixed-point extensions of first-order logic were extensively studied earlier. One is the theory of inductive definitions (for example by Moschovakis \cite{Moschovakis}, 1974). The other is semantics of programming languages where a fixed-point extension of first-order logic is known as first-order $\mu$-calculus. But neither of the two fields put finite structures into the center of attention.
 \end{par}
 \begin{par}
 An inductive definition needs a certain number of iterations before it closes; this number is called the depth. Immerman gave an elaborate definition of depth for positive inductive definitions over finite structures (see \cite{Immerman}), then he proved that inductive depth equals parallel-time and equals the depth of circuit in Circuit Complexity.
 \end{par}
 \begin{par}
  He also introduced another view of depth as the number of iterations of a quantifier block. He proved that the complexity class $\textrm{P}$ is exactly the set of boolean queries expressible by first-order quantifier blocks iterated polynomially and $\textrm{PSPACE}$ is exactly the set of boolean queries expressible by first-order quantifier blocks iterated exponentially.
   \end{par}
   \begin{par}
   The theorems proved by Immerman showed the importance of depth as a complexity measure. Here we study depth.
   \end{par}
\section{Summary}
\begin{par}
The thesis consists of three chapters. In the first chapter we introduce the preliminary definitions and facts we need from logic and complexity. The last section is on complexity, and the first section is for extra notations that are not established within the definitions of the thesis. The middle three sections are about finite structures, first-order logic and its fixed-point extensions, the  Ehrenfeucht-Fra\"{\i}ss\'{e} game and its importance in proving non-expressibility results.
 \end{par}
 \begin{par}
 Particularly, at the end of the third section, we mention an example from \cite{Flum} using the algebraic version of the game in proving non-expressibility of connectivity in first-order logic, and then, in the fourth section, we present fixed-point extensions of first-order logic in which connectivity is expressible, or in fact, in which the path relation, which is the transitive closure of the edge relation, in graphs, and transitive closure in general, and more complicated kinds of recursion, are expressible.
 \end{par}
 \begin{par}
The second section is devoted to finite structures, especially, graphs and binary strings. We deal here with finite structures only because the objects computers have and hold are always finite. Inputs, databases, programs are all finite objects that can be conveniently modeled as finite logical structures. Binary strings are important because every finite ordered structure can be coded as a binary string, and this is how the structure is introduced as an input to the Turing machine.
 \end{par}
 \begin{par}
 Graphs are important because every finite structure may be thought of as a graph, or as Immerman expressed it: "Everything is a Graph" \cite{Immerman}. Later in the first chapter we will be able to define precisely what we mean by saying that every finite structure may be thought of as a graph.
  \end{par}
  \begin{par}
  Graphs are important also because the computation of a Turing machine can be represented as a graph, called the configuration graph, in which vertices represent possible configurations (the sum of machine's state and work-tape inscription and positions of heads) of the machine and edges represent the possibility of a transition, in one step, by the machine's transition function from that configuration to the other.
  \end{par}
  \begin{par}
Those configuration graphs showed
their importance in the proofs of completeness of particular graph queries for certain complexity classes. For example, in the proofs of the completeness of reachability for nondeterministic logspace and the completeness of alternating reachability for deterministic polynomial time, one obtains from a structure (being tested by the machine for its satisfaction of a query from one of these complexity classes) the configuration graph of the machine's computation on it and then the problem is proved to be reducible to or equivalent to testing that graph for its satisfaction of the respective graph query. We used configuration graphs in the proof of theorem \ref{3.3.10}.
\end{par}
\begin{par}
 Another importance of graphs is that one of the most famous models of computation and its measures of complexity are defined in terms of graphs and properties of graphs, namely, Boolean Circuits, and measuring complexity via their depths.
 \end{par}
\begin{par}
The second chapter is devoted to the notion of depth of inductive definitions over finite structures. Immerman defined depth for positive formulas only. Here we define depth for all formulas which have a fixed-point and set it to be $\infty$ if the fixed-point does not exist.
 \end{par}
 \begin{par}
 We also define the inflationary-depth of a formula as the depth of the inflationary formula obtained from it. These definitions are in the first section then, in the second section, we exhibit the basic theorems relating depth to complexity classes from Immerman's "Descriptive Complexity" \cite{Immerman}.
 \end{par}
 \begin{par}
  In the last section we investigate the relationship between the depth of a formula and its inflationary-depth and find that no one of them is always greater than the other, i.e., there are formulas for which the depth is greater than the inflationary-depth and formulas for which the inflationary-depth is greater than the depth.
   \end{par}
   \begin{par}
   We investigate also the relationship between the depth of the simultaneous fixed-point of two formulas and the depth of their nested fixed-point and show that no one of them is always greater than the other if the relation variables through which the fixed-point is computed are not restricted to be positive.
   \end{par}
  \begin{par}
What motivated us to compare the depth of the nested fixed-point and the depth of the simultaneous fixed-point of the same two formulas is that in the proof of Lemma 8.2.6 of \cite{Flum} they show that under some conditions on two formulas $\varphi$ and $\psi$, we have that their nested fixed-point equals their simultaneous fixed-point. We wanted to see which one of the two ways, nested or simultaneous, to reach that fixed-point, consumes more time or needs a bigger number of iterations.
\end{par}
\begin{par}
In the third chapter, in the first section, we present finite variable logics and their respective games (pebble games). In the second section we study the expressive power of number of variables and quantifier rank, motivated by the theorems, exhibited in Section 2.2 from Immerman's "Descriptive Complexity", relating respectively the number of variables in an inductive definition and its depth (which turns out to be equal to the number of iterations of a quantifier block) to the number of processors in a parallel machine and its parallel-time.
 \end{par}
 \begin{par}
 We prove that there is a property of words expressible with two variables and quantifier rank $2^n+2$ but not expressible with quantifier rank $n$ with any number of variables (Proposition \ref{3.3.5}).
  \end{par}
  \begin{par}
  In the third section we focus on infinitary finite variable logic $\textrm{L}_{\infty\omega}^{\omega}$ which contains fixed-point logic and we introduce a rough relationship between the depth of an inductive definition in structures of size $n$ and the number of $\vee$-symbols in the first-order formula expressing that inductive definition in structures of size $n$.
  \end{par}
  \begin{par}
   We introduce a new complexity measure $\textrm{FO}_{\bigvee}[f(n),g(n)]$ which counts the number, $f(n)$, of $\vee$-symbols, and the number, $g(n)$, of variables, in first-order formulas needed to express a given property. We prove that for $f(n)\geq \log{n}$, $\textrm{NSPACE}[f(n)] \subseteq \textrm{FO}_{\bigvee}[f(n)+\left(\frac{f(n)}{\log{n}}\right)^2,\frac{f(n)}{\log{n}}]$, and that for any $f(n),g(n)$, $\textrm{FO}_{\bigvee}[f(n),g(n)]\subseteq \textrm{DSPACE}[f(n)g(n)\log{n}]$.
    \end{par}
    \begin{par}
    In the last section we talk more about pebble games and depth. We ask two questions and suggest a conjecture concerning depth.
    \end{par}

\section{List of Results}
\begin{enumerate}
\item[(1)] Examples showing that, in general, there is no order comparison between the depth of a formula and its inflationary depth, and no order comparison between the depth of the iteration it takes to come to the simultaneous fixed-point and the depth of the iteration it takes to come to the nested fixed-point of two formulas if the relation variables through which the induction is made are not restricted to be positive. These examples are exhibited in Section 2.3 which is devoted for them, and therefore, titled \textbf{Examples}.
\item[(2)] Proposition \ref{3.3.5}, where we prove that there is a property of words expressible with two variables and quantifier rank $2^n+2$ but not expressible with quantifier rank $n$ with any number of variables.
\item[(3)] Theorem \ref{3.3.10}, where we prove that for $f(n)\geq \log{n}$, $$\textrm{NSPACE}[f(n)] \subseteq \textrm{FO}_{\bigvee}[f(n)+\left(\frac{f(n)}{\log{n}}\right)^2,\frac{f(n)}{\log{n}}].$$
\item[(4)] Theorem \ref{3.3.11}, where we prove that for any $f(n),g(n)$,
$$\textrm{FO}_{\bigvee}[f(n),g(n)]\subseteq \textrm{DSPACE}[f(n)g(n)\log{n}].$$
\end{enumerate}

\newpage
\thispagestyle{empty}
\mbox{}
.

\chapter{Preliminaries}

This is a chapter with the preliminaries we need from logic and complexity. It is mainly from \cite{Flum} and \cite{Immerman}.
\section{Notation}
Most of our notation is from \cite{Flum} and \cite{Immerman}. Almost all notations are established within the definitions in the thesis, except a few things which we mention here.$\newline$
We write $:=$ to mean \emph{equals by definition}. $\newline$
$\mathbb{N}$ denotes the natural numbers $\{0,1,\ldots\}$. $\newline$
We write $\backslash$ to denote the difference of sets, i.e., $A\backslash B := \{x\in A \; | \; x \notin B\}$. For a finite set $A$, we write $||A||$ to denote the number of elements in it. $\newline$ For a function $f$, we write $do(f)$ and $rg(f)$ to denote its domain and range respectively. We write $(f)^r(x)$ to denote the result of applying $f$ $r$ times on $x$, i.e., $\overset{r \; \text{times}}{\overbrace{f(f(\ldots f}}(x)\ldots))$. For two functions $f,g$ from $\mathbb{N}$ to $\mathbb{N}$, we write $f(n) = O(g(n))$ to mean that there is a constant $k$ such that $f(n)\leq k \cdot g(n)$ for every sufficiently large $n$, and write $f(n)=\Theta(g(n))$ to mean that $f(n)=O(g(n))$ and $g(n)=O(f(n))$. $\newline$
For a binary relation $\sim$, we write $\nsim$ to denote its complement.$\newline$
We write $\overline{x}$ to denote the tuple $(x_1,\ldots,x_n)$ when $n$ is understood from the context, and when $\overline{a}$ and $\overline{b}$ are of the same length n, we write $\overline{a} \mapsto \overline{b}$ to denote the map from $\{a_1,\ldots,a_n\}$ to $\{b_1,\ldots,b_n\}$ under which the image of $a_i$ is $b_i$ for every $i \in \{1,\ldots,n\}$. $\newline$
For a symbol $s$ we sometimes write $s^k$ to denote the string $\overset{k\;\text{times}}{\overbrace{s\; s \ldots s}}$.$\newline$
We sometimes write $M(w)\downarrow$ to mean that Turing Machine $M$ accepts input $w$.

\section{Finite Structures}

\begin{definition} (Structures)
\begin{enumerate}
\item[(1)]\emph{Vocabularies} are finite sets that consist of \emph{relation symbols} $P, Q, R,\ldots$ and \emph{constant} \emph{symbols} (for short: \emph{constants}) $c,d,\ldots$. Every relation symbol is equipped with a natural number $\geq 1$, its \emph{arity}. We denote vocabularies by $\tau,\sigma,\ldots$. A vocabulary is \emph{relational}, if it does not contain constants.

\item[(2)] A \emph{structure} $\mathcal{A}$ of vocabulary $\tau$ (by short: a $\tau$-\emph{structure}) consists of a nonempty set $A$, the \emph{universe} or \emph{domain} of $\mathcal{A}$ (also denoted by $|\mathcal{A}|$), of an $n$-ary relation $R^{\mathcal{A}}$ on $A$ for every $n$-ary relation symbol $R$ in $\tau$, and of an element $c^{\mathcal{A}}$ of $A$ for every constant $c$ in $\tau$.

\item[(3)] An $n$-ary \emph{relation} $S$ on $A$ is a subset of $A^n$, the set of $n$-tuples of elements of $A$. We mostly write $Sa_1 \ldots a_n$ instead of $(a_1,\ldots,a_n)\in S$.

\item[(4)] A structure $\mathcal{A}$ is \emph{finite}, if its universe $A$ is a finite set. We denote the cardinality of $A$ if it is finite by $\|\mathcal{A}\|$.
\end{enumerate}
\end{definition}

\begin{proviso}
All structures in the thesis are assumed to be finite.
\end{proviso}

\begin{definition}(Orderings)
\begin{enumerate}
\item[(a)]Let $\tau = \{<\}$ with a binary relation symbol $<$. A $\tau$-structure $\newline$ $\mathcal{A} = (A, <^{\mathcal{A}})$ is called an \emph{ordering} if for all $a,b,c \in A$ :
    \begin{enumerate}
        \item[(1)] not $a<^{\mathcal{A}} a$.
        \item[(2)] $a<^{\mathcal{A}} b$ or $b<^{\mathcal{A}} a$ or $a=b$.
        \item[(3)] if $a<^{\mathcal{A}} b$ and $b<^{\mathcal{A}} c$ then $a<^{\mathcal{A}} c$.
    \end{enumerate}
\item[(b)] Let $S$ be a binary relation symbol (representing the successor relation), and $min$ and $max$ constants (for the first and last element of the ordering). A finite $\{<,S,min,max\}$-structure $\mathcal{A}$ is an ordering if, in addition to $(1),(2),(3)$, for all $a,b \in A$ :
    \begin{enumerate}
        \item[(4)] $S^{\mathcal{A}} ab$ iff ($a<^{\mathcal{A}}b$ and for all $c$, if $a<^{\mathcal{A}}c$ then $b<^{\mathcal{A}}c$ or $b=c$).
        \item[(5)] $min^{\mathcal{A}} <^{\mathcal{A}} a$ or $min^{\mathcal{A}}=a$.
        \item[(6)] $a<^{\mathcal{A}} max^{\mathcal{A}}$ or $a=max^{\mathcal{A}}$.
    \end{enumerate}
\item[(c)] Suppose that $\tau_0$ is a vocabulary with $\{<\} \subseteq \tau_0 \subseteq \{<,S,min,max\}$ and let $\sigma$ be an arbitrary vocabulary with $\tau_0 \subseteq \sigma$. A finite $\sigma$-structure $\mathcal{A}$ is said to be \emph{ordered}, if the \emph{reduct} $\mathcal{A}|\tau_0$ (i.e., the $\tau_0$-structure obtained from $\mathcal{A}$) by forgetting the interpretations of the symbols in $\sigma \backslash \tau_0$) is an ordering.
\end{enumerate}
\end{definition}

\begin{proviso}
Unless stated otherwise, we identify the universe of structures of size $n$ with $\{0,\ldots,n-1\}$, and for orderings we assume that $<$ is interpreted with the usual ordering inherited from $\mathbb{N}$.
\end{proviso}

\begin{definition} (Isomorphism)
Fixing a vocabulary $\tau$, two $\tau$-structures $\mathcal{A}$ and $\mathcal{B}$ are \emph{isomorphic}, written $\mathcal{A}\cong\mathcal{B}$ if there is an \emph{isomorphism} from $\mathcal{A}$ to $\mathcal{B}$, i.e., a bijection $\pi : \mathcal{A} \rightarrow \mathcal{B}$ preserving relations and constants, that is,
for any $n$-ary $R\in \tau$ and $a_1,\ldots,a_n\in A$,
    $R^{\mathcal{A}} a_1 \ldots a_n \;\; \text{iff} \;\; R^{\mathcal{B}} \pi(a_1)\ldots\pi(a_n)$ and,
for any constant $c\in \tau, \; \pi(c^{\mathcal{A}})=c^{\mathcal{B}}$.

\end{definition}

\begin{definition}(Partial Isomorphisms)$\newline$
Assume $\mathcal{A}$ and $\mathcal{B}$ are structures. Let $p$ be a map with $do(p) \subseteq A$ and $rg(p) \subseteq B$, where $do(p)$ and $rg(p)$ denote the domain and the range of $p$, respectively. Then $p$ is said to be a \emph{partial isomorphism} from $\mathcal{A}$ to $\mathcal{B}$ if
\begin{enumerate}
\item[-] $p$ is injective.
\item[-] for every $c \in \tau$ : $c^{\mathcal{A}} \in do(p)$ and $p(c^{\mathcal{A}})=c^{\mathcal{B}}$.
\item[-] for every $n$-ary $R \in \tau$ and all $a_1,\ldots,a_n \in do(p)$,
                             $$R^{\mathcal{A}} a_1 \ldots a_n \;\;\;\;\; \text{iff} \;\;\;\;\; R^{\mathcal{B}} p(a_1) \ldots p(a_n).$$
\end{enumerate}

We write $\textrm{Part}(\mathcal{A},\mathcal{B})$ for the set of partial isomorphisms from $\mathcal{A}$ to $\mathcal{B}$.
\end{definition}

\begin{definition} (Disjoint Union) $\newline$
For relational $\tau$, the \emph{disjoint union} of $\tau$-structures $\mathcal{A}$ and $\mathcal{B}$ with $A\cap B=\emptyset$ is the structure $\mathcal{A}\dot{\cup}\mathcal{B}$ with universe $A\cup B$ and $$R^{\mathcal{A}\dot{\cup}\mathcal{B}}:= R^{\mathcal{A}}\cup R^{\mathcal{B}}\;\; \text{ for any} \; R \;\text{in}\;\tau.$$
In case $\mathcal{A}$ and $\mathcal{B}$ are structures with $A \cap B \neq \emptyset$, we take isomorphic copies $\mathcal{A}^\prime$ of $\mathcal{A}$ and $\mathcal{B}^\prime$ of $\mathcal{B}$ with disjoint universes (e.g., with universes $A\times \{1\}$ and $B\times \{2\}$) and set $\mathcal{A}\dot{\cup}\mathcal{B}:= \mathcal{A}^\prime\dot{\cup}\mathcal{B}^{\prime}$.
\end{definition}

Now we go to a particular kind of structures called graphs. Visually speaking, a graph is a collection of points or vertices linked by line segments or edges.

\begin{definition} (Graphs) $\newline$
Let $\tau = \{E\}$ where $E$ is a binary relation symbol. A \emph{graph} (or, \emph{undirected graph}) is a $\tau$-structure $\mathcal{G}=(G, E^{\mathcal{G}})$ satisfying
\begin{enumerate}
\item[(1)] for all $a\in G$ : not $E^{\mathcal{G}} aa$.
\item[(2)] for all $a,b \in G$ : if $E^{\mathcal{G}} ab$ then $E^{\mathcal{G}} ba$.
\end{enumerate}
By \textrm{GRAPH} we denote the class of graphs. If only (1) is required, we speak of a \emph{digraph} (or, \emph{directed } \emph{graph}). The elements of $G$ are sometimes called \emph{points} or \emph{vertices}, the elements of $E^{\mathcal{G}}$ \emph{edges}.
\end{definition}

\begin{definition} \label{1.1.6} $\newline$
\begin{enumerate}
\item[(1)] Let $\mathcal{G}$ be a digraph. If $n\geq1$ and
$$E^{\mathcal{G}} a_0a_1, E^{\mathcal{G}} a_1a_2, \ldots, E^{\mathcal{G}} a_{n-1}a_n$$ then $a_0,\ldots,a_n$ is a \emph{path} from $a_0$ to $a_n$ of \emph{length} $n$ (in undirected graphs this is also a path from $a_n$ to $a_0$).
\item[(2)] If $a_0=a_n$ then $a_0,\ldots,a_n$ is a \emph{cycle}. $\mathcal{G}$ is \emph{acyclic} if it has no cycle.
\item[(3)] Let $\mathcal{G}$ be a graph. Write $a\sim b$ if $a=b$ or if there is a path from $a$ to $b$. Clearly, $\sim$ is an equivalence relation. The equivalence class of $a$ is called the (\emph{connected}) \emph{component} of $a$. $\mathcal{G}$ is \emph{connected} if $a\sim b$ for all $a,b \in G$, that is, if there is only one connected component. Let \textrm{CONN} be the class of connected graphs.
\item[(4)] Denote by $d(a,b)$ the shortest length of a path from $a$ to $b$; more precisely, define the \emph{distance } \emph{function} $d: G\times G \rightarrow \mathbb{N} \cup \{\infty\}$ by
    $$d(a,b)=\infty\;\;\;\text{iff}\;\;\;a\nsim b\; ;\qquad d(a,b)=0\;\;\;\text{iff}\;\;\;a=b\; ;$$and otherwise,$$d(a,b)=\mathrm{min}\{n\geq 1\;|\;\text{there is a path from }a\;\text{to }b\;\text{of length }n\}.$$
\end{enumerate}
\end{definition}

Having seen graphs let us now define strings and see how every ordered structure can be coded as a binary string (which makes it possible for a structure to be introduced as an input for a Turing machine).

\begin{definition}
A \emph{string} (also called a \emph{word}) of length $n$ on $k \geq 1$ symbols is an ordered structure of size $n$, over the vocabulary $\newline \tau = \{< , R_1,\ldots,R_{k}\}$, in which the unary relation $R_i$, $1\leq i \leq k$, is interpreted with the set of elements that represent positions in which the $i$-th symbol occurs. (Of course a string on $k$ symbols can also be represented as an ordered structure on a vocabulary with just $k-1$ unary relation symbols).
\end{definition}

For example the binary string $w = 011001$ is represented by the following ordered structure : $\mathcal{W} = (W,<^{\mathcal{W}},R^{\mathcal{W}})$, where $W=\{0,1,2,3,4,5\}$, and $R^{\mathcal{W}}=\{1,2,5\}$ (the positions in which $1$ occurs).

\begin{definition}(Binary Encoding of Structures)$\newline$
Let $\tau =\{<, R_1,\ldots,R_r,c_1,\ldots,c_s\}$ where each $R_i$ is of arity $a_i$, and let $\newline$ $\mathcal{A} = (\{0,1,\ldots,n-1\}, <^{\mathcal{A}}, R_1^{\mathcal{A}},\ldots,R_r^{\mathcal{A}},c_1^{\mathcal{A}},\ldots,c_s^{\mathcal{A}})$ be an ordered structure of vocabulary $\tau$. The relation $R_i^{\mathcal{A}}$ is a subset of $|\mathcal{A}|^{a_i}$, and this contains exactly $n^{a_i}$ tuples. We encode this relation as a binary string $bin^{\mathcal{A}}(R_i)$ of length $n^{a_i}$ where $``1"$ in a given position indicates that the corresponding tuple is in $R_i^{\mathcal{A}}$. Similarly, for each constant $c_j^{\mathcal{A}}$, its number is encoded as a binary string $bin^{\mathcal{A}}(c_j)$ of length $\left\lceil \log n \right\rceil$ (this binary string is its normal representation in binary). The binary encoding of the structure $\mathcal{A}$ is then just the concatenation of these binary strings coding its relations and constants,
$$bin(\mathcal{A}) = bin^{\mathcal{A}}(R_1)bin^{\mathcal{A}}(R_2) \ldots bin^{\mathcal{A}}(R_r)bin^{\mathcal{A}}(c_1)\ldots bin^{\mathcal{A}}(c_s)$$
We do not need any separators between the various relations and constants because the vocabulary $\tau$ and the length of $bin(\mathcal{A})$ determine where each section belongs. Observe that the length of $bin(\mathcal{A})$ is given by
      $$||bin(\mathcal{A})|| = n^{a_1}+\ldots + n^{a_r}+s \left\lceil \log n \right\rceil$$
In the special case where $\tau$ includes no relation symbols other than $<$, we pretend that there is a unary relation symbol that is always false. For example, if $\tau = \{<\}$, then $bin(\mathcal{A})=0^{||\mathcal{A}||}$. We do this to insure that the size of $bin(\mathcal{A})$ is at least as large as $||\mathcal{A}||$.
\end{definition}
\begin{example}
Consider the graph $\mathcal{G}:=(G,E^{\mathcal{G}},s^{\mathcal{G}},t^{\mathcal{G}})$, where, $G=\{0,1,2\}$, $E^{\mathcal{G}}=\{(0,1),(1,2)\}$, $s^{\mathcal{G}}=0$ and $t^{\mathcal{G}}=2$. Then its binary code is $010001000010$.
\end{example}

\section{First-Order Logic and Games}
\begin{definition} (Syntax of First-Order Logic)
Fix a vocabulary $\tau$. Each formula of first-order logic will be a string of symbols taken from the alphabet consisting of
\begin{enumerate}
\item[-] $v_1,v_2,v_3,\ldots$ (the \emph{variables})
\item[-] $\neg,\vee$ (the \emph{connectives not, or})
\item[-] $\exists$ (the \emph{existential quantifier})
\item[-] $=$ (the \emph{equality symbol})
\item[-] $),($
\item[-]the symbols in $\tau$.
\end{enumerate}

A \emph{term} of vocabulary $\tau$ is a variable or a constant in $\tau$. Henceforth, we often shall use the letters $x,y,z,\ldots$ for variables and $t,t_1,\ldots$ for terms.$\newline$
The \emph{formulas} of first-order logic of vocabulary $\tau$ are those strings which are obtained by finitely many applications of the following rules:
    \begin{enumerate}
       \item[(F1)] If $t_0$ and $t_1$ are terms then $t_0=t_1$ is a formula.
       \item[(F2)] If $R$ in $\tau$ is $n$-ary and $t_1,\ldots,t_n$ are terms then $Rt_1 \ldots Rt_n$ is a formula.
       \item[(F3)] If $\varphi$ is a formula then $\neg \varphi$ is a formula.
       \item[(F4)] If $\varphi$ and $\psi$ are formulas then $(\varphi \vee \psi)$ is a formula.
       \item[(F5)] If $\varphi$ is a formula and $x$ a variable then $\exists x \varphi$ is a formula.
    \end{enumerate}
Formulas obtained by (F1) or (F2) are called \emph{atomic} formulas. For formulas $\varphi$ and $\psi$ we use $(\varphi\wedge\psi), (\varphi\rightarrow\psi), (\varphi\leftrightarrow\psi)$, and $\forall x \varphi$ as abbreviations for the formulas $\neg(\neg \varphi \vee \neg \psi), (\neg\varphi\vee\psi), ((\neg\varphi\vee\psi) \wedge (\neg \psi \vee\varphi))$, and $\neg \exists x \neg \varphi$, respectively.$\newline$
Denote by $\textrm{FO}[\tau]$ the set of formulas of first-order logic of vocabulary $\tau$; and denote it by just $\textrm{FO}$ when the vocabulary is understood from the context.
\end{definition}

\begin{definition} (free variables) $\newline$
$free(\varphi)$ the set of \emph{free} variables of a formula $\varphi$ :
           \begin{enumerate}
              \item[-] if $\varphi$ is atomic then the set $free(\varphi)$ is the set of variables occurring in $\varphi$.
              \item[-]$free(\neg \varphi) := free(\varphi)$
              \item[-]$free(\varphi\vee\psi):= free(\varphi)\cup free(\psi)$
              \item[-]$free(\exists x \varphi):=free(\varphi)\backslash \{x\}$.
           \end{enumerate}
We write $\varphi(x_1,\ldots,x_n)$ to indicate that $x_1,\ldots,x_n$ are distinct and $\newline free(\varphi) \subseteq \{x_1,\ldots,x_n\}$ without implying that all $x_i$ are actually free in $\varphi$.
A \emph{bound} occurrence of a variable in a formula is an occurrence that lies in the scope of a corresponding quantifier.
A \emph{sentence} is a formula without free variables.
\end{definition}

\begin{definition} (Semantics of First-Order Logic) $\newline$
 Let $\mathcal{A}$ be a $\tau$-structure. An \emph{assignment} in $\mathcal{A}$ is a function $\alpha$ with domain the set of variables and with values in $A$, $\alpha : \{v_n \; | \; n\geq 1\}\rightarrow A$. Extend $\alpha$ to a function defined for all terms by setting $\alpha(c):= c^{\mathcal{A}}$ for all constants in $\tau$. Denote by $\alpha\frac{a}{x}$ the assignment that agrees with $\alpha$ on all variables except that $\alpha\frac{a}{x}(x)=a$.\\
    We define the relation $\mathcal{A}\vDash \varphi[\alpha]$, called the satisfaction relation (``the assignment $\alpha$ \emph{satisfies} the formula $\varphi$ in $\mathcal{A}$" or ``$\varphi$ is \emph{true} in $\mathcal{A}$ under $\alpha$"), as follows:\\
    $\mathcal{A}\vDash t_1=t_2 [\alpha]\;\;\;\; \text{iff} \;\;\;\alpha(t_1)=\alpha(t_2)$\\
    $\mathcal{A} \vDash Rt_1...t_n[\alpha] \;\;\; \text{iff} \;\;\; R^{\mathcal{A}}\alpha(t_1)...\alpha(t_n)$\\
    $\mathcal{A}\vDash \neg \varphi[\alpha] \;\;\;\;\;\; \;\;\;\;\text{iff} \;\;\; \text{not}\; \mathcal{A}\vDash\varphi[\alpha]$\\
    $\mathcal{A}\vDash (\varphi\vee\psi)[\alpha] \;\;\; \text{iff} \;\;\; \mathcal{A}\vDash\varphi[\alpha] \;\text{or}\; \mathcal{A}\vDash\psi[\alpha]$\\
    $\mathcal{A}\vDash \exists x\varphi[\alpha]\;\;\; \;\;\;\;\;\text{iff} \;\;\;\text{there is an} \; a\in A\; \text{such that} \; \mathcal{A}\vDash\varphi[\alpha\frac{a}{x}]$\\\\
    Note that the truth or falsity of $\mathcal{A}\vDash\varphi[\alpha]$ depends only on the values of $\alpha$ for those variables $x$ which are free in $\varphi$. That is, if $\alpha_1(x)=\alpha_2(x)$ for all $x\in free(\varphi)$, then $\mathcal{A}\vDash\varphi[\alpha_1]$ iff $\mathcal{A}\vDash \varphi[\alpha_2]$. Thus, if $\varphi=\varphi(x_1,\ldots,x_n)$ and $a_1=\alpha(x_1),\ldots,a_n=\alpha(x_n)$, then we may write $\mathcal{A}\vDash\varphi[a_1,\ldots,a_n]$ for $\mathcal{A}\vDash\varphi[\alpha]$. In particular, if $\varphi$ is a sentence, then the truth or falsity of $\mathcal{A}\vDash\varphi[\alpha]$ is completely independent of $\alpha$. Thus we write $\mathcal{A}\vDash\varphi$ (read : $\mathcal{A}$ is a \emph{model} of $\varphi$, or $\mathcal{A}$ \emph{satisfies} $\varphi$), if for some (hence every) assignment $\alpha$, $\mathcal{A}\vDash \varphi[\alpha]$.
Fromulas $\varphi$ and $\psi$ are \emph{equivalent} if $\varphi\leftrightarrow\psi$ is true in all structures under all assignments (we sometimes write $\varphi \equiv \psi$ to express this).
\end{definition}

It is assumed that first-order logic contains two zero-ary relation symbols $\textrm{T}, \textrm{F}$. In every structure, $\textrm{T}$ and $\textrm{F}$ are interpreted as $\textrm{TRUE}$ (i.e., as being true) and $\textrm{FALSE}$, respectively ($\textrm{TRUE}$ corresponds to the zero-ary relation $\{\emptyset\}$, and $\textrm{FALSE}$ to the zero-ary relation $\emptyset$). Hence, the atomic formula $\textrm{T}$ is equivalent to $\exists x (x=x)$ and $\textrm{F}$ to $\neg \exists x (x=x)$.$\newline$

If $\Phi = \{\varphi_1,\ldots,\varphi_n\}$ we sometimes write $\bigwedge \Phi$ for $\varphi_1 \wedge \ldots \wedge \varphi_n$ and $\bigvee \Phi$ for $\varphi_1 \vee\ldots \vee\varphi_n$. In case $\Phi = \emptyset$ we set $\bigwedge \Phi =\textrm{T}$ and $\bigvee \Phi = \textrm{F}$. Then, for arbitrary finite $\Phi$,
$$\mathcal{A} \vDash \bigwedge \Phi \;\;\; \text{iff} \;\;\; \text{for all}\; \varphi \in \Phi, \; \mathcal{A}\vDash \varphi. $$

\begin{definition} (Quantifier Rank) $\newline$
The \emph{quantifier rank} $qr(\varphi)$ of a formula $\varphi$ is the maximum number of nested quantifiers occurring in it :
\begin{enumerate}
\item[-] $qr(\varphi):=0$, if $\varphi$ is atomic
\item[-] $qr(\neg\varphi):=qr(\varphi)$
\item[-] $qr(\varphi\vee\psi):= max\{qr(\varphi), qr(\psi)\}$
\item[-] $qr(\exists x \varphi):= qr(\varphi)+1$
\end{enumerate}
\end{definition}

\begin{definition}
For structures $\mathcal{A}$ and $\mathcal{B}$ and $m\in \mathbb{N}$ we write $\mathcal{A} \equiv_m \mathcal{B}$ and say that $\mathcal{A}$ and $\mathcal{B}$ are $m$-equivalent, if $\mathcal{A}$ and $\mathcal{B}$ satisfy the same first-order sentences of quantifier rank $\leq m$. They are \emph{elementarily equivalent} if they satisfy the same first-order sentences, in symbols, $\mathcal{A} \equiv \mathcal{B}$.
\end{definition}
\begin{definition}
We say that a class $K$ of structures is \emph{definable} (or that its characteristic property is \emph{expressible}) in a logic $\mathcal{L}$ if there is a sentence $\varphi \in \mathcal{L}$ such that for every structure $\mathcal{A}$ : $\mathcal{A} \in K \;\;\; \text{iff} \;\;\; \mathcal{A}\vDash \varphi$.
\end{definition}
\begin{proviso}
Throughout the text all classes $K$ of structures considered will tacitly be assumed to be closed under isomorphisms, i.e.$\newline$ $\mathcal{A}\in K$ and $\mathcal{A}\cong\mathcal{B}$ imply $\mathcal{B}\in K$.
\end{proviso}

Now we define the \emph{Ehrenfeucht-Fra\"{i}\emph{ss}\'{e} games}.

\begin{definition} (Games : Ehrenfeucht)$\newline$
Let $\mathcal{A}$ and $\mathcal{B}$ be structures, $\overline{a} \in A^s$, $\overline{b} \in B^s$, and $m\in \mathbb{N}$. The \emph{Ehrenfeucht} \emph{game} $G_m(\mathcal{A}, \overline{a}, \mathcal{B}, \overline{b})$ is played by two players called \emph{Spoiler} and \emph{Duplicator}. Each player has to make $m$ moves in the course of a play. The players take turns. In the $i$-th move Spoiler first selects a structure, $\mathcal{A}$ or $\mathcal{B}$, and an element in this structure. If Spoiler chooses $e_i$ in $\mathcal{A}$ then Duplicator in its $i$-th move must choose an element $f_i$ in $\mathcal{B}$. If Spoiler chooses $f_i$ in $\mathcal{B}$ then Duplicator must choose an element $e_i\in \mathcal{A}$.$\newline$
At the end, elements $e_1,\ldots,e_m$ in $\mathcal{A}$ and $f_1,\ldots,f_m$ in $\mathcal{B}$ have been chosen. Duplicator \emph{wins} iff $\overline{a}\;\overline{e} \mapsto \overline{b}\;\overline{f} \in \textrm{Part}(\mathcal{A},\mathcal{B})$ (in case $m=0$ we just require that $\overline{a} \mapsto \overline{b} \in \textrm{\textrm{Part}}(\mathcal{A},\mathcal{B})$). Otherwise Spoiler wins. Equivalently, Spoiler wins if, after some $i\leq m$, $\overline{a}e_1 \ldots e_i \mapsto \overline{b}f_1\ldots f_i$ is not a partial isomorphism. We say that a player, Spoiler or Duplicator, has a \emph{winning strategy} in $G_m(\mathcal{A}, \overline{a}, \mathcal{B}, \overline{b})$, or shortly, that he \emph{wins} $G_m(\mathcal{A}, \overline{a}, \mathcal{B}, \overline{b})$, if it is possible for it to win each play whatever choices are made by the opponent. If $s=0$ (and hence $\overline{a}$ and $\overline{b}$ are empty), we denote the game by $G_m(\mathcal{A},\mathcal{B})$.
\end{definition}

\begin{example}
Spoiler has a winning strategy in $G_3(\mathcal{A},\mathcal{B})$, where, $\newline\mathcal{A}:=(\{0,1\},<)$ and $\mathcal{B}:=(\{0,1,2\},<)$ and in both cases $<$ denotes the natural ordering. Spoiler can begin by choosing $0$ in $\mathcal{B}$ and Duplicator has to reply with $0$ in $\mathcal{A}$ because if it replies with $1$ then, no matter what element (other than $0$) Spoiler chooses from $\mathcal{B}$ in the second move, Duplicator will not find an element in $\mathcal{A}$ greater than $1$ to reply with. In its second move Spoiler chooses $2$ from $\mathcal{B}$ and Duplicator has to reply with the only element in $\mathcal{A}$ greater than $0$, which is $1$. In its third move Spoiler chooses $1$ from $\mathcal{B}$ (which is between $0$ and $2$) but Duplicator cannot find an element in $\mathcal{A}$ between $0$ and $1$ (which correspond respectively to $0$ and $2$ in $\mathcal{B}$).
\end{example}

\begin{definition} (Games : Algebraic Version - Fra\"{\i}ss\'{e}) $\newline$
Structures $\mathcal{A}$ and $\mathcal{B}$ are said to be $m$-\emph{isomorphic}, written $\mathcal{A} \cong_m \mathcal{B}$, if there is a sequence $(I_j)_{j\leq m}$ with the following properties:
\begin{enumerate}
\item[(a)] Every $I_j$ is a nonempty set of partial isomorphisms from $\mathcal{A}$ to $\mathcal{B}$.
\item[(b)] (\emph{Forth property}) For every $j < m$, $p\in I_{j+1}$, and $a\in A$ there is $q\in I_j$ such that $q\supseteq p$ and $a\in do(q)$.
\item[(c)] (\emph{Back property}) For every $j < m$, $p\in I_{j+1}$, and $b\in B$ there is $q\in I_j$ such that $q\supseteq p$ and $b\in rg(q)$.
\end{enumerate}
If $(I_j)_{j\leq m}$ has properties (a), (b), and (c), we write $(I_j)_{j\leq m} : \mathcal{A} \cong_m \mathcal{B}$ and say that $\mathcal{A}$ and $\mathcal{B}$ are $m$-\emph{isomorphic via} $(I_j)_{j\leq m}$.
\end{definition}

\begin{theorem} \cite{Ehren}, \cite{Fraisse} $\newline$
For structures $\mathcal{A}$ and $\mathcal{B}$, $\overline{a}\in A^s$, $\overline{b}\in B^s$, and $m\geq 0$ the following are equivalent :
\begin{enumerate}
\item[(i)] Duplicator wins $G_m(\mathcal{A},\overline{a},\mathcal{B},\overline{b})$.
\item[(ii)] There is $(I_j)_{j\leq m}$ with $\overline{a} \mapsto \overline{b} \in I_m$ such that $(I_j)_{j\leq m} : \mathcal{A} \cong_m \mathcal{B}$.
\item[(iii)] $\overline{a}$ satisfies in $\mathcal{A}$ the same formulas of quantifier rank $\leq m$ as $\overline{b}\in \mathcal{B}$.
\end{enumerate}
\end{theorem}

\begin{corollary}$\newline$
For structures $\mathcal{A}$, $\mathcal{B}$ and $m\geq 0$ the following are equivalent :
\begin{enumerate}
\item[(i)] Duplicator wins $G_m(\mathcal{A},\mathcal{B})$.
\item[(ii)] $\mathcal{A} \cong_m \mathcal{B}$.
\item[(iii)] $\mathcal{A} \equiv_m \mathcal{B}$.
\end{enumerate}
\end{corollary}

The equivalence of (ii) and (iii) is known as \emph{Fra\"{\i}ss\'{e}'s Theorem}, and the equivalence of (i) and (ii) shows that Ehrenfeucht's game and Fra\"{\i}ss\'{e}'s game are different formulations of the same thing, therefore, one often speaks of the \emph{Ehrenfeucht-Fra\"{i}\emph{ss}\'{e} game}.
$\newline$
The following theorem with the previous corollary illustrate how games can be used in proving non-expressibility :
\begin{theorem}\cite{Flum}\label{1.2.10}$\newline$
For a class $K$ of structures the following are equivalent :
\begin{enumerate}
\item[(i)] $K$ is not definable in first-order logic.
\item[(ii)] For each $m$ there are structures $\mathcal{A}$ and $\mathcal{B}$ such that:
              $$\mathcal{A}\in K, \mathcal{B}\notin K \;\;\;\;\; \text{and} \;\;\;\;\; \mathcal{A}\equiv_m \mathcal{B}.$$
\end{enumerate}
\end{theorem}

\begin{example}\cite{Flum} \label{1.2.13}
Let $\tau = \{<,min,max\}$. Suppose that $\mathcal{A}$ and $\mathcal{B}$ are ordered $\tau$-structures, $||\mathcal{A}||>2^m$ and $||\mathcal{B}|| > 2^m$. Then $\mathcal{A} \cong_m \mathcal{B}$. Hence, the class of orderings of even cardinality is not definable in first-order logic.
For a \emph{proof}, given any ordering $\mathcal{C}$, we define its distance function $d$ by
$$d(a,a^{\prime}) := ||\{b\in C \; | \; (a<b\leq a^{\prime}) \; \text{or} \; (a^{\prime} < b \leq a)\}||.$$

And, for $j\geq 0$, we introduce the "truncated" $j$-distance function $d_j$ on $C\times C$ by

\[
 d_j(a,a^{\prime}): =
  \begin{cases}
   d(a,a^{\prime}) & \text { if }  d(a,a^{\prime}) < 2^j \\
    \infty   &  \text {otherwise.}
  \end{cases}
\]

Now suppose that $\mathcal{A}$ and $\mathcal{B}$ are orderings with $||\mathcal{A}||,||\mathcal{B}|| > 2^m$. For $j \leq m$ set
$$I_j := \{p \in \textrm{Part}(\mathcal{A},\mathcal{B}) \; | \; d_j(a,a^{\prime}) = d_j(p(a),p(a^{\prime})) \; \text{for} \; a,a^{\prime} \in do(p)\}.$$

Then $(I_j)_{j\leq m} : \mathcal{A} \cong_m \mathcal{B}$ : By assumption on the cardinalities of $\mathcal{A}$ and $\mathcal{B}$ we have $\{(min^{\mathcal{A}},min^{\mathcal{B}}),(max^{\mathcal{A}},max^{\mathcal{B}})\} \in I_j$ for every $j \leq m$. To give a proof of the forth property of $(I_j)_{j\leq m}$ (the back property can be proven analogously), suppose $j<m$, $p\in I_{j+1}$, and $a\in A$. We distinguish two cases, depending on whether or not the following condition
$$(*) \;\;\; \text{there is an} \; a^{\prime}\in do(p) \; \text{such that}\; d_j(a,a^{\prime}) < 2^j$$
is satisfied. $\newline$
If $(*)$ holds then there is exactly one $b\in B$ for which $p\cup\{(a,b)\}$ is a partial isomorphism preserving $d_j$-distances. Now assume that $(*)$ does not hold and let $do(p)=\{a_1,\ldots,a_r\}$ with $a_1<\ldots <a_r$. We restrict ourselves to the case $a_i<a<a_{i+1}$ for some $i$. Then, $\newline$ $d_j(a_i,a)=\infty$ and $d_j(a,a_{i+1})=\infty$; hence $d_{j+1}(a_i,a_{i+1})=\infty$ and therefore, $d_{j+1}(p(a_i),p(a_{i+1}))=\infty$. Thus there is a $b$ such that $p(a_i)<b<p(a_{i+1})$, $d_j(p(a_i),b)=\infty$, and $d_j(b,p(a_{i+1}))=\infty$. One easily verifies that $q := p \cup \{(a,b)\}$ is a partial isomorphism in $I_j$.
\end{example}

\begin{example} \cite{Flum}
For $l\geq 1$, let $\mathcal{G}_l$ be the graph given by a cycle of length $l+1$. To be precise, set$\newline$
$G_l := \{0,\ldots,l\}$, $E^{G_l} := \{(i,i+1) \; | \; i< l\} \cup \{(i+1,i)\; | \; i < l\} \cup \{(0,l),(l,0)\}$.
Thus, for $l,k\geq 1$, the disjoint union $\mathcal{G}_l \dot{\cup} \mathcal{G}_k$ consists of a cycle of length $l+1$ and of a cycle of length $k+1$. We show :
$$\text{If} \;\; l\geq 2^m \;\; \text{then} \;\; \mathcal{G}_l \cong_m \mathcal{G}_l \dot{\cup} \mathcal{G}_l.$$
In fact, for $j\in \mathbb{N}$, we define the distance function $d_j$ on a graph $\mathcal{G}$ by

\[
 d_j(a,b) : =
  \begin{cases}
   d(a,b) & \text { if }  d(a,b) < 2^{j+1} \\
    \infty   &  \text {otherwise.}
  \end{cases}
\]

(where $d$ denotes the distance function on $\mathcal{G}$ as introduced in $\newline$ Definition \ref{1.1.6}). To show that $\mathcal{G}_l$ and $\mathcal{G}_l \dot{\cup} \mathcal{G}_l$ are $m$-isomorphic, one verifies $(I_j)_{j\leq m} :\mathcal{G}_l \cong_m \mathcal{G}_l \dot{\cup} \mathcal{G}_l$ where $I_j$ is the set of $p\in \textrm{Part}(\mathcal{G}_l,\mathcal{G}_l \dot{\cup} \mathcal{G}_l)$ with
$||do(p)||\leq m-j \;\; \text{and} \;\; d_j(a,b)=d_j(p(a),p(b)) \; \text{for} \; a,b\in do(p)$. (The proof is similar to that of example \ref{1.2.13}). $\newline$
Then by Theorem \ref{1.2.10}, the class $\textrm{CONN}$ of connected graphs is not definable in first-order logic, since for each $m$ we have
$$\mathcal{G}_{2^m}\in \textrm{CONN},\; \mathcal{G}_{2^m} \dot{\cup} \mathcal{G}_{2^m} \notin \textrm{CONN}, \; \mathcal{G}_{2^m} \equiv_m \mathcal{G}_{2^m} \dot{\cup} \mathcal{G}_{2^m}.$$
Consequently, the path relation, the transitive closure of the edge relation on the class $\textrm{GRAPH}$, is not first-order definable.
In fact, suppose $\psi(x,y)$ is a first-order formula defining the path relation on $\textrm{GRAPH}$. Then $\textrm{CONN}$ would be the class of models of  $\forall x \forall y (\neg x=y \rightarrow \psi(x,y))$ (and the graph axioms).
\end{example}

\begin{definition} (Second-Order Logic) $\newline$
\emph{Second-order logic}, $\textrm{SO}$, is an extension of first-order logic which allows to quantify over relations. In addition to the symbols of first-order logic, its alphabet contains, for each $n\geq 1$, countably many $n$-ary \emph{relation} (or \emph{predicate}) \emph{variables} $V_1^n,V_2^n,\ldots$. To denote relation variables we use letters $X,Y,\ldots$. The set of second-order formulas of a vocabulary $\tau$ is the set generated by the rules for first-order formulas extended by :
\begin{enumerate}
\item[-] If $X$ is $n$-ary and $t_1,\ldots,t_n$ are terms then $Xt_1\ldots t_n$ is a formula.
\item[-] If $\varphi$ is a formula and $X$ is a relation variable then $\exists X \varphi$ is a formula.
\end{enumerate}

\end{definition}
The free occurrence of a variable or of a relation variable in a second-order formula is defined in the obvious way and the notion of satisfaction is extended canonically. Then, given $\varphi = \varphi(x_1,\ldots,x_n,Y_1,\ldots,Y_k)$ with free (individual and relation) variables among $x_1,\ldots,x_n,Y_1,\ldots,Y_k$, a $\tau$-structure $\mathcal{A}$, elements $a_1,\ldots,a_n\in A$, and relations $R_1,\ldots,R_k$ over $A$ of arities corresponding to $Y_1,\ldots,Y_k$, respectively, $$\mathcal{A} \vDash \varphi[a_1,\ldots,a_n,R_1,\ldots,R_k]$$
means that $a_1,\ldots,a_n$ together with $R_1,\ldots,R_k$ satisfy $\varphi$ in $\mathcal{A}$.

\section{Fixed-Point Logics}

Here in this section we introduce some of the well-known fixed-point extensions of first-order logic. Let us first pave the road to the definitions:
$\newline\newline$
Fix a vocabulary $\tau$. Consider a second-order formula
 $\varphi(R,\overline{Y},x_1,\ldots,x_k,\overline{u})$, where $R$ is a relation variable of arity $k$, and let $\overline{T}$ be an interpretation of $\overline{Y}$ and $\overline{b}$ an interpretation of $\overline{u}$,  $\newline\newline$
$\bullet$ For any $\tau$-structure $\mathcal{A}$, the function
$$F^{\varphi,\mathcal{A}} : \mathcal{P}(A^k) \rightarrow \mathcal{P}(A^k)$$
$$F^{\varphi,\mathcal{A}}(S) :=\{\overline{a}\in A^k \; | \; \mathcal{A}\vDash \varphi[S,\overline{T},a_1,\ldots,a_k,\overline{b}]\}$$
gives rise to a sequence of $k$-ary relations :
$$\emptyset, F^{\varphi,\mathcal{A}} (\emptyset), (F^{\varphi,\mathcal{A}})^2(\emptyset), (F^{\varphi,\mathcal{A}})^3(\emptyset),\ldots$$
Denote its members by $F_0^{\varphi,\mathcal{A}},F_1^{\varphi,\mathcal{A}},F_2^{\varphi,\mathcal{A}},\ldots$, i.e., $F_0^{\varphi,\mathcal{A}} = \emptyset$ and $F_{n+1}^{\varphi,\mathcal{A}}=F^{\varphi,\mathcal{A}}(F_n^{\varphi,\mathcal{A}})$. $F_n^{\varphi,\mathcal{A}}$ is called the $n$-th stage of $F^{\varphi,\mathcal{A}}$ (or of $\varphi$ in $\mathcal{A}$ under the given interpretations of $\overline{Y}$ and $\overline{u}$).
Suppose that there is an $n_0\in \mathbb{N}$ such that $F_{n_0+1}^{\varphi,\mathcal{A}}=F_{n_0}^{\varphi,\mathcal{A}}$, that is, $F^{\varphi,\mathcal{A}}(F_{n_0}^{\varphi\mathcal{A}})=F_{n_0}^{\varphi,\mathcal{A}}$. Then, $F_m^{\varphi,\mathcal{A}}=F_{n_0}^{\varphi,\mathcal{A}}$ for all $m\geq n_0$. We denote $F_{n_0}^{\varphi,\mathcal{A}}$ by $F_{\infty}^{\varphi,\mathcal{A}}$ and say that \emph{the fixed-point} $F_{\infty}^{\varphi,\mathcal{A}}$ of $F^{\varphi,\mathcal{A}}$ \emph{exists} (It may be also called the fixed-point of $\varphi$ in $\mathcal{A}$ "under the given interpretations of $\overline{Y}$ and $\overline{u}$". But if the free variables of $\varphi$ are among $R$ and $x_1,\ldots,x_k$ only, i.e., there are no extra variables to be interpreted, we just call it the fixed-point of $\varphi$ in $\mathcal{A}$).$\newline$ In case the fixed-point $F_{\infty}^{\varphi,\mathcal{A}}$ does not exist, we agree to set $F_{\infty}^{\varphi,\mathcal{A}} := \emptyset$. When we talk about the stages and fixed-point of $\varphi$ in general, regardless of the interpretations of the extra free variables and regardless of the structure in which they are interpreted, or when the interpretations and the structures are understood from context, we just say the "stages" and the "the fixed-point", and write  $F_n^{\varphi}$ and $F_{\infty}^{\varphi}$. $\newline$ (Note that the notation $F^{\varphi,\mathcal{A}}$ does not make explicit all relevant data.) $\newline$
$\bullet$ Clearly, for any structure $\mathcal{A}$, the sequence for the formula $R\overline{x}\vee\varphi$ is increasing (in fact more, the function $F^{(R\overline{x}\vee\varphi),\mathcal{A}}$ is inflationary, i.e., has the property that $S \subseteq F^{(R\overline{x}\vee\varphi),\mathcal{A}} (S)$ for any $S \subseteq A^k$), and since the structure $\mathcal{A}$ is finite, the sequence must come to the fixed-point in at most $||\mathcal{A}||^k$ stages. This fixed-point (of $R\overline{x}\vee\varphi$ in $\mathcal{A}$ "under the given interpretations of the extra free variables") is called the inflationary fixed-point of $\varphi$ in $\mathcal{A}$ "under the given interpretations of the extra free variables". (Sometimes we call $R\overline{x}\vee\varphi$ "the inflationary formula obtained from $\varphi$")
$\newline$
$\bullet$ $\varphi$ is said to be $R$-positive if, when expressed using only the connectives $\neg,$ $\vee$, all occurrences of $R$ in $\varphi$ are positive, i.e., are in the scope of an even number of negations. It can be easily proved (by induction on formulas) that if $\varphi$ is $R$-positive then it is monotone in the sense that :
For every structure $\mathcal{A}$, and every $S_1,S_2 \in A^k,$ $S_1\subseteq S_2 \Rightarrow F^{\varphi,\mathcal{A}}(S_1)\subseteq F^{\varphi,\mathcal{A}}(S_2)$. $\newline$
Clearly, if $\varphi$ is monotone then for every structure $\mathcal{A}$,
 $$\emptyset \subseteq F^{\varphi,\mathcal{A}}(\emptyset) \subseteq (F^{\varphi,\mathcal{A}})^2(\emptyset) \subseteq (F^{\varphi,\mathcal{A}})^3(\emptyset) \subseteq \ldots $$
 and hence its fixed-point exists (the sequence comes to it at a stage $\leq ||\mathcal{A}||^k$). This fixed-point is also the \emph{least} fixed-point : Let $S\subseteq A^k$ be such that $F^{\varphi,\mathcal{A}}(S)=S$. $\emptyset \subseteq S$ i.e. $F^{\varphi,\mathcal{A}}_0 \subseteq S$. Assuming that $F^{\varphi,\mathcal{A}}_i \subseteq S$, where $i\geq 0$, we have by the monotonicity of $\varphi$ that $F^{\varphi,\mathcal{A}}(F^{\varphi,\mathcal{A}}_i) \subseteq F^{\varphi,\mathcal{A}}(S)$ i.e. $F^{\varphi,\mathcal{A}}_{i+1} \subseteq S$. Hence, by induction, $F^{\varphi,\mathcal{A}}_i \subseteq S$ for every $i\geq 0$. Thus $F^{\varphi,\mathcal{A}}_{\infty}\subseteq S$. This shows that $F^{\varphi,\mathcal{A}}_{\infty}$ is contained in every other fixed-point of $F^{\varphi,\mathcal{A}}$, hence it is the least fixed-point of $F^{\varphi,\mathcal{A}}$.

 \begin{definition} (Fixed-Point Extensions of First-Order Logic)
 \begin{enumerate}
   \item[(1)]  Partial Fixed-Point Logic $\textrm{FO}(\textrm{PFP})$ :$\newline$
The class of formulas of $\textrm{FO}(\textrm{PFP})$ of vocabulary $\tau$ is given by the calculus $\newline \newline$
$\bullet$ $\frac{}{\varphi}$ where $\varphi$ is an atomic second-order formula over $\tau$. $\newline\newline$
$\bullet$ $\frac{\varphi}{\neg \varphi}$, $\frac{\varphi,\psi}{(\varphi\vee\psi)}$, $\frac{\varphi}{\exists x \varphi}$ $\newline\newline$
$\bullet$ $\frac{\varphi}{[\textrm{PFP}_{R,\overline{x}}\varphi]\overline{t}}$ where the lengths of $\overline{x}$ and $\overline{t}$ are the same and coincide with the arity of $R$.$\newline$
(The expression $\frac{\varphi_1,\ldots,\varphi_l}{\varphi}$ means that if $\varphi_1,\ldots,\varphi_l$ are formulas then $\varphi$ is a formula).
The semantics is defined inductively with respect to this calculus, the meaning of $[\textrm{PFP}_{R,\overline{x}}\varphi]\overline{t}$ being $\overline{t}\in F^{\varphi}_{\infty}$. In particular, $[\textrm{PFP}_{R,\overline{x}}\varphi]\overline{t}$ is false if the fixed-point of $\varphi$ does not exist. $\newline$
Sentences are formulas without free first-order and second-order variables where the free occurrence of variables are defined in the standard way, adding the clause
$$free([\textrm{PFP}_{R,\overline{x}}\varphi]\overline{t}) := free(\overline{t}) \cup (free(\varphi)\backslash \{\overline{x},R\})$$

   \item[(2)]  Inflationary Fixed-Point Logic $\textrm{FO}(\textrm{IFP})$ :$\newline \newline$
The class of formulas of $\textrm{FO}(\textrm{IFP})$ of vocabulary $\tau$ is given by the calculus with the same rules as the calculus for $\textrm{FO}(\textrm{PFP})$ except for the last rule which is replaced by $\newline$
"$\frac{\varphi}{[\textrm{IFP}_{R,\overline{x}}\varphi]\overline{t}}$ where the lengths of $\overline{x}$ and $\overline{t}$ are the same and coincide with the arity of $R$." $\newline$
The semantics for $[\textrm{IFP}_{R,\overline{x}}\varphi]\overline{t}$ being $\overline{t}\in F^{(R\overline{x}\vee\varphi)}_{\infty}$.

   \item[(3)] Least Fixed-Point Logic $\textrm{FO}(\textrm{LFP})$ :$\newline$
The class of formulas of $\textrm{FO}(\textrm{LFP})$ of vocabulary $\tau$ is given by the calculus with the same rules as the calculus for $\textrm{FO}(\textrm{PFP})$ but the last rule is restricted to $R$-positive formulas $\varphi$. In this case, as we have shown above, $\varphi$ is monotone and hence its fixed-point exists and it is also the least fixed-point. $\newline$Thus, in this case, the formula $[\textrm{PFP}_{R,\overline{x}}\varphi]\overline{t}$ expresses that $\overline{t}$ is in the least fixed-point of $\varphi$. We therefore write $[\textrm{LFP}_{R,\overline{x}}\varphi]\overline{t}$ instead of $[\textrm{PFP}_{R,\overline{x}}\varphi]\overline{t}$.
 \end{enumerate}
 \end{definition}
 We sometimes write $[\textrm{PFP}_{R,\overline{x}}\varphi]$, $[\textrm{IFP}_{R,\overline{x}}\varphi]$, or $[\textrm{LFP}_{R,\overline{x}}\varphi]$ to denote the fixed-points.

\begin{remark}$\newline$
 Clearly, $\textrm{FO}(\textrm{LFP}) \leq \textrm{FO}(\textrm{IFP}) \leq \textrm{FO}(\textrm{PFP})$, (i.e. every formula in $\textrm{FO}(\textrm{LFP})$ is equivalent to a formula in $\textrm{FO}(\textrm{IFP})$, and every formula in $\textrm{FO}(\textrm{IFP})$ is equivalent to a formula in $\textrm{FO}(\textrm{PFP})$).
Actually, on every class of finite structures, $$\textrm{FO}(\textrm{LFP}) = \textrm{FO}(\textrm{IFP}) \;\;\;\;\;\;\; \text{(Gurevich and Shelah 1986 \cite{Gurevich})}$$
\end{remark}
\begin{remark}$\newline$
It can be easily proved that the extra free first-order variables in fixed-point formulas, for example as $\overline{u}$ in page 20, can always be avoided at the expense of relations of higher arity. Thus, we can normally assume that no extra free first-order variables are present in the fixed-point formula.
\end{remark}
\begin{remark}$\newline$
These logics are strong enough to express transitive closure. Given a global binary relation $R$ over the set of $k$-tuples (of any structure), its transitive closure in any structure is expressible by $$[\textrm{LFP}_{X,\overline{x},\overline{y}}\varphi(X,\overline{x},\overline{y})]\overline{x}\overline{y}$$

where $\varphi(X,\overline{x},\overline{y}) := R\overline{x}\overline{y} \vee \exists z_1 \ldots \exists z_k (X\overline{x}\overline{z}\wedge R\overline{z}\overline{y})$.
\end{remark}

Let us now present simultaneous fixed-points: $\newline$
Given two formulas $\varphi(X,Y,x_1,\ldots,x_k)$ and $\psi(X,Y,y_1,\ldots,y_l)$, where $X$ is a relation variable of arity $k$ and $Y$ is a relation variable of arity $l$, and a structure $\mathcal{A}$, the functions
  $$F^1 : A^k\times A^l \rightarrow A^k$$
  $$F^1(S_1,S_2) := \{\overline{a} \in A^k \; | \; \mathcal{A} \vDash \varphi[S_1,S_2,a_1,\ldots,a_k]\}$$
  and
   $$F^2 : A^k\times A^l \rightarrow A^l$$
  $$F^2(S_1,S_2) := \{\overline{b} \in A^l \; | \; \mathcal{A} \vDash \psi[S_1,S_2,b_1,\ldots,b_l]\}$$
  give rise to two sequences :
  $$F^i_{(0)} := \emptyset,\; F^i_{(n+1)} := F^i(F^1_{(n)},F^2_{(n)}) \;\;\; \text{for} \;i=1,2.$$

If we have for some $n$ that $(F^1_{(n)},F^2_{(n)}) = (F^1_{(n+1)},F^2_{(n+1)})$, we set $(F^1_{(\infty)},F^2_{(\infty)}) := (F^1_{(n)},F^2_{(n)})$ and say that the \emph{simultaneous fixed-point} $(F^1_{(\infty)},F^2_{(\infty)})$ of $(F^1,F^2)$ \emph{exists} (we may also call it the simultaneous fixed-point of $\varphi$ and $\psi$ in $\mathcal{A}$).$\newline$
Note that for $i=1,2$,
$$F^i(F^1_{(\infty)},F^2_{(\infty)}) = F^i_{(\infty)}.$$

\section{Complexity}
In this section we introduce the definitions and facts we need from complexity theorey. This section is mainly form Immerman's "Descriptive Complexity" \cite{Immerman}. We assume that the reader is familiar with the Turing machine.
 \begin{definition}
 A \emph{query} is any mapping $I: \textrm{STRUC}[\sigma]\rightarrow \textrm{STRUC}[\tau]$ from the finite structures of one vocabulary to the finite structures of another vocabulary, that is polynomially bounded. That is, there is a polynomial $p$ such that for all $\mathcal{A}\in \textrm{STRUC}[\sigma]$, $\|I(\mathcal{A})\| \leq p(\|\mathcal{A}\|)$. A \emph{boolean query} is a map $I_b : \textrm{STRUC}[\sigma] \rightarrow \{0,1\}$, from the finite structures of a vocabulary to $\{0,1\}$. A boolean query may also be thought of as a subset of $\textrm{STRUC}[\sigma]$ - the set of finite structures $\mathcal{A}$ for which $I(\mathcal{A})=1$.
 \end{definition}

 Any sentence $\varphi$, from any logic over any vocabulary $\sigma$, defines a boolean query $I_{\varphi}$ on $\textrm{STRUC}[\sigma]$ where $I_{\varphi} (\mathcal{A}) = 1$ iff $\mathcal{A}\vDash \varphi$.

  \begin{definition} \label{1.4.2}(First-Order Queries) $\newline$
 Let $\sigma$ and $\tau$ be any two vocabularies where $\tau =\{R_1,\ldots,R_r, c_1, \ldots,c_s \}$ and each $R_i$ has arity $a_i$, and let $k$ be a fixed natural number. A \emph{k-ary first-order query} is a map
                 $$I: \textrm{STRUC}[\sigma]\rightarrow \textrm{STRUC}[\tau]$$
defined by an $r+s+1$-tuple of first-order formulas, $\varphi_0,
\varphi_1,\ldots,\varphi_r, \psi_1,\ldots,\psi_s$, from
$\textrm{FO}[\sigma]$. For each structure $\mathcal{A}\in
\textrm{STRUC}[\sigma]$, these formulas describe a structure
$I(\mathcal{A})\in \textrm{STRUC}[\tau]$,
               $$I(\mathcal{A})= (|I(\mathcal{A})|,R_1^{I(\mathcal{A})},\ldots,R_r^{I(\mathcal{A})}, c_1^{I(\mathcal{A})}, \ldots,c_s^{I(\mathcal{A})})$$
The universe of $I(\mathcal{A})$ is a first-order definable subset
of $A^k$,
$$|I(\mathcal{A})|=\{ (b^1,\ldots,b^k) \in A^k \; | \; \mathcal{A}\vDash \varphi_0[b^1,\ldots,b^k] \}$$

Each relation $R_i^{I(\mathcal{A})}$ is a first-order definable
subset of $|I(\mathcal{A})|^{a_i}$,
$$R_i^{I(\mathcal{A})} = \{ ( ( b_1^1,\ldots, b_1^k ), \ldots , ( b_{a_i}^1,\ldots,b_{a_i}^k)) \in |I(\mathcal{A})|^{a_i} \; | \; \mathcal{A}\vDash \varphi_i[b_1^1,\ldots,b_{a_i}^k] \}$$
Each constant symbol $c_j^{I(\mathcal{A})}$ is a first-order
definable element of $|I(\mathcal{A})|$,$\newline$
$c_j^{I(\mathcal{A})} =$ the unique $ (b^1, \ldots , b^k)
\in |I(\mathcal{A})|$ such that $\mathcal{A}\vDash
\psi_j[b^1,\ldots,b^k].$
(Every $\psi_j$ should be such that there is exactly one tuple $(b^1,\ldots,b^k)$ satisfying $\varphi_0 \wedge \psi_j$, otherwise $I$ does not define a first-order query)
 \end{definition}
 A first-order query is either boolean, and thus defined by a first-order sentence, or is a $k$-ary first-order query, for some $k$.
$\newline\newline$
We can now say precisely  what we meant by saying that every structure may be thought of as a graph : \cite{Immerman} $\newline$
Let $\tau_g$ denote the vocabulary of graphs. For any vocabulary $\sigma$ there exist first-order queries $I : \textrm{STRUC}[\sigma] \rightarrow \textrm{STRUC}[\tau_g]$ and $I^{-1} : \textrm{STRUC}[\tau_g] \rightarrow \textrm{STRUC}[\sigma]$ with the following property,
$$\text{for all} \; \mathcal{A}\in \textrm{STRUC}[\sigma],\;\;\;\;\; I^{-1}(I(\mathcal{A})) \cong \mathcal{A} $$

\begin{definition}
Let $I: \textrm{STRUC}[\sigma]\rightarrow \textrm{STRUC}[\tau]$ be a query. Let $M$ be a Turing machine. Suppose that for all $\mathcal{A}\in \textrm{STRUC}[\sigma]$, $M(bin(\mathcal{A})) = bin(I(\mathcal{A}))$. Then we say that $M$ \emph{computes} $I$.
\end{definition}

We assume that the reader is familiar with the following classical complexity classes : $\textrm{L}$ deterministic logspace, $\textrm{P}$ deterministic polynomial time, and $\textrm{PSPACE}$ deterministic polynomial space.

\begin{definition}($\mathcal{Q}(\mathcal{C})$, the queries computable in $\mathcal{C}$) $\newline$
Let $I: \textrm{STRUC}[\sigma] \rightarrow \textrm{STRUC}[\tau]$ be
a query, and $\mathcal{C}$ a complexity class. We say that $I$ is
\emph{computable in} $\mathcal{C}$ (or in $\mathcal{C}$ for short) iff the boolean query $I_b$ is an
element of $\mathcal{C}$, where $ \newline I_b = \{(\mathcal{A}, i, a)\; |$
The $i$-th bit of $bin(I(\mathcal{A}))$ is $``a" \}.$ Let
$\mathcal{Q}(\mathcal{C})$ be the set of all queries computable in
$\mathcal{C}$:
$\mathcal{Q}(\mathcal{C}) = \mathcal{C} \cup \{I\; | \; I_b \in \mathcal{C}\}.$
\end{definition}

\begin{definition}(Many-One Reduction)$\newline$
Let $\mathcal{C}$ be a complexity class, and let $K\subseteq
\textrm{STRUC}[\sigma]$ and $H \subseteq \textrm{STRUC}[\tau]$ be
boolean queries. Suppose that the query $I:\textrm{STRUC}[\sigma]
\rightarrow \textrm{STRUC}[\tau]$ is an element of
$\mathcal{Q}(\mathcal{C})$ with the property that for all
$\mathcal{A}\in \textrm{STRUC}[\sigma]$,
$$\mathcal{A}\in K \Leftrightarrow I(\mathcal{A})\in H$$

Then $I$ is a $\mathcal{C}$-\emph{many-one reduction} from $K$ to $H$. We
say that $K$ is $\mathcal{C}$-\emph{many-one reducible} to $H$, in symbols,
$K \leq_{\mathcal{C}} H$. For example, when $I$ is a first-order
query, the reduction is called a first-order reduction ($\leq_{\text{fo}}$); when $I \in \mathcal{Q}(\textrm{L})$, the reduction is called a logspace reduction ($\leq_{\log}$); and when $I \in \mathcal{Q} (\textrm{P})$, the reduction is called a polynomial-time reduction ($\leq_{\text{p}}$).
\end{definition}

\begin{definition}
Let $K$ be a boolean query, let $\mathcal{C}$ be a complexity class, and let $\leq_r$ be a reducibility relation.
We say that $K$ is $\mathcal{C}$-\emph{complete} under $\leq_r$ iff :
\begin{enumerate}
\item[(1)] $K\in \mathcal{C}$, and,
\item[(2)] for all $H\in \mathcal{C}$, $H \leq_{r} K$.
\end{enumerate}
\end{definition}

\begin{definition} (Closure under First-Order Reductions) $\newline$
A set of boolean queries $\mathcal{S}$ is \emph{closed under first order reductions} iff whenever there are boolean queries $K$ and $H$ such that $H \in \mathcal{S}$ and $K \leq_{\text{fo}} H$, we have that $K \in \mathcal{S}$. We say that a logic $\mathcal{L}$ is closed under first-order reductions iff the set of boolean queries definable in $\mathcal{L}$ is so closed.

\end{definition}

\begin{definition} (Alternating Turing Machine)
An \emph{alternating Turing machine} is a Turing machine whose states are divided into two groups : the existential states and the universal states. The notion of when such a machine accepts an input is defined by induction : The alternating machine in a given configuration $C$ \emph{accepts} iff
\begin{enumerate}
\item[1.] $C$ is in a final accepting state, or
\item[2.] $C$ is in an existential state and there exists a next configuration $C^{\prime}$ that accepts, or
\item[3.] $C$ is in a universal state, there is at least one next configuration,and all next configurations accept.
\end{enumerate}
\end{definition}

Note that this is a generalization of the notion of acceptance for a nondeterministic Turing machine, which is an alternating Turing machine all of whose states are existential.

\begin{definition} (Boolean Circuits) $\newline$
A \textit{boolean circuit} is a directed acyclic graph
$$C=(V,E,G_{\wedge},G_{\vee},G_{\neg},I,r)\in \textrm{STRUC}[\tau_c]; \;\; \tau_c = \{ E,G_{\wedge},G_{\vee},G_{\neg},I,r \}$$
where $\newline$
The vertices $v$ with no edges entering them, are called leaves, and the input relation $I(v)$ represents the fact that the leaf $v$ is on (i.e. recieved $1$ as input). The vertices that are not leaves are called internal vertices. An internal vertex $w$ is said to be an "and"-gate if $G_{\wedge}(w)$ holds, an "or"-gate if $G_{\vee}(w)$ holds, and a "not"-gate if $G_{\neg}(w)$ holds. Any internal vertex should be in exactly one of $G_{\wedge},G_{\vee},$ or $G_{\neg}$. A "not"-gate should have indegree $1$. $r$ is a vertex with no outgoing edges and is called the root.
\end{definition}

\begin{definition}(Computation by Boolean Circuits)$\newline$
Let $S\subseteq \textrm{STRUC}[\tau_s]$ be a boolean query on binary strings (where $\tau_s$ is the vocabulary of binary strings). In circuit complexity, $S$ would be computed by an infinite sequence of circuits
$$\mathcal{C}=\{C_i\; | \; i=1,2,\ldots\},$$
where $C_n$ is a circuit with $n$ input bits (i.e., vertices with no ingoing edges, or leaves) and a single output bit r (i.e. there is only one vertex with no outgoing edges which is the root). For $w\in \{0,1\}^n$, let $C_n(w)$ be the value at $C_n$'s output gate, when the bits of $w$ are placed in its $n$ input gates (this defines the values at the input gates), and where the values at the other gates are defined inductively as follows (until we come to the output gate at the end) : If the gate is an "and"-gate then its value is $1$ iff the values of all gates from which there are edges entering it are $1$, if the gate is an "or"-gate then its value is $1$ iff at least one of the gates from which there are edges entering it has value $1$, and if the gate is a "not"-gate then its value is $1$ iff the value of the unique gate from which there is an edge entering it is $0$.
$\newline$ We say that $\mathcal{C}$ \textit{computes} $S$ iff for all $n$ and for all $w\in\{0,1\}^n$,
                            $$w\in S \;\;\;\;\; \Leftrightarrow \;\;\;\;\; C_n(w)=1.$$
\end{definition}

\begin{definition}
The \textit{size} of a boolean circuit is the number of vertices in it, and its \textit{depth} is the length of a longest path from root to leaf.
\end{definition}

\begin{definition} (Uniformity) $\newline$
Let $\mathcal{C}= \{C_i\; | \; i=1,2,\ldots\}$ be a sequence of circuits, and let $I : \textrm{STRUC}[\tau_s] \rightarrow\textrm{STRUC}[\tau_c]$ be a query such that for all $n\geq 1$, $I(0^n) = C_n$. That is, on input a string of $n$ zeros the query produces circuit $n$. If $I$ is a first-order query, then $\mathcal{C}$ is said to be a \textit{first-order uniform} sequence of circuits. Similarly, if $I\in \textrm{L}$, then $\mathcal{C}$ is \textit{logspace} \textit{uniform}. If $I\in \textrm{P}$, then $\mathcal{C}$ is \textit{polynomial-time uniform}, and so on.
\end{definition}
From now on, when we mention uniformity without specifying the kind, we mean first-order uniformity.

\newpage
\thispagestyle{empty}
\mbox{}
.
\chapter{Depth}

This chapter is devoted to the notion of depth of inductive definitions over finite structures. Roughly speaking, depth, is the number of iterations (of the function $F^{\varphi}$ on $\emptyset$) needed to come to the fixed-point (of $\varphi$).

\section{The Notion of Depth}

\begin{definition}$\newline$
Let $\varphi(R,x_1,\ldots,x_k)$ be a formula, where $R$ is a relation variable of arity $k$, let $\mathcal{A}$ be a structure of size $n$.
\begin{enumerate}
\item[(1)] Define the \emph{depth} of $\varphi$ \emph{in} $\mathcal{A}$, in symbols $|\varphi^{\mathcal{A}}|$, to be the minimum $r$ such that
$$F^{\varphi,\mathcal{A}}_{r+1} = F^{\varphi,\mathcal{A}}_r$$
in case the fixed-point of $F^{\varphi,\mathcal{A}}$ exists. If the fixed-point does not exist we agree to set $|\varphi^{\mathcal{A}}|:=\infty$. $\newline$
It can be easily shown that $|\varphi^{\mathcal{A}}| < 2^{n^k}$, in case the fixed-point of $F^{\varphi,\mathcal{A}}$ exists.
\item[(2)] Define the depth of $\varphi$ as a function of $n$ equal to the maximum depth of $\varphi$ in $\mathcal{A}$ (over all structures $\mathcal{A}$ of size $n$) :
             $$|\varphi|(n) := \underset{||\mathcal{A}||=n}{max} \{|\varphi^{\mathcal{A}}|\}$$
\item[(3)] Define the \emph{inflationary-depth} of $\varphi$ \emph{in} $\mathcal{A}$, in symbols $|\varphi^{\mathcal{A}}|_{\textrm{IFP}}$, as the depth of $(R\overline{x}\vee \varphi)$ in $\mathcal{A}$, and the \emph{inflationary-depth} of $\varphi$ (as a function of $n$), in symbols $|\varphi|_{\textrm{IFP}}$, as the depth of $(R\overline{x}\vee \varphi)$ (as a function of $n$).
\end{enumerate}
\end{definition}

We may also think intuitively of the depth for the simultaneous fixed-point and for the nested fixed-point as the total number of iterations needed to come to the fixed-point.

Least fixed-point logic can be fragmentized into smaller logics by imposing restrictions on depth. Before we mention the definition from Immerman's "Descriptive Complexity", we must mention following proposition :
\begin{proposition} \label{2.1.2} \cite{Flum},\cite{Immerman} $\newline$
On ordered structures, every $\textrm{FO}(\textrm{IFP})$-sentence is equivalent to an $\textrm{FO}(\textrm{IFP})$-sentence in which $\textrm{IFP}$ occurs at most once. The same applies to $\textrm{FO}(\textrm{PFP})$ and $\textrm{PFP}$, and also to $\textrm{FO}(\textrm{LFP})$ and $\textrm{LFP}$.
\end{proposition}

Hence when we talk about a fixed-point formula we may allow ourselves to imagine it as one in which the fixed-point operator occurs only once in front of a first-order formula, (or more precisely, in front of a second-order formula without any occurrences of a fixed-point operator or a second-order quantifier, but we sometimes call such formulas first-order for easiness).

\begin{definition}
Let $\textrm{IND}[f(n)]$  be the sub-logic of $\textrm{FO}(\textrm{LFP})$ in which only fixed-points of first-order formulas $\varphi$ for which $|\varphi|$ is $O(f(n))$ are included.
\end{definition}

From a previous remark, the fixed-point of a positive formula in a finite structure $\mathcal{A}$ comes at a stage $\leq||\mathcal{A}||^k$, where $k$ is the arity of the relation variable through which the fixed-point is computed, i.e., the depth of positive formulas is always polynomially bounded. Hence we may assume that $f(n)$ is polynomially bounded. $\newline$
It is clear now that, $\textrm{FO}(\textrm{LFP})=\underset{k\geq 1}{\bigcup}\textrm{IND}[n^k]$

\begin{example} (A Well-Known Example)
In graphs, the transitive closure of the edge relation is the least fixed-point of the formula
$$\varphi(R,x,y):= Exy\vee \exists z(E(x,z)\wedge R(z,y))$$ i.e. for any $u\neq v$ there is a path from $u$ to $v$ iff $$[\textrm{LFP}_{R,x,y}\varphi(R,x,y)]uv \;\;\;\;\; \text{holds}.$$
In any graph $\mathcal{G}$, for any $k\geq1$,$\newline$
 $(F^{\varphi,\mathcal{G}})^k(\emptyset)= \{(x,y)\; | \; x\neq y$ and there is a path from $x$ to $y$ of length $\leq k \}$, and since the distance (the shortest length of a path) from a vertex to another vertex connected to it in $\mathcal{G}$ is at most $n-1$ if $\|\mathcal{G}\| = n$, the fixed-point is obtained at most at $k=n-1$  i.e. after $n-1$ iterations of the function $F^{\varphi,\mathcal{G}}$ on $\emptyset$. Since for every $n$ there is actually a graph on $n$ vertices in which there is a path of length $n-1$, then $\varphi(R,x,y)$ is of depth $n-1$, and therefore connectivity is expressible in $\textrm{IND}[n]$. Actually it is even expressible in $\textrm{IND}[\log n]$ as the path relation can be expressed as the fixed-point of the formula
                     $$\psi(R,x,y):= E(x,y) \vee \exists z (R(x,z) \wedge R(z,y))$$
                     which is of depth $\left\lceil \log n \right\rceil +1 $.
\end{example}

We now present another view of depth, also from Immerman's "Descriptive Complexity", depending on some form of writing positive formulas that was originally introduced by Moschovakis \cite{Moschovakis} (Moschovakis and Immerman both studied inductive definitions but the work of Moschovakis is different from the work of Immerman in that Moschovakis focused mainly on infinite structures while Immerman studies finite ones). The following theorem shows this form :

\begin{theorem}
Let $\varphi(R,x_1,\ldots,x_k)$ be an $R$-positive first-order formula, where $R$ is a relation variable of arity $k$. Then there are quantifiers $Q_1,\ldots,Q_s$, and quantifier-free first-order formulas $M_1,\ldots,M_s,M_{s+1}$ in which $R$ does not occur such that $$\varphi (R,x_1 \ldots x_k) \equiv \newline\newline (Q_1z_1, M_1)\ldots(Q_sz_s,M_s)(\exists x_1 \ldots \exists x_k,M_{s+1}) R x_1,\ldots,x_k$$
 where for formulas $\psi$ and $M$, $(\exists x,M) \psi$ means $\exists x (M \wedge \psi)$, and $(\forall x, M) \psi$ means $\forall x (M\rightarrow \psi)$.
\end{theorem}

\begin{proof}
This theorem is proved for structures with two distinguished elements $0$ and $1$.$\newline$
 The proof is by induction on the complexity of $\varphi$. We assume that all negations have been pushed all the way inside. $\newline$ There are two base cases :
\begin{enumerate}
\item[(1)] If $\varphi \equiv R v_1 \ldots v_k$, then,
                $$\varphi \equiv (\exists z_1 \ldots \exists z_k,M_1)(\exists x_1 \ldots \exists x_k,M_2)R x_1 \ldots x_k$$
                where
                $$M_1 \equiv z_1=v_1\wedge \ldots \wedge z_k=v_k,$$
                and
                $$M_2 \equiv x_1=z_1 \wedge \ldots \wedge x_k=z_k$$
\item[(2)]If $\varphi$ is quantifier-free and $R$ does not occur in $\varphi$, then, $$\varphi \equiv (\forall z, \neg \varphi)(\exists x_1 \ldots \exists x_k, x_1 \neq x_1)R x_1 \ldots x_k$$
\end{enumerate}
In the inductive cases $\varphi \equiv \exists v \psi$ and $\varphi \equiv \forall v \psi$, we simply put $(\exists v, v=v)$ and $(\forall v, v=v)$ in front of the quantifier block for $\psi$. Now there remains the cases for $\wedge$ and $\vee$.
$\newline$ Suppose that $\varphi \equiv \alpha \wedge \beta$ and

$$\alpha \equiv (Q_1y_1,N_1)\ldots (Q_ty_t,N_t)(\exists x_1 \ldots \exists x_k,N_{t+1}) R x_1 \ldots x_k$$ and
$$\beta \equiv (P_1z_1,M_1) \ldots (P_sz_s,M_s)(\exists x_1 \ldots \exists x_k,M_{s+1})R x_1 \ldots x_k$$

where the $Q$'s and $P$'s are quantifiers. We may assume that the $y$'s and $z$'s are disjoint and, are both disjoint from the free variables of $\varphi$.
Let
$$QB_1 := (Q_1y_1,N_1^{\prime}) \ldots (Q_ty_t,N_t^{\prime})$$ and
$$QB_2 := (P_1z_1,M_1^{\prime}) \ldots (P_sz_s,M_s^{\prime})$$
where $N_i^{\prime}:= N_i \vee (v=0)$; and $M_i^{\prime} := M_i \vee (v=1)$.$\newline$
Let $u_1, \ldots , u_k$ be a new set of variables, and for any formula $\psi(x_1, \ldots , x_k)$, let $\psi(\overline{u}/\overline{x})$ denote the formula $\psi$ with variables $u_1,\ldots,u_k$ substituted for the free occurrences of $x_1,\ldots,x_k$, and define the quantifier-free formulas,
                  $$\theta := (v=1 \wedge N_{t+1}(\overline {u}/\overline{x})) \vee (v=0\wedge M_{s+1}(\overline{u} / \overline{x}))$$ and $$\rho := (u_1 = x_1 \wedge \ldots\wedge u_k=x_k)$$
                  We show that $$\varphi \equiv (\forall v, (v=0 \vee v=1)) QB_1 QB_2(\exists \overline{u},\theta) (\exists \overline{x}, \rho) R x_1 \ldots x_k \;\;\;\; (*)$$

Put $\psi := QB_1 QB_2(\exists \overline{u},\theta) (\exists \overline{x}, \rho) R x_1 \ldots x_k$.The right hand side of $(*)$ holds iff in both cases, $v=0$ and $v=1$, the formula $\psi$ holds. When $v=1$, the formula $\theta$ is equivalent to $N_{t+1}(\overline{u}/\overline{x})$ whose free variables are from $\{y_1,\ldots, y_t,u_1,\ldots,u_k\}\cup free(\alpha)$. Also all the formulas $M_i^{\prime}$ are true when $v=1$, thus $QB_2$ is the same as $P_1z_1 \ldots  P_sz_s$. And since $z_1,\ldots,z_s$ are disjoint from $y_1,\ldots,y_t,u_1,\ldots,u_k$, and $free(\alpha)$, then $$QB_2 (\exists \overline{u},\theta) (\exists \overline{x},\rho)R x_1 \ldots x_k\; \text{is equivalent to}$$ $$(\exists\overline{u},N_{t+1}(\overline{u}/\overline{x})) (\exists \overline{x},\rho)R x_1 \ldots x_k,\; \text{which is equivalent to}$$ $$(\exists\overline{x}, N_{t+1}) R x_1 \ldots x_k.$$ $N_i^{\prime}$ is equivalent to $N_i$ for every $1\leq i \leq t$ when $v=1$, so $QB_1$ is the same as $(Q_1y_1,N_1) \ldots (Q_ty_t,N_t)$.Thus $\psi$ is equivalent to $$(Q_1y_1,N_1) \ldots (Q_ty_t,N_t)(\exists \overline{x}, N_{t+1}) R x_1 \ldots x_k$$ i.e. equivalent to $\alpha$. It can be similarly seen that, in case $v=0$, $\psi$ is equivalent to $\beta$. Thus the right hand side holds iff both $\alpha$ and $\beta$ hold i.e. iff $\varphi$ holds. It can now be similarly seen that if $\varphi \equiv \alpha \vee \beta $ then $\varphi \equiv (\exists v, (v=0 \vee v=1)) \psi$.
\end{proof}

Let us write $QB$ to denote the quantifier block $$(Q_1 z_1,M_1)\ldots(Q_s z_s,M_s)(\exists x_1 \ldots \exists x_k,M_{s+1})$$ Thus in particular, for any structure $\mathcal{A}$, any $r \in \mathbb{N}$, any $\overline{a} \in A^k$,

$$\overline{a} \in (F^{\varphi,\mathcal{A}})^r (\emptyset) \; \; \text{iff} \; \; \mathcal{A} \vDash ([QB]^r \textrm{F})[\overline{a}]$$

Here $[QB]^r$ means $QB$ literally repeated $r$ times. It follows immediately that if $t(n) = |\varphi| (n)$ and $\mathcal{A}$ is any structure of size $n$ then $$\mathcal{A} \vDash ([\textrm{LFP}_{R,x_1,\ldots,x_k} \varphi]\overline{y} \leftrightarrow ([QB]^{t(n)} \textrm{F} )(\overline {y}))[\overline{a}]
\;\;\;\;\;\text{for all} \;\overline{a}\in A^k.$$

This directly inspires the following (apparently) more general definition :

\begin{definition} \cite{Immerman}
$\textrm{FO}[t(n)]$ is defined to be the class of properties definable by quantifier blocks iterated $t(n)$ times. (This is the same as being iterated $O(t(n))$ times since a quantifier block may be any constant size).
More precisely, a class $S$ of finite structures is a member of $\textrm{FO}[t(n)]$ iff there exist quantifier-free first-order formulas $M_i$, $0\leq i \leq s$, a quantifier block $QB = (Q_1x_1,M_1)...(Q_sx_s,M_s)$, a tuple $\overline{c}$ of constants (if necessary), and a function $f(n)=O(t(n))$ such that for any structure $\mathcal{A}$,
$$\mathcal{A}\in S \;\; \Leftrightarrow \;\; \mathcal{A}\vDash ([QB]^{f(\|\mathcal{A}\|)}M_0)(\overline{c}/\overline{x})$$
The reason for the substitution of constants is that the quantifier block $QB$ may contain some free variables that must be substituted for to build a sentence.
\end{definition}
Clearly, for all polynomially bounded $t(n)$, $$\textrm{IND}[t(n)]\subseteq \textrm{FO}[t(n)]$$

\section{The Basic Theorems}
In this section we will see how depth can be regarded as a complexity measure through theorems relating it to well-known complexity classes. For example, it turns out that inductive depth eaquals parallel-time. The theorems in this section are from Immerman's "Descriptive Complexity" \cite{Immerman} and therefore we must impose here the same proviso he imposed at the beginning of his book :
\begin{proviso}\label{2.2}
We assume that the numeric relations and constants : $<$, $\textrm{PLUS}$, $\textrm{TIMES}$, $\textrm{BIT}$, $\textrm{SUC}$, $0$, $1$, $max$ are all present in all vocabularies and are interpreted as follows :
$0$, $1$, $max$ are interpreted as the minimum, second, and maximum elements under $<$ (which is the usual total ordering inherited from $\mathbb{N}$); this means we assume that all structures contain at least two elements $0$ and $1$. $\textrm{PLUS}(i,j,k)$ means that $i+j = k$, $\textrm{TIMES}(i,j,k)$ means that $i \times j = k$, $\textrm{BIT}(i,j)$ means that bit $j$ in the binary representation of $i$ is $1$, and $\textrm{SUC}$ is the usual successor.
\end{proviso}

We begin with the theorem showing that on ordered structures least fixed-point logic corresponds to the famous complexity class $\textrm{P}$ :
\begin{theorem}\label{2.2.1}(Immerman and Vardi 1982 \cite{lfpptime},\cite{Vardi}) $\newline$
$$\textrm{FO}(\textrm{LFP}) = \textrm{P}  $$
$\newline$
i.e. a class of ordered structures is definable in $\textrm{FO}(\textrm{LFP})$ if and only if it is accepted by a deterministic polynomial-time machine.
\end{theorem}

For a proof we must first mention some definitions and facts.

\begin{definition}(Alternating Reachability $\textrm{REACH}_a$) $\newline$
Let $\tau_{ag}= \{ E, A, s, t\}$ be the vocabulary of alternating graphs. An alternating graph $\mathcal{G} = (G, E^{\mathcal{G}} , A^{\mathcal{G}}, s^{\mathcal{G}}, t^{\mathcal{G}})$ is a directed graph
whose vertices are labeled universal or existential. $A^{\mathcal{G}}\subseteq G$
is the set of universal vertices. $\newline$ Let
$P^\mathcal{G}_a xy$ be the smallest relation on vertices of $\mathcal{G}$ such that :

\begin{enumerate}
\item[(1)] $P_a^\mathcal{G}xx$
\item[(2)] If $x$ is existential and $P_a^\mathcal{G}zy$ holds for some edge $(x,z)$ then $P^\mathcal{G}_a xy$
\item[(3)] If $x$ is universal, and there is at least one edge leaving $x$, and $P^\mathcal{G}_a zy$ holds for all edges $(x,z)$, then $P^\mathcal{G}_a xy$
\end{enumerate}

$$\textrm{REACH}_a := \{ \mathcal{G} \; | \; P_a^\mathcal{G} st \}$$

\end{definition}

It is clear that $P_a^{\mathcal{G}}st \;\; \text{iff} \;\; [\textrm{LFP}_{R,x,y}\varphi]st$ where $$\varphi := x=y \vee (Ax\wedge \exists z Exz \wedge \forall z (Exz \rightarrow Rzy)) \vee (\neg Ax \wedge \exists z (Exz \wedge Rzy))$$
also it is easy to see that if $(x,y)\in (F^{\varphi,\mathcal{G}})^k(\emptyset)$ then the distance between $x$ and $y$ in the underlying graph is $k-1$, hence the depth of $\varphi$ is $n$. Thus $\textrm{REACH}_a \in \textrm{IND}[n]$.

\begin{theorem}\label{2.2.2} \cite{I83} $\newline$
$\textrm{REACH}_a$ is complete for $\textrm{P}$ under first-order reductions.
\end{theorem}

\begin{definition}

 Let $\tau$, $\sigma$, $I$ be as in Definition \ref{1.4.2}. Then $I$ also defines a dual map, which is called $\hat{I}$, from $\textrm{FO}(\textrm{LFP})[\tau]$ to $\textrm{FO}(\textrm{LFP})[\sigma]$. For any formula $\varphi \in \textrm{FO}(\textrm{LFP})[\tau] $, $\hat{I}(\varphi)$ is the result of replacing all relation and constant symbols in $\varphi$ by the corresponding formulas from the definition of $I$, using a map $f_{I}$ defined as follows :
 \begin{enumerate}
 \item[-] Each variable is mapped to a $k$-tuple of variables : $$ f_I (v) := v^1,\ldots,v^k $$
 \item[-] Input relations are replaced by their corresponding formulas : $$f_I (R_i v_1 \ldots v_{a_i}) := \varphi_i (f_I(v_1),\ldots,f_I(v_{a_i}))$$
\item[-] Quantifiers are replaced by restricted quantifiers : $$f_I (\exists v) := (\exists f_I(v), \varphi_0 (f_I(v)))$$
\item[-] The equality relation and other numeric relations are replaced by their appropriate formulas.
\item[-] Second-order variables have their arities multiplied by $k$; and second-order quantifiers must be restricted.
\item[-] On boolean connectives, $f_I$ is the identity.
\item[-] Constant $c_i$
s replaced by a $k$-tuple of special variables (i.e. variables not used in any other place of the formula being formed) : $$f_I(c_i) := z_i^1,\ldots,z_i^k$$
    and these variables must be quantified before they are used. Typically, these quantifiers can be placed at the beginning of the formula.
 \end{enumerate}
Now the mapping $\hat{I}$ is defined as follows, for $\theta \in \textrm{FO}(\textrm{LFP})[\tau]$ : $\newline\hat{I}(\theta) := \newline \small{(\exists z_1^1 \ldots \exists z_1^k, (\varphi_0\wedge\psi_1)(z_1^1,\ldots,z_1^k)) \ldots (\exists z_s^1 \ldots \exists z_s^k, (\varphi_0\wedge\psi_s)(z_s^1,\ldots,z_s^k)) (f_I (\theta))}$
 \end{definition}

 \begin{proposition} \cite{Immerman}
Let $\tau$, $\sigma$, and $I$ be as in Definition \ref{1.4.2}. Then for all formulas $\theta \in \textrm{FO}(\textrm{LFP})[\tau]$ and all structures $\mathcal{A}\in \textrm{STRUC}[\sigma]$, and all assignments $\alpha$ in $\mathcal{A}$,
               $$\mathcal{A} \vDash \hat{I}(\theta)[\alpha] \;\; \text{iff} \;\; I(\mathcal{A})\vDash \theta[\alpha^{\prime}]$$
               Where $\alpha^{\prime}$ is defined at a variable $x$ iff  $(\alpha(x^1),\ldots,\alpha(x^k))\in I(\mathcal{A})$ (Recall that $f_I(x)=x^1,\ldots,x^k$). In this case $\alpha^{\prime}(x) = ( \alpha(x^1),\ldots,\alpha(x^k) )$.
\end{proposition}

 \begin{remark}
Let $\sigma,\tau,I$ be as in Definition \ref{1.4.2}, let $K$ be a class of $\sigma$-structures and $H$ a class of $\tau$-structures. Suppose that $K$ is first-order reducible to $H$ via $I$, and $H$ is definable by the least-fixed point sentence $\varphi$. Then $\mathcal{A} \in K$ iff $I(\mathcal{A})\in H$ iff $I(\mathcal{A}) \vDash \varphi$ iff $\mathcal{A} \vDash \hat{I}(\varphi)$. Thus, to prove that $\textrm{FO}(\textrm{LFP})$ is closed under first-order reductions, it suffices to prove that for any least-fixed point formula $\varphi$ from $\textrm{FO}(\textrm{LFP})[\tau]$, $\hat{I} (\varphi)$ is in $\textrm{FO}(\textrm{LFP})[\sigma]$. But this follows directly from the fact that $\textrm{FO}(\textrm{LFP})$ is closed under boolean operations and first-order quantification.
\end{remark}

\begin{remark}
Note that if $\theta := [\textrm{LFP}_{R,x,y}\varphi(R,x,y)]xy$ is of depth $n^m$, then : $\newline \hat{I}(\theta)$ is $[\textrm{LFP}_{S,x^1,\ldots,x^k,y^1,\ldots,y^k} (\forall \overline{t}(S\overline{t}\rightarrow \varphi_o(\overline{t}))\wedge \varphi_0(\overline{x}) \wedge \varphi_0(\overline{y}) \wedge \hat{I}(\varphi))]\overline{x}\overline{y}$, and is of depth $n^{km}$, because for any structure $\mathcal{A}$, and any formula $\varphi$, $\varphi$ takes on $I(\mathcal{A})$ exactly the same number of iterations as $\hat{I}(\varphi)$ on $\mathcal{A}$ to come to its fixed point, while $\|I(\mathcal{A})\| = \|\mathcal{A}\|^k$. For this reason we tend to think that $\textrm{IND}[n^m]$ would not turn out to be closed under first-order reductions.
\end{remark}

\begin{remark}
The hierarchy $\textrm{IND}[n] \subseteq \textrm{IND}[n^2] \subseteq \textrm{IND}[n^3] \subseteq \ldots$ collapses to the $i$-th level if and only if $\textrm{IND}[n^i]$ is closed under first-order reductions. In that case $\textrm{P} = \textrm{IND}[n^i]$. To see this, note that from Theorem \ref{2.2.1} $\textrm{P}=\textrm{FO}(\textrm{LFP})$, and since $\textrm{FO}(\textrm{LFP}) = \underset{k=1}{\overset{\infty}{\bigcup}} \textrm{IND}[n^k]$, then $\textrm{P} = \underset{k=1}{\overset{\infty}{\bigcup}}\textrm{IND}[n^k] $. So, if this hierarchy collapses to the $i$-th level, then $\textrm{P} = \textrm{IND}[n^i]$, and then $\textrm{IND}[n^i]$ is closed under first-order reductions because $\textrm{P}$ is so closed. $\textrm{REACH}_a \in \textrm{IND}[n^i]$, for any $i \geq 1$, and it is $\textrm{P}$-complete under first-order reductions. Thus, if $\textrm{IND}[n^i]$ is closed under first-order reductions, then $\textrm{P}=\textrm{IND}[n^i]$, i.e., the hierarchy collapses to the $i$-th level.
\end{remark}

We also need the following fact :

\begin{theorem} \cite{Immerman}
Every first-order boolean query is computable in deterministic logspace.
\end{theorem}

Now we prove Theorem \ref{2.2.1}
\begin{proof}
\begin{enumerate}
\item[] $\textrm{FO}(\textrm{LFP}) \subseteq \textrm{P}$ : $\newline$
Let $\mathcal{A}$ be an input structure, let $n = \|\mathcal{A}\|$, and let $[\textrm{LFP}_{R,x_1,\ldots,x_k}\varphi]\overline{y}$ be a least fixed-point formula (where $\varphi$ is first-order). This fixed point evaluated on $\mathcal{A}$ is $(F^{\varphi,\mathcal{A}})^{n^k}(\emptyset)$. We know from the previous theorem that for any $S \subseteq A^k$, any $( a_1,\ldots,a_k )$, whether $\mathcal{A}$ satisfies $\varphi(S,a_1,\ldots,a_k)$ can be determined in logspace. Let $M_0$ be the logspace Turing machine that computes the query $\varphi(S,a_1,\ldots,a_k)$ (The vocabulary now contains an extra $k$-ary relation symbol $S$, and extra $k$ constants). $\newline$ Let $M$ be the Turing machine that, starting with $S = \emptyset$ (let it have a tape for coding $S$ in $n^k$ bits), cycles through all possibilities for $(a_1,\ldots,a_k)$, substitutes each of these for $( x_1,\ldots,x_k )$ and runs $M_0$ to see whether $\mathcal{A} \vDash \varphi(S,a_1,\ldots,a_k)$, then writes $1$ on the bit corresponding to $( a_1,\ldots,a_k )$ in the tape of $S$  if this holds.$\newline$ The machine $M$ does this for $n^k$ times (the number of tuples). $S$ is coded on $n^k$ bits of a tape, so rewriting a bit of it takes, in addition to the logspace work of $M_0$, a polynomial time (to reach the bit that will be rewritten). From this and since $M_0$ is run $n^k$ times (the number of stages to reach the fixed-point), and $M_0$ is logspace, and $\textrm{L} \subseteq \textrm{P}$, then $M$ is a polynomial time machine.
\item[] $\textrm{P} \subseteq \textrm{FO}(\textrm{LFP})$ : $\newline$
         Since any query in $\textrm{P}$ is first-order reducible to $\textrm{REACH}_a$ by Theorem \ref{2.2.2}, and $\textrm{REACH}_a \in \textrm{FO}(\textrm{LFP})$, and $\textrm{FO}(\textrm{LFP})$ is closed under first-order reductions, then any query in $\textrm{P}$ is in $\textrm{FO}(\textrm{LFP})$.
\end{enumerate}
\end{proof}

There is an analogue of this theorem for polynomial space. The complexity class of polynomial space has several descriptive characterizations. One of these is in terms of depth of inductive definitions; it is the set of boolean queries expressible by first-order quantifier blocks iterated exponentially :

\begin{theorem} $\newline$\label{2.2.11}
$$\textrm{PSPACE} = \textrm{FO}(\textrm{PFP}) = \textrm{FO}[2^{n^{O(1)}}]$$
\end{theorem}
The equality of $\textrm{PSPACE}$ and $\textrm{FO}(\textrm{PFP})$ was shown by Vardi (1982) \cite{Vardi}, and the fact that $\textrm{PSPACE} = \textrm{FO}[2^{n^{O(1)}}]$ was shown by Immerman (1980) \cite{Im}.
$\newline$
We now exhibit without proof theorems relating certain inductive depths (possibly with another restriction) with certain complexity classes (with a corresponding restriction).$\newline$ We begin with the theorem showing that inductive depth equals parallel-time and equals circuit depth in circuit complexity, but first we should specify the definitions of these complexity measures :

\begin{definition} \cite{Immerman} $\newline$
A \emph{concurrent random access machine} (\textrm{CRAM}) consists of a large number of processors, all connected to a common, global memory. The processors are identical except that they each contain a unique processor number. At each step, any number of processors may read or write any word of global memory. If several processors try to write the same word at the same time, then the lowest numbered processor succeeds.$\newline$ This is the "priority write" model. The results in this section remain true if instead we use the ``common write" model, in which the program guarantees that different values will never be written to the same location at the same time. $\newline$
Each processor has a finite set of registers, including the following, Processor : containing the number between $1$ and $p(n)$ of the processor (the number of processors in the machine is a function $p(n)$ of the size of the input); Address : containing an address of global memory; Contents : containing a word to be written or read from global memory; and ProgramCounter : containing the line number of the instruction to be executed next. The instructions of a \textrm{CRAM} consist of the following : $\newline\newline$
\textrm{READ} : Read the word of global memory specified by Address into Contents.$\newline$
\textrm{WRITE} : Write the Contents register into the global memory location specified by Address.$\newline$
\textrm{OP} $R_a$ $R_b$ : Perform \textrm{OP} on $R_a$ and $R_b$ and leave the result in $R_b$. Here \textrm{OP} may be Add, Subtract, or, Shift (where Shift$(x,y)$ is the instruction that causes the word $x$ to be shifted $y$ bits to the right).$\newline$
\textrm{MOVE} $R_a$ $R_b$ : Move $R_a$ to $R_b$. $\newline$
\textrm{BLT} $R$ $L$ : Branch to line $L$ if the content of $R$ is less than zero.

The above instructions each increment the ProgramCouter, with the exception of \textrm{BLT} which replaces it by $L$, when the content of $R$ is less than zero. $\newline$
We assume initially that the contents of the first $|bin(\mathcal{A})|$ words of global memory contain one bit each of the input string $bin(\mathcal{A})$. We assume also that a section of global memory is specified as the output. One of the bits of the output may serve as a flag indicating that the output is available. $\newline$
Our measure of parallel-time complexity will be time on a $\textrm{CRAM}$. Define $\textrm{CRAM}[t(n)]$ to be the set of boolean queries computable in parallel-time $t(n)$ on a $\textrm{CRAM}$ that has at most polynomially many processors.
\end{definition}

\begin{definition} (Circuit Complexity) $\newline$
Let $t(n)$ be a polynomially-bounded function and let $S$ be a class of finite structures (of some fixed vocabulary). $S$ is in $\textrm{AC}[t(n)]$ if there is a uniform class of circuits $C_1,C_2,\ldots$ with the following properties :
\begin{enumerate}
\item[(1)] For all structures $\mathcal{A}$, $\mathcal{A}\in S \;\;\; \text{iff} \;\;\; C_{||\mathcal{A}||} \; \text{accepts} \; \mathcal{A}$.
\item[(2)] There is a function $f(n)= O(t(n))$ such that for every $n$, the depth of $C_n$ is $f(n)$.
\item[(3)] For every $n$, the gates of $C_n$ consist of unbounded fan-in "and" and "or" gates.

\end{enumerate}
\end{definition}

\begin{theorem}\label{2.2.14}\cite{Expressibility},\cite{Immerman} $\newline$
For all polynomially-bounded parallel-time constructible $t(n)$ :
 $$\textrm{IND}[t(n)] = \textrm{FO}[t(n)] = \textrm{CRAM}[t(n)] = \textrm{AC}[t(n)]$$
\end{theorem}

The number of variables in an inductive definition determines the number of processors needed in the corresponding $\textrm{CRAM}$ computation. The intuitive idea is that using $k$ $\log{n}$-bit variables, we can name approximately $n^k$ different parts of the $\textrm{CRAM}$. Thus, very roughly, $k$ variables correspond to $n^k$ processors. The coming theorem illustrates the relationship between the number of variables in an inductive definition and the number of processors in the corresponding $\textrm{CRAM}$.$\newline$ Before the theorem we must mention some definitions :

\begin{definition} \cite{Immerman}
$\textrm{CRAM-PROC}[t(n),p(n)]$ is the complexity class $\textrm{CRAM}[t(n)]$ restricted to machines using at most $O(p(n))$ processors.
\end{definition}

\begin{definition} \cite{Immerman}
$\textrm{IND-VAR}[t(n),s]$ is the complexity class $\textrm{IND}[t(n)]$ restricted to inductive definitions using at most $s$ distinct variables.
\end{definition}

\begin{theorem}\label{2.2.17} \cite{Expressibility}
If the maximum size of a register word and $t(n)$ are both $o(\sqrt{n})$ and $k\geq 1$ is a natural number, then
$$\small{\textrm{CRAM-PROC}[t(n),n^k]  \subseteq  \textrm{IND-VAR} [t(n), 2k+2]  \subseteq \textrm{CRAM-PROC}[t(n),n^{2k+2}]}$$
\end{theorem}
We conclude the section with a theorem illustrating the relationship between inductive depth and alternating complexity. (Recall that the definition of alternating Turing machines involved alternating between universal and existential states and that in $\textrm{FO}[t(n)]$ when the quantifier block is iterated the quantifiers alternate).

\begin{definition}\cite{Immerman}

$\textrm{ASPACE-ALT}[f(n),g(n)]$ is the class of boolean queries accepted by alternating Turing machines simultaneously using space $f(n)$ and making at most $g(n)$ alternations between existential and universal states starting with existential.
\end{definition}

\begin{theorem} (Ruzzo and Tompa 1984 \cite{Stockmeyer})$\newline$
For $t(n) \geq \log{n}$,
$$\textrm{ASPACE-ALT}[\log{n},t(n)] = \textrm{AC}[t(n)] = \textrm{FO}[t(n)]$$
\end{theorem}

\section{Examples}
In this section we give examples to show that, in general, there is no order comparison between the depth of a formula and its inflationary depth, and no order comparison between the depth of the iteration it takes to come to the simultaneous fixed-point of two formulas and the depth of the iteration it takes to come to their nested fixed-point, if the relation variables through which the iteration is made are not restricted to be positive. (Throughout the section $S$ denotes the usual successor).

\begin{example}
Consider the formula $\newline\varphi(X,x) := \newline(\forall y \neg X y \wedge x=min)
 \newline\vee \newline (\exists y(Xy \wedge \forall z(Xz\rightarrow z=y)) \wedge \newline(\exists y (Xy\wedge Syx) \vee (X max \wedge (Sx max \vee x= max))))\newline \vee \newline (\exists y \exists z (Xy \wedge Xz \wedge \forall u (Xu \rightarrow (u=y\vee u=z)) \wedge y<z \wedge\newline ( (y\neq 0 \wedge (x=z \vee Sxy)) \vee (y=0\wedge (Sxz \vee \exists u (Sxu\wedge Suz))))))\newline \vee \newline (X0 \wedge X1 \wedge \forall y (Xy \rightarrow (y=0 \vee y=1)) \wedge (x=0 \vee x=1))$
$\newline\newline$
For any positive integer $n$ and any structure $\mathcal{A}$ of size $n$, the sequence $\emptyset , F^{\varphi,\mathcal{A}}(\emptyset), (F^{\varphi,\mathcal{A}})^2(\emptyset),\ldots$ is as follows :$\newline$
$\emptyset, \{0\},\{1\},\ldots,\{n-1\},\{n-2,n-1\},\{n-3,n-1\},\newline \{n-4,n-1\},\ldots,\{0,n-1\},\{n-3,n-2\},\{n-4,n-2\},\ldots,\newline \{0,n-2\},\{n-4,n-3\},\ldots,\{1,2\},\{0,2\},\{0,1\},\{0,1\},\{0,1\},\ldots$
$\newline\newline$
The fixed-point is $\{0,1\}$ and it comes after $n+(n-1)+(n-2)+\ldots+2+1$ iterations, i.e., after $\frac{n(n+1)}{2}$ iterations.
While for the inflationary fixed-point, which is the fixed-point of the formula $Xx\vee  \varphi(X,x)$, we have the sequence :
$\emptyset, \{0\}, \{0,1\},\{0,1\},\ldots$.
So the fixed-point comes just  after two iterations. Hence $|\varphi|(n) > |\varphi|_{\textrm{IFP}}(n)$; $|\varphi|(n) \in \Theta(n^2)$ while $|\varphi|_{\textrm{IFP}}(n)$ is constant.
\end{example}

\begin{example}
Consider the formula $\newline$
$\psi(X,x) := \newline (\forall y \neg Xy \wedge \forall z (z<x \rightarrow Xz)) \vee (\exists y Xy \wedge x=1) \newline \vee(\exists u \exists v (Xu\wedge Xv \wedge u\neq v) \wedge \forall z (z<x \rightarrow Xz))$ $\newline\newline$
For any positive integer $n$ and any structure $\mathcal{A}$ of size $n$, the sequence $\emptyset , F^{\psi,\mathcal{A}}(\emptyset), (F^{\psi,\mathcal{A}})^2(\emptyset),\ldots$ is as follows :$\newline$
$\emptyset, \{0\},\{1\},\{1\},\{1\},\ldots$ $\newline$
That is the fixed-point of this formula comes just after two iterations, while for its inflationary fixed-point we have the sequence : $\newline\newline$
$\emptyset,\{0\},\{0,1\},\{0,1,2\},\{0,1,2,3\},\ldots,\{0,1,\ldots,n-1\},\{0,1,\ldots,n-1\},\ldots$ $\newline$
That is, its inflationary fixed-point comes after $n$ iterations. Thus $|\psi|(n)=2$ while $|\psi|_{\textrm{IFP}}(n)=n$.

\end{example}

The following example shows that the nested fixed-point of two formulas can actually take more iterations than their simultaneous fixed-point when each one of them is positive in both $X$ and $Y$.

\begin{example} $\newline$
Consider the formulas $\newline$ $\varphi(X,Y,x) := x=0 \vee \exists z (Xz\wedge Szx)$ $\newline$ and $\newline$ $\psi(X,Y,y) := y=0 \vee \exists z (Yz \wedge Szy)$. $\newline\newline$ In any structure of size $n$ the computation of their simultaneous fixed-point goes on as follows : $\newline\newline$
$(\emptyset,\emptyset), (\{0\},\{0\}),(\{0,1\},\{0,1\}),\ldots, (\{0,1,\ldots,n-1\},\{0,1,\ldots,n-1\}),\newline(\{0,1,\ldots,n-1\},\{0,1,\ldots,n-1\}),\ldots$
$\newline\newline$
i.e. it takes $n$ iterations to compute the simultaneous fixed-point.$\newline$
While for the nested fixed-point $[\textrm{LFP}_{X,x}\varphi(X,[\textrm{LFP}_{Y,y}\psi(X,Y,y)],x)]$, First $X$ is interpreted with $\emptyset$ and $n$ iterations take place to compute the fixed-point of $\psi$ before we can evaluate the new interpretation of $X$ which is $\{0\}$. In general, with every interpretation of $X$ it takes $n$ iterations to compute the fixed-point of $\psi$ before we can evaluate the new interpretation of $X$; and it can be easily seen that the interpretations of $X$ grow as follows $\emptyset, \{0\}, \{0,1\},\ldots, \{0,1,\ldots,n-1\}, \{0,1,\ldots,n-1\},\ldots $. Thus it takes a total of $n(n+1)$ iterations to come to the nested fixed-point.
\end{example}

By intuition one expects the nested fixed-point of two formulas to take more iterations than their simultaneous fixed-point if the relation variables through which the iteration is made are positive.$\newline$ \textbf{Open Question :} Are there formulas $\varphi(X,Y,\overline{x})$ and $\psi(X,Y,\overline{y})$ positive in both $X$ and $Y$ such that their simultaneous fixed-point takes ``more" iterations than one of the two nested fixed-points ?$\newline$
If the occurrences of $X$ or $Y$ are not restricted to be positive, we can find such examples :

\begin{example}$\newline$
Let $\newline$ $\varphi(X,Y,x) := \newline x=0 \vee (x=1 \wedge \forall z (z\neq max \rightarrow Yz)) \vee  \exists z (Xz \wedge z\neq 0 \wedge Szx) \newline \vee \newline (X0 \wedge X1 \wedge \forall z Yz \wedge x=x)$ $\newline \newline$ and $\newline \newline$ $\psi(X,Y,y) := \newline  y=0 \vee \exists z (Yz \wedge Szy \wedge y\neq max)  \vee (\exists z Yz \wedge (y=y \rightarrow \neg \exists z Xz)) \newline \vee \newline (X0 \wedge X1 \wedge \neg \exists z Yz \wedge y=y ) \vee (\forall z Yz \wedge y=y)$ $\newline\newline$
$\varphi$ is positive in both $X$ and $Y$, but $\psi$ is not positive in $X$ nor in $Y$. $\newline\newline$
In any structure $\mathcal{A}$ of size $n$ the computation of their simultaneous fixed-point goes on as follows : $\newline$

$(\emptyset,\emptyset), (\{0\},\{0\}), (\{0\},\{0,1\}), (\{0\},\{0,1,2\}), \ldots, (\{0\},\{0,1,\ldots,n-2\}), \newline (\{0,1\},\{0,1,\ldots,n-2\})
, (\{0,1,2\},\{0,1,\ldots,n-2\}), \ldots,\newline(\{0,1,\ldots,n-1\},\{0,1,\ldots,n-2\}),
(\{0,1,\ldots,n-1\},\{0,1,\ldots,n-2\}), \ldots $
$\newline$
i.e. it takes $2n-2$ iterations to come to the simultaneous fixed-point. While for the nested fixed-point $[\textrm{LFP}_{X,x}\varphi(X,[\textrm{LFP}_{Y,y}\psi(X,Y,y)],x)]$, when first $X$ is interpreted with $\emptyset$, the computation of the fixed-point of $\psi(\emptyset,Y,y)$ goes on as follows :
$\emptyset,\{0\},\{0,1,\ldots,n-1\},\{0,1,\ldots,n-1\},\ldots$, that is, its fixed-point comes in just two iterations. Now $X$ has the new interpretation $\{a\in A \; | \; \mathcal{A} \vDash \varphi[\emptyset,\{0,1,\ldots,n-1\},a]\}$ which is $\{0,1\}$. With this interpretation of $X$ the computation of the fixed-point of $\psi(\{0,1\},Y,y)$ goes on as follows : $\emptyset, \{0,1,\ldots,n-1\},\{0,1,\ldots,n-1\},\ldots$, that is, its fixed point comes in just one iteration.
The new interpretation of $X$ should be $\{a\in A \; | \; \mathcal{A} \vDash \varphi[\{0,1\},\{0,1,\ldots,n-1\},a] \}$ which is $\{0,1,\ldots,n-1\}$. With this new interpretation of $X$ the fixed-point of $\psi(\{0,1,\ldots,n-1\},Y,y)$ is $\{0,1,\ldots,n-1\}$; and $\{a\in A \; | \; \mathcal{A} \vDash \varphi[\{0,1,\ldots,n-1\},\{0,1,\ldots,n-1\},a]\}$ is equal to $\{0,1,\ldots,n-1\}$. Hence $\{0,1,\ldots,n-1\}$ is the nested fixed-point and it comes after a total of $(2+1+1+1)$, i.e., 5 iterations.

\end{example}

\begin{example}
 $\newline$
Let $\varphi(X,Y,x) := (\forall u \neg Xu \wedge x=0) \vee (\exists u Xu \wedge x=1)$ $\newline\newline$ and $\newline\newline$ $\psi(X,Y,y) := (\exists u Xu \wedge \forall u \neg Yu \wedge y=y) \vee y=0 \vee \exists z (Yz \wedge Szy)$ $\newline\newline$ $\varphi$ is not positive in $X$, and $\psi$ is not positive in $Y$.
$\newline\newline$ For any structure $\mathcal{A}$ of size $n$ the computation of their simultaneous fixed-point goes on as follows : $\newline$

$(\emptyset,\emptyset), (\{0\},\{0\}), (\{1\},\{0,1\}), (\{1\},\{0,1,2\}), \ldots,\newline (\{1\},\{0,1,\ldots,n-1\}),(\{1\},\{0,1,\ldots,n-1\}), \ldots$ $\newline\newline$
i.e. it takes $n$ iterations to come to the simultaneous fixed-point.
$\newline$
While for the nested fixed-point $[\textrm{LFP}_{Y,y}\psi([\textrm{LFP}_{X,x}\varphi(X,Y,x)],Y,y)]$, when first $Y$ is interpreted with $\emptyset$, the computation of the fixed point of $\varphi(X,\emptyset,x)$ goes on as follows :
$\emptyset,\{0\},\{1\},\{1\},\ldots$, that is, it takes just two iterations to come to its fixed-point. Now $Y$ has the new interpretation $\{a\in A \; | \; \mathcal{A} \vDash \psi[\{1\},\emptyset,a]\}$ which is $\{0,1,\ldots,n-1\}$. With this interpretation of $Y$ the fixed point of $\varphi(X,\{0,1,\ldots,n-1\},x)$ is $\{1\}$; and $\{a\in A \; | \; \mathcal{A} \vDash \psi[\{1\},\{0,1,\ldots,n-1\},a]\}$ is equal to $\{0,1,,\ldots,n-1\}$. Hence $\{0,1,\ldots,n-1\}$ is the nested fixed point and it comes after a total of $(2+1)$, i.e., 3 iterations.
\end{example}

One more example to show that the nested fixed-point can take actually more iterations than the simultaneous fixed-point; but here the occurrences of the relation variables are not restricted to be positive.

\begin{example}
$\newline$
Let $\newline$ $\varphi(X,Y,x) := \newline(\exists u Yu \wedge\newline \exists z (\forall u (Yu \rightarrow z\leq u) \wedge (\forall v (\forall u(Yu \rightarrow v\leq u)\rightarrow v\leq z)) \wedge (x=z \vee Sxz))) \newline\vee \newline \exists z (Xz\wedge Sxz) \vee x=0$ $\newline\newline$ and $\newline\newline$ $\psi(X,Y,y) := \newline\small{\exists z (\forall u (Xu \rightarrow u<z) \wedge (\forall v (\forall u (Xu \rightarrow u <v) \rightarrow z \leq v)) \wedge y=z)}  \vee \newline \exists z (Yz \wedge Szy) \vee Yy$.
$\newline\newline$
For any structure $\mathcal{A}$ of size $n$ the computation of their simultaneous fixed-point goes on as follows : $\newline$
$(\emptyset,\emptyset), (\{0\},\{0\}),
(\{0\},\{0,1\}), (\{0\},\{0,1,2\}),\newline (\{0\},\{0,1,\ldots,n-1\}),(\{0\},\{0,1,\ldots,n-1\}),\ldots $
$\newline$
i.e. it takes $n$ iterations to come to the simultaneous fixed-point. While for the nested fixed-point $[\textrm{LFP}_{X,x}\varphi(X,[\textrm{LFP}_{Y,y}\psi(X,Y,y)],x)]$, when first $X$ is interpreted with $\emptyset$, the computation of the fixed-point of $\psi(\emptyset,Y,y)$ goes on as follows :$\newline$
$\emptyset,\{0\},\{0,1\},\{0,1,2\},\ldots,\{0,1,\ldots,n-1\},\{0,1,\ldots,n-1\},\ldots$ $\newline$ that is $n$ iterations to come to the fixed-point. The new interpretation of $X$ is $\{a\in A \; | \; \mathcal{A} \vDash \varphi[\emptyset,\{0,1,\ldots,n-1\},a]\}$ which is $\{0\}$. With this new interpretation of $X$, the computation of the fixed-point of $\psi(\{0\},Y,y)$ goes on as follows : $\newline$ $\emptyset, \{1\},\{1,2\},\{1,2,3\},\ldots,\{1,2,\ldots,n-1\},\{1,2,\ldots,n-1\},\ldots$ $\newline$ that is $n-1$ iterations to come to the fixed-point; and then the new interpretation of $X$ is $\{a\in A \; |\; \mathcal{A} \vDash \varphi[\{0\},\{1,2,\ldots,n-1\},a]\}$ which is $\{0,1\}$.
$\newline$
In general the interpretations of $X$ grow as follows $\newline$ $\emptyset, \{0\},\{0,1\},\{0,1,2\},\ldots,\{0,1,\ldots,n-1\},\{0,1,\ldots,n-1\},\ldots$, $\newline$ and $\{0,1,\ldots,n-1\}$ is the nested fixed point, while for every one of these interpretations before the fixed-point, say $\{0,1,\ldots,m\}$ with $m<n-1$, the computation of the fixed-point of $\psi(\{0,1,\ldots,m\},Y,y)$ goes on as follows $\newline$
$\emptyset,\{m+1\},\{m+1,m+2\},\{m+1,m+2,m+3\},\ldots,\newline\{m+1,m+2,\ldots,n-1\},\{m+1,m+2,\ldots,n-1\},\ldots$ $\newline\newline$ i.e. $n-(m+1)$ iterations to come to the fixed-point. Hence the total number of iterations to come to the nested fixed-point is $\newline$ $n+1+(n-1)+1+(n-2)+1+\ldots+2+1+1+1$ i.e. $\frac{n(n+1)}{2}+n$.
\end{example}

\chapter{Finite Variable Logics}
We have seen in the previous chapter that parallel-time and inductive depth are equal and are equal to the depth of iteration of a quantifier block (with each iteration the quantifier rank of the resulting formula increases). We have also seen that the number of variables in an inductive definition is related in some way to the number of processors in the corresponding $\textrm{CRAM}$. Accordingly, it is important to study the expressive power of number of variables and quantifier rank. This takes us to Finite Variable Logics.

\section{Pebble Games}
In this section we introduce the basics of finite variable logics and their respective games (pebble games). This section is mainly from \cite{Flum}.
\begin{definition}
$\textrm{FO}^k$ is the fragment of first-order logic of formulas that use at most $k$ distinct variables (free and bound), and $\textrm{FO}[k,n]$ is the fragment of first-order logic of formulas of quantifier rank at most $n$ that use at most $k$ distinct variables (free and bound).
\end{definition}

\begin{definition} (Pebble Games) $\newline$
Fix a vocabulary $\tau$ and let $*$ be a symbol that does not belong to the universe of any structure. For $\overline{a}\in (A\cup\{*\})^s$, $\overline{a} = a_1 \ldots a_s$, let $\newline$ $supp(\overline{a}) : = \{i \; | \; a_i \in A\}$ be the \emph{support} of $\overline{a}$, and if $a\in A$, let $\overline{a} \frac{a}{i}$ denote $a_1\ldots a_{i-1}a\; a_{i+1}\ldots a_s$. For $\overline{a}\in (A\cup\{*\})^s$ and $\overline{b}\in (B\cup\{*\})^s$ we say that $\overline{a}\mapsto\overline{b}$ is an $s$-\emph{partial isomorphism} from $\mathcal{A}$ to $\mathcal{B}$ if $supp(\overline{a}) = supp(\overline{b})$ and $\overline{a}^{\prime} \mapsto \overline{b}^{\prime}$ is a partial isomorphism from $\mathcal{A}$ to $\mathcal{B}$, where $\overline{a}^{\prime}$ and $\overline{b}^{\prime}$ are the subsequences of $\overline{a}$ and $\overline{b}$ with indices in the support.
$\newline$ Let $\mathcal{A}$ and $\mathcal{B}$ be structures, $\overline{a}\in (A\cup\{*\})^s$, $\overline{b}\in (B\cup\{*\})^s$ with $supp(\overline{a}) = supp(\overline{b})$. In the \emph{pebble game} $G_m^s(\mathcal{A},\overline{a},\mathcal{B},\overline{b})$ we have $s$ pebbles $\alpha_1,\ldots ,\alpha_s$ for $\mathcal{A}$ and $s$ pebbles $\beta_1,\ldots ,\beta_s$ for $\mathcal{B}$. Initially, $\alpha_i$ is placed on $a_i$ if $a_i\in A$, and off the board if $a_i=*$, and similarly for $B$, $\beta_i$ and $b_i$.
$\newline$
 A play of this game consists of $m$ moves. In its $j$-th move, Spoiler selects a structure $\mathcal{A}$ or $\mathcal{B}$, and a pebble for this structure (being off the board or already placed on an element). If it selects $\mathcal{A}$ and $\alpha_i$, it places $\alpha_i$ on some element of $\mathcal{A}$, and then Duplicator places $\beta_i$ on some element of $\mathcal{B}$. If Spoiler selects $\mathcal{B}$ and $\beta_i$, it places $\beta_i$ on an element of $\mathcal{B}$ and Duplicator places $\alpha_i$ on some element of $\mathcal{A}$. (Note that there may be several pebbles on the same element).$\newline$
Duplicator wins the game if for each $j\leq m$ we have that $\overline{e} \mapsto \overline{f}$ is an $s$-partial isomorphism, where $\overline{e}= e_1 \ldots e_s$ are the elements marked by $\alpha_1,\ldots,\alpha_s$ after the $j$-th move ( $e_i = *$ in case $\alpha_i$ is off the board ) and where $\overline{f} = f_1 \ldots f_s$ are the corresponding values given by $\beta_1,\ldots,\beta_s$. For $j=0$ this means $\overline{a} \mapsto \overline{b}$ is an $s$-partial isomorphism. The pebble game $G^s_{\infty}(\overline{A},\overline{a},\overline{B},\overline{b})$ with infinitely many moves is defined similarly. We use $G_m^s(\mathcal{A},\mathcal{B})$ as abbreviation for $G_m^s(\mathcal{A},*\ldots*,\mathcal{B},*\ldots*)$ and $G_{\infty}^s(\mathcal{A},\mathcal{B})$ for $G_{\infty}^s(\mathcal{A},*\ldots*,\mathcal{B},*\ldots*)$.
\end{definition}

\begin{example}
Consider the orderings $\mathcal{A} :=(\{a,b\},<)$ and $\newline\mathcal{B}:=(\{c,d,e\},<)$ where $a<b$ and $c<d<e$. Spoiler has a winning strategy in $G_3^3(\mathcal{A},\mathcal{B})$. Spoiler first pebbles $c$ in $\mathcal{B}$ and Duplicator has to pebble $a$ in $\mathcal{A}$ with the corresponding pebble, because if it pebbles $b$ instead, it will not be able to reply in the second move when Spoiler pebbles something greater than $c$ with a second pebble. In the second move Spoiler pebbles $e$ in $\mathcal{B}$ with a second pebble and Duplicator has to pebble $b$ in $\mathcal{A}$ with the corresponding pebble. In the third move Spoiler pebbles $d$ in $\mathcal{B}$ with the third pebble and Duplicator cannot reply because there is no element between $a$ and $b$ in $\mathcal{A}$ corresponding to the element $d$ between $c$ and $e$ in $\mathcal{B}$.
\end{example}

\begin{definition}
$\newline$
 Structures $\mathcal{A}$ and $\mathcal{B}$ are $s$-$m$-isomorphic, in symbols $\mathcal{A} \cong_m^s \mathcal{B}$, if there is a sequence  $(I_j)_{j\leq m}$ of nonempty sets of $s$-partial isomorphisms with the following properties :
 \begin{enumerate}
\item[] ( $s$-forth property ) For $j < m$, $\overline{a} \mapsto \overline{b} \in I_{j+1}$, $1 \leq i \leq s$, and $a \in A$, there is $b\in B$ such that $\overline{a} \frac{a}{i} \mapsto \overline{b}\frac{b}{i} \in I_j$.
\item[] ( $s$-back property ) For $j < m$, $\overline{a} \mapsto \overline{b} \in I_{j+1}$, $1 \leq i \leq s$, and $b \in B$, there is $a\in A$ such that $\overline{a} \frac{a}{i} \mapsto \overline{b}\frac{b}{i} \in I_j$.
 \end{enumerate}
 We then write $(I_j)_{j\leq m} : \mathcal{A} \cong_m^s \mathcal{B}$.
\end{definition}

 \begin{definition}
 Structures $\mathcal{A}$ and $\mathcal{B}$ are $s$-partially isomorphic, in symbols $\mathcal{A} \cong_{part}^s \mathcal{B}$, iff there is a nonempty set $I$ of $s$-partial isomorphisms with the forth and back properties, respectively :
 \begin{enumerate}
 \item[] For $\overline{a} \mapsto \overline{b} \in I$, $1 \leq i \leq s$, and $a \in A$, there is $b\in B$ such that $\overline{a} \frac{a}{i} \mapsto \overline{b}\frac{b}{i} \in I$.
 \item[] For $\overline{a} \mapsto \overline{b} \in I$, $1 \leq i \leq s$, and $b \in B$, there is $a\in A$ such that $\overline{a} \frac{a}{i} \mapsto \overline{b}\frac{b}{i} \in I$.
 \end{enumerate}
 We then write $I: \mathcal{A} \cong_{part}^s \mathcal{B}$.
 \end{definition}
 \begin{definition}
 For $m\in \mathbb{N}$, any structure $\mathcal{A}$, and $\overline{a} \in (A\cup\{*\}^s)$, the $s$-$m$-\emph{isomorphism type} $\psi_{\overline{a}}^m$ ( $=$ $^s\psi_{\mathcal{A},\overline{a}}^m$ ) of $\overline{a}$ in $\mathcal{A}$ is given by :$\\$
  $$\psi_{\overline{a}}^0 (\overline{v}) := \bigwedge \{\psi\; |\; \psi \;\; \text{atomic or negated atomic, and} \;\; \mathcal{A} \vDash \psi[\overline{a}] \},$$
  $$\\$$
  $$\psi_{\overline{a}}^{m+1}(\overline{v}) := \psi_{\overline{a}}^0 \wedge \underset{1\leq i \leq s}{\bigwedge} (\underset{a\in A}{\bigwedge} \exists v_i \psi^m_{\overline{a}\frac{a}{i}} \wedge \forall v_i \bigvee_{a\in A} \psi^m_{\overline{a}\frac{a}{i}}).$$
  $\\$
  In particular, $\psi_{\mathcal{A}}^m := \psi^m_{* \ldots *}$ is an $\textrm{FO}^s$-sentence of quantifier rank $m$.
 \end{definition}
 \begin{definition}
 The sets $W_m^s(\mathcal{A},\mathcal{B})$ and $W_{\infty}^s(\mathcal{A},\mathcal{B})$ of $s$-partial isomorphisms corresponding to winning positions of Duplicator in the respective games are defined as follows :
  $$ W_m^s(\mathcal{A},\mathcal{B}) := \{\overline{a} \mapsto \overline{b} \;| \;\; \text{Duplicator wins} \;\; G_m^s(\mathcal{A},\overline{a},\mathcal{B},\overline{b})\}, \text{and} $$

  $$ W_{\infty}^s(\mathcal{A},\mathcal{B}) := \{\overline{a} \mapsto \overline{b} \;| \;\; \text{Duplicator wins} \;\; G_{\infty}^s(\mathcal{A},\overline{a},\mathcal{B},\overline{b})\} $$
 \end{definition}

We write $\mathcal{A}\equiv_m^s \mathcal{B}$ to express that $\mathcal{A}$ and $\mathcal{B}$ satisfy the same $\textrm{FO}^s$-sentences of quantifier rank $\leq m$.

The following theorem shows that the logics and the games fit together.
When we write $\mathcal{A}\vDash \varphi[\overline{a}]$ for $\overline{a}\in (A\cup\{*\})^s$ we tacitly assume that the free variables of $\varphi$ have indices in $supp(\overline{a})$ (that is, $i\in supp(\overline{a})$ whenever $v_i\in free(\varphi)$).

\begin{theorem} \cite{Flum}$\newline$
Let $\mathcal{A}$ and $\mathcal{B}$ be structures, and let $\overline{a}\in (A\cup\{*\})^s$ and $\overline{b}\in (B\cup\{*\})^s$ with $supp(\overline{a}) = supp(\overline{b})$. Then the following are equivalent :
              \begin{enumerate}
                    \item[(i)] Duplicator wins $G_m^s(\overline{A},\overline{a},\overline{B},\overline{b})$.
                    \item[(ii)] $\overline{a} \mapsto \overline{b} \in W_m^s(\mathcal{A},\mathcal{B})$ and $(W_j^s(\mathcal{A},\mathcal{B}))_{j\leq m} : \mathcal{A} \cong_m^s \mathcal{B}$.
                    \item[(iii)] There is $(I_j)_{j\leq m}$ with $\overline{a} \mapsto \overline{b} \in I_m$ such that $(I_j)_{j \leq m} : \mathcal{A} \cong_m^s \mathcal{B}$.
                    \item[(iv)] $\mathcal{B} \vDash \psi_{\overline{a}}^m[\overline{b}]$.
                    \item[(v)] $\overline{a}$ satisfies in $\mathcal{A}$ the same $\textrm{FO}^s$-formulas of quantifier rank $\leq m$ as $\overline{b}$ in $\mathcal{B}$.
              \end{enumerate}
\end{theorem}

\begin{corollary}\cite{Flum}
            \begin{enumerate}
                    \item[(i)] Duplicator wins $G_m^s(\mathcal{A},\mathcal{B})$.
                    \item[(ii)] $(W_j^s(\mathcal{A},\mathcal{B}))_{j\leq m} : \mathcal{A} \cong_m^s \mathcal{B}$.
                    \item[(iii)] $\mathcal{A} \cong_m^s \mathcal{B}$.
                    \item[(iv)] $\mathcal{B} \vDash \psi_{\mathcal{A}}^m$.
                    \item[(v)]  $\mathcal{A} \equiv_m^s \mathcal{B}$.
               \end{enumerate}
\end{corollary}

From the previous theorem we deduce the following technique for proving non-expressibility in $\textrm{FO}[s,m]$ :

\begin{remark}
 When we want to prove that some property is not expressible in $\textrm{FO}[s,m]$ we may prove this by proving that there is a structure $\mathcal{A}$ satisfying the property and a structure $\mathcal{B}$ not satisfying the property such that Duplicator wins $G_m^s(\mathcal{A},\mathcal{B})$. We can also prove that there is a sentence in the logic $\textrm{FO}[k,n]$ not equivalent to any sentence in the logic $\textrm{FO}[s,m]$ by proving that there are two structures $\mathcal{A}$ and $\mathcal{B}$ such that Spoiler wins $G_n^k(\mathcal{A},\mathcal{B})$ while Duplicator wins $G_m^s(\mathcal{A},\mathcal{B})$.
\end{remark}

\section{Quantifier Rank and Number of Variables}
In this section we investigate the inclusion and strict inclusion relations among the logics $\textrm{FO}[k,n]$.

\begin{proposition} For every $n\geq 0$,
$$\textrm{FO}[1,n] \subsetneq \textrm{FO}[2,n] \subsetneq \textrm{FO}[3,n] \subsetneq \ldots$$
\end{proposition}

\begin{proof}
The formula $\varphi := \underset{1\leq i < j \leq k+1}\bigwedge (v_i \neq v_j)$ belongs to $\textrm{FO}[k+1,n]$ for every $n$ but is not equivalent to a formula in $\textrm{FO}[k,n]$ for any $n$. A formula in $\textrm{FO}[k,n]$ can have at most $k$ variables and thus, at most $k$ free variables, and so for any satisfiable formula in $\textrm{FO}[k,n]$ there is a first-order assignment satisfying it that gives the same value to $v_k$ and $v_{k+1}$. This assignment, of course, does not satisfy $\varphi$. Thus $\varphi$ is not equivalent to any satisfiable formula in $\textrm{FO}[k,n]$. Moreover, since $\varphi$ is satisfiable, it is not equivalent to any contradictory formula. Hence $\varphi$ is not equivalent to any formula in $\textrm{FO}[k,n]$. (Note that this also proves that for $k^{\prime} > k$ and $n > n^{\prime}$, $\textrm{FO}[k^{\prime},n^{\prime}] \nsubseteq \textrm{FO}[k,n]$, and in general, $\textrm{FO}^1 \subsetneq \textrm{FO}^2 \subsetneq \ldots$ )
\end{proof}

\begin{proposition}For every $n\geq 0$ and every $k\geq 3$,
$$\textrm{FO}[k,n]\subsetneq \textrm{FO}[k,n+1]$$
(This is a rephrasing of Proposition 6.15. in Immerman's "Descriptive Complexity" \cite{Immerman})
\end{proposition}
\begin{proof}
$\newline$
This proof is valid for both, structures with and without ordering. If one wants to consider the proof for unordered structures, then when we mention ordering in the proof one should understand it as the ordering of the names of the vertices, (but keep in mind that we are using that natural ordering of the names of vertices in the proof even when we are considering unordered structures).$\newline$
We will show that the property of directed graphs $\mathcal{G} = (G, E^{\mathcal{G}},s,t)$ with two distinguished elements $s,t$ saying that "the distance between $s,t$ is $2^{n+1}$" is expressible in $\textrm{FO}[3,n+1]$ but not expressible with a first-order formula of quantifier rank $\leq n$. And thus, we actually prove more than the statement of the proposition, in particular if $k > k^{\prime}\geq 3$ and $n^{\prime} > n$ then $\textrm{FO}[k^{\prime},n^{\prime}] \nsubseteq \textrm{FO}[k,n]$.

Set
$\varphi_0(x,y) := Exy$, this is of quantifier rank $0$ and it uses no variables outside $\{x,y,z\}$.$\newline$
Assuming that the formulas $\varphi_i(v,w)$ with $v,w \in \{x,y,z\}$ have been defined, and that they use variables only from $\{x,y,z\}$, and that they are of quantifier rank $i$, then $\newline$
$\varphi_{i+1}(x,y):= \exists z (\varphi_i(x,z)\wedge\varphi_i(z,y))$ is of quantifier rank $i+1$ and uses variables only from the set $\{x,y,z\}$. $\newline$
The other $\varphi_{i+1}(v,w)$ with $v,w \in \{x,y,z\}$ can be defined similarly. $\newline$
It can be easily seen, by induction, that $\varphi_{n+1}(s,t)$ expresses that the distance between $s,t$ is $2^{n+1}$.$\newline$
Now we show that Duplicator wins $G_n(\mathcal{A},\mathcal{B})$, where $\newline$ $\mathcal{A}=(\{0,\ldots,2^{n+1}+1\},\{(i,i+1)\;|\; 0\leq i \leq 2^{n+1}\},0,2^{n+1}+1)$ and, $\newline$ $\mathcal{B} = (\{0,\ldots,2^{n+1}\}, \{(i,i+1)\;|\; 0\leq i \leq 2^{n+1}-1\},0,2^{n+1})$.

Let the $d_j$-distance be defined as in Example \ref{1.2.13}.

For $0\leq j \leq n$, set $\newline$
$$I_j := \{ p \in \textrm{Part}(\mathcal{A}, \mathcal{B})\; | \; p \; \text{preserves order and}\; d_{j+1}(a,a^{\prime}) = d_{j+1}(p(a),p(a^{\prime}))$$ $$\text{for all} \; a,a^{\prime} \in do(p) \}$$
$d_{j+1}(s^{\mathcal{A}},t^{\mathcal{A}}) = \infty = d_{j+1}(s^{\mathcal{B}},t^{\mathcal{B}})\; \text{for all} \; 0\leq j \leq n$. So $\newline$ $\{(s^{\mathcal{A}},s^{\mathcal{B}}) , (t^{\mathcal{A}},t^{\mathcal{B}})\} \in I_j \; \text{for all} \; 0 \leq j \leq n$. Hence $I_j \neq \emptyset \quad$ for all $0 \leq j \leq n$.
Suppose $j < n$, $p\in I_{j+1}$, $a\in A$,
If there is $a^{\prime} \in do(p)$ such that $d_{j+1}(a,a^{\prime}) < 2^{j+1}$, then there is exactly one $b\in B$ for which $p\cup \{(a,b)\}$ is a partial isomorphism preserving order and $d_{j+1}$-distances.

If there is no such $a^{\prime}$, let $do(p) = \{a_1,\ldots,a_m\}$ with $a_1 < \ldots < a_m$. Then for some $1\leq i \leq m-1$, $a_i < a < a_{i+1}$, and $d_{j+1}(a_i,a)=\infty$, $d_{j+1}(a,a_{i+1})=\infty$, hence $d_{j+2}(a_i,a_{i+1})=\infty$ and therefore,$\newline$ $d_{j+2}(p(a_i),p(a_{i+1}))=\infty$. Thus there is a $b$ such that $p(a_i) < b < p(a_{i+1})$, $d_{j+1}(p(a_i),b) = \infty$ and $d_{j+1}(b,p(a_{i+1}))=\infty$. It is easy to see that $p\cup\{(a,b)\}$ is a partial isomorphism in $I_j$.
$\newline$
(Note that here when $j=0$ that $\newline$ "$d_{j+2}(a_i,a_{i+1})=\infty$ and $d_{j+2}(p(a_i),p(a_{i+1}))=\infty$" means that $\newline$ $d(a_i,a_{i+1}),d(p(a_i),p(a_{i+1})) \geq 4$ and so Duplicator can find a vertex between $p(a_i),p(a_{i+1})$ that is not adjacent to any one of them as $a$ is not adjacent to any of $a_i,a_{i+1}$).$\newline$
The back property can be proven similarly.
\end{proof}

\begin{proposition}\label{3.3.5}$\newline$
There is a property of words expressible in $\textrm{FO}[2,2^n +2]$ but not expressible in $\textrm{FO}[k,n]$ for any $k$.
\end{proposition}
Before we give a proof we need first to mention some definitions and facts.
Let $\Sigma$ be a finite alphabet and let $\Sigma^{\star}$ denote the set of finite words over it.

\begin{definition} (Boundary Position)\cite{strings} $\newline$
A \emph{boundary position} denotes the first or last occurrence of a letter in a given word. Boundary positions are of the form $d_a$ where $d\in \{\rhd,\lhd\}$ and $a\in \Sigma$. The interpretation of a boundary position $d_a$ on a word $w=w_1...w_{|w|}\in \Sigma^{\star}$ is defined as follows :

\[
 d_a(w)=
  \begin{cases}
   \min\{i\in[1,|w|] \;| \; w_i=a\} & \text { if }d=\rhd \\
  \max\{i\in[1,|w|] \;| \; w_i=a\} & \text { if }d=\lhd \\
  \end{cases}
\]
$\newline$
If there are no occurrences of $a$ in $w$ then we set $d_a(w)$ to be undefined. A boundary position can also be specified with respect to a position $q\in[1,|w|]$ :
\[
 d_a(w,q)=
  \begin{cases}
   \min\{i\in[q+1,|w|]\;|\;w_i=a\} & \text { if }d=\rhd \\
  \max\{i\in[1,q-1]\;|\;w_i=a\} & \text { if }d=\lhd \\
  \end{cases}
\]

\end{definition}
\begin{definition} (Rankers) \cite{strings}$\newline$
Let $n$ be a positive integer. An $n$-\emph{ranker} $r$ is a sequence of $n$ boundary positions. The interpretation of an $n$-ranker $r=(p_1,...,p_n)$ on a word $w$ is defined as follows :

\[
 r(w)=
  \begin{cases}
   p_1(w) & \text { if }r=(p_1) \\
   \text{undefined} & \text { if } (p_1,\ldots,p_{n-1})(w) \;\text{is undefined} \\
   p_n(w,(p_1,\ldots,p_{n-1})(w)) & \text{otherwise} \\
  \end{cases}
\]

Instead of writing $n$-rankers as a formal sequence $(p_1,\ldots,p_n)$, we often use the simpler notation $p_1\ldots p_n$. We denote the set of all $n$-rankers by $R_n$, and the set of all $n$-rankers that are defined over a word $w$ by $R_n(w)$.
\end{definition}
\begin{theorem}(Expressibility of the Defined-ness of a Ranker, Weis and Immerman 2007 \cite{strings}) $\newline$
Let $n$ be a positive integer, and let $r\in R_n$. There is a formula $\newline$ $\varphi_r \in \textrm{FO}[2,n]$ such that for all $w\in \Sigma^{\star}$, $w\vDash \varphi_r \Leftrightarrow r\in R_n(w)$.
\end{theorem}

Now we prove Proposition \ref{3.3.5} :
\begin{proof}
The property of "the defined-ness of the $2^n+2$-ranker $\underset{2^n+2}{\underbrace{\rhd_1\rhd_1 \ldots \rhd_1}}$" is expressible in $\textrm{FO}[2,2^n+2]$ but not expressible in $\textrm{FO}[k,n]$ for any $k$.
From the theorem we have just mentioned of Immerman and Weis , this property is expressible in $\textrm{FO}[2,2^n +2]$; but from the proof of Example \ref{1.2.13}, (of the non-expressibility of evenness of ordered structures in first-order logic), Duplicator has a winning strategy on the following two strings in the $n$-round game : $\newline$ $w_1$ and $w_2$ strings of $1$'s of lengths $2^n+2$ and $2^n+1$ respectively (so the ranker $\underset{2^n+2}{\underbrace{\rhd_1\rhd_1 \ldots \rhd_1}}$ is defined in $w_1$ but not defined in $w_2$). $\newline$
It is important to note that in Example \ref{1.2.13} the ordering vocabulary contained $max$ and $min$, but here the vocabulary of strings does not necessarily contain these constants. The presence of these constants made it possible, in the proof of the second case for the forth property in the example, to restrict ourselves to the case when $a$ is between two elements $a_i$ and $a_{i+1}$ from the domain of the partial isomorphism. To recover this we may change the definition of $I_j$ so that it assumes that the first element and the last element of $A$ are in the domain of the partial isomorphism.
\end{proof}

\section{Infinite Disjunctions and a Revisit to Depth}
In this section we present infinitary logic (logic with infinite disjunctions). All the fixed-point extensions of first-order logic that we have exhibited are contained in infinitary finite variable logic. Provoked by this, we suggest a rough relationship between depth and the number of disjunctions. Then we introduce a new complexity measure $\textrm{FO}_{\bigvee}[f(n),g(n)]$ which counts the number, $f(n)$, of $\vee$-symbols, and the number, $g(n)$, of variables, in first-order formulas needed to express a given property. We prove that for $f(n)\geq \log{n}$, $\textrm{NSPACE}[f(n)] \subseteq \textrm{FO}_{\bigvee}[f(n)+\left(\frac{f(n)}{\log{n}}\right)^2,\frac{f(n)}{\log{n}}]$, and that for any $f(n),g(n)$, $\textrm{FO}_{\bigvee}[f(n),g(n)]\subseteq \textrm{DSPACE}[f(n)g(n)\log{n}]$, and as a corollary we have : $$\underset{k\geq 1}{\bigcup} \textrm{FO}_{\bigvee}[n^k+\left(\frac{n^k}{\log{n}}\right)^2,\frac{n^k}{\log{n}}]= \textrm{FO}[2^{n^{O(1)}}].$$

\begin{definition} (The Infinitary Logic $\textrm{L}_{\infty\omega}$) $\newline$
Let $\tau$ be a vocabulary. The class of $\textrm{L}_{\infty\omega}$-formulas over $\tau$ is given by the following clauses :
\begin{enumerate}
\item[-] it contains all atomic first-order formulas over $\tau$
\item[-] if $\varphi$ is a formula then so is $\neg \varphi$
\item[-] if $\varphi$ is a formula and $x$ a variable then $\exists x \varphi$ is a formula
\item[-] if $\Psi$ is a set of formulas (possibly infinite) then $\bigvee \Psi$ is a formula.
\end{enumerate}
The semantics is a direct extension of the semantics of first-order logic with $\bigvee \Psi$ being interpreted as the disjunction over all formulas in $\Psi$; hence, neglecting the interpretation of the free variables,
$$\mathcal{A}\vDash \Psi \;\;\; \text{iff} \;\;\; \text{for some} \; \psi \in \Psi, \; \mathcal{A}\vDash \psi.$$
We set $\bigwedge \Psi := \neg \bigvee \{\neg \psi \; | \; \psi \in \Psi\}$. Then $\bigwedge \Psi$ is interpreted as the conjunction over all formulas in $\Psi$. By identifying $(\varphi\vee \psi)$ with $\bigvee \{\varphi,\psi\}$ we see that $\textrm{L}_{\infty\omega}$ is an extension of first-order logic.
\end{definition}

In finite model theory, when talking about an $\textrm{L}_{\infty\omega}$-formula we may assume that it contains only a countable number of disjunctions as the following proposition suggests :

\begin{proposition}\cite{Flum}
In the finite, every $\textrm{L}_{\infty\omega}$- formula $\varphi(\overline{x})$ is equivalent to an $\textrm{L}_{\infty\omega}$-formula $\psi(\overline{x})$ with only countably many disjunctions and whose free variables are from the free variable of $\varphi$.
\end{proposition}

\begin{proof}
By induction on the rules for $\textrm{L}_{\infty\omega}$-formulas. The translation procedure preserves the "structure" of formulas and only replaces infinitary disjunctions by countable ones. In the main step suppose that $$\varphi(\overline{x}) = \bigvee \{\varphi_i(\overline{x}) \; | \; i \in I\}$$
is an $\textrm{L}_{\infty\omega}$-formula. For each finite structure $\mathcal{C}$ with universe $\{1,2,\ldots,||\mathcal{C}||\}$ and each $\overline{c}\in C$, if there exists $i \in I$ such that $\mathcal{C}\vDash \varphi_i[\overline{c}]$, choose such an $i$. Let $I_0$ be the set of $i$'s chosen in this way. Then $I_0$ is countable and $\bigvee \{\varphi_i(\overline{x}) \; | \; i \in I\}$ and $\bigvee \{\varphi_i(\overline{x}) \; | \; i \in I_0\}$ are equivalent in the finite.
\end{proof}

\begin{definition} (Infinitary Finite Variable Logics $L^s_{\infty\omega}$) $\newline$
$L^s_{\infty\omega}$ is the fragment of $L_{\infty\omega}$ of formulas that use at most $s$ distinct variables (free and bound); and $L^{\omega}_{\infty\omega} := \underset{s\geq 1} {\bigcup} L^s_{\infty\omega}$.
\end{definition}

\begin{theorem} \cite{Flum}
Let $\mathcal{A}$ and $\mathcal{B}$ be structures, and let $\overline{a}\in (A\cup\{*\})^s$ and $\overline{b}\in (B\cup\{*\})^s$ with $supp(\overline{a}) = supp(\overline{b})$. Then the following are equivalent:
              \begin{enumerate}
                    \item[(i)] Duplicator wins $G_{\infty}^s(\overline{A},\overline{a},\overline{B},\overline{b})$.
                    \item[(ii)] $\overline{a} \mapsto \overline{b} \in W_{\infty}^s(\mathcal{A},\mathcal{B})$ and $W_{\infty}^s(\mathcal{A},\mathcal{B}): \mathcal{A} \cong_{part}^s \mathcal{B}$.
                    \item[(iii)] There is $I$ with $\overline{a} \mapsto \overline{b} \in I$ such that $I : \mathcal{A} \cong_{part}^s \mathcal{B}$.
                    \item[(iv)] $\overline{a}$ satisfies in $\mathcal{A}$ the same $\textrm{L}_{\infty\omega}^s$-formulas as $\overline{b}$ in $\mathcal{B}$.
              \end{enumerate}
\end{theorem}

\begin{corollary} \cite{Flum}
The following are equivalent :
               \begin{enumerate}
                    \item[(i)] Duplicator wins $G_{\infty}^s(\mathcal{A},\mathcal{B})$.
                    \item[(ii)] $W_{\infty}^s (\mathcal{A},\mathcal{B}): \mathcal{A} \cong_{part}^s \mathcal{B}$.
                    \item[(iii)] $\mathcal{A} \cong^s_{part} \mathcal{B}$.
                    \item[(iv)] $\mathcal{A} \equiv^{L^s_{\infty\omega}} \mathcal{B}$.
               \end{enumerate}
\end{corollary}

\begin{theorem}\cite{Flum} $\newline$
$$\textrm{FO}(\textrm{PFP}) \subseteq \textrm{L}_{\infty\omega}^{\omega}$$
\end{theorem}

For a proof we first prove the following lemma :
\begin{lemma}\cite{Flum} $\newline$
Let $\varphi(X,\overline{x})$ be a first-order formula where all variables are among $v_1,\ldots,v_k$ and $X$. Suppose $X$ is $s$-ary and $\overline{x}=x_1\ldots x_s$ (with $x_1,\ldots,x_s$ among $v_1,\ldots,v_k$). Then for every $n$, there is a formula $\varphi^n(\overline{x})$ in $\textrm{FO}^{k+s}$ defining the stage $F_n^{\varphi}$.
\end{lemma}
\begin{proof}$\newline$
Let $y_1=v_{k+1},\ldots,y_s=v_{k+s}$. Then $\varphi^n(\overline{x})$ can be defined inductively by
                              $$\varphi^0(\overline{x}) := \neg x_1=x_1,$$
                              $$\varphi^{n+1}(\overline{x}) := \varphi(X,\overline{x})\frac{\exists \overline{y} (\overline{y} = - \wedge \exists \overline{x} (\overline{x} = \overline{y} \wedge \varphi^n(\overline{x})))}{X-}$$
                              (i.e., we replace in $\varphi(X,\overline{x})$ each occurrence of an atomic subformula of the form $X\overline{t}$ by $\exists \overline{y} (\overline{y} = \overline{t} \wedge \exists \overline{x} (\overline{x} = \overline{y} \wedge \varphi^n(\overline{x})))$; note that some variables of $\overline{x}$ may occur in $\overline{t}$).
\end{proof}

Now we prove $\textrm{FO}(\textrm{PFP}) \subseteq \textrm{L}_{\infty\omega}^{\omega}$.
\begin{proof}$\newline$
By Proposition \ref{2.1.2} it suffices to show for first-order $\varphi$ that $[\textrm{PFP}_{X,\overline{x}}\varphi]\overline{t}$ is equivalent to a formula of $\textrm{L}^{\omega}_{\infty\omega}$.$\newline$
So suppose that $\varphi,k,s,X,\overline{x},\overline{y}$ are as in the preceding Lemma. Then $[\textrm{PFP}_{X,\overline{x}}\varphi]\overline{t}$ is equivalent to the $\textrm{L}_{\infty\omega}^{k+s}$-formula (we may assume without loss of generality that the variables in $\overline{t}$ are in $\{x_1,\ldots,x_s\}$)
                          $$\underset{n\geq 0}{\bigvee}(\forall \overline{x} (\varphi^n(\overline{x}) \leftrightarrow \varphi^{n+1}({\overline{x})}) \wedge \varphi^n(\overline{t}))$$
where, to stay within $\textrm{L}_{\infty\omega}^{k+s}$, we take  $\exists \overline{y} (\overline{y} = \overline{t} \wedge \exists \overline{x} (\overline{x} = \overline{y} \wedge \varphi^n(\overline{x})))$ for $\varphi^n(\overline{t})$.
\end{proof}
This, and the idea of the dependence of the depth (or the number of stages to compute the fixed point) on the size of the structure, provoke us to break this $\textrm{L}_{\infty\omega}^{\omega}$-formula into a sequence of first-order formulas - with the same upper bound on the number of variables - in which the number of disjunctions is a function of the size of the structure (Recall the dependence of the number of iterations of the quantifier block on the size of the structure). So for every $n\geq 1$, every structure $\mathcal{A}$ of size $n$ :
      $$\mathcal{A} \vDash  [\textrm{PFP}_{X,\overline{x}}\varphi]\overline{t} \leftrightarrow \overset{t(n)}{\underset{i = 0}{\bigvee}}(\forall \overline{x} (\varphi^i(\overline{x}) \leftrightarrow \varphi^{i+1}({\overline{x})}) \wedge \varphi^i(\overline{t})) $$
       where $t(n)$ is the depth of $\varphi(X,\overline{x})$.

To make this more elaborate we define the following complexity classes:

\begin{definition} \label{3.3.8} $\newline$
Let $\textrm{FO}_{\bigvee}^k[f(n)]$ denote the set of all classes of structures definable by a uniform sequence of $\textrm{FO}^k$-sentences in which the number of disjunctions is $O(f(n))$, more precisely, a class $S$ of structures is in $\textrm{FO}_{\bigvee}^k[f(n)]$ iff there is a sequence $\varphi_1,\varphi_2,\ldots$ of $\textrm{FO}^k$-sentences such that :
\begin{enumerate}
\item[(1)]For every positive integer $n$, every structure $\mathcal{A}$ of size $n$ :

              $$\mathcal{A} \in S \;\;\; \text{iff} \;\;\; \mathcal{A} \vDash \varphi_n.$$
\item[(2)]There is a constant number $c$ such that, for every positive integer $n$, the number of $\vee$-symbols in $\varphi_n$ is at most $cf(n)$ (assuming that the formulas are expressed using $\vee$'s and $\neg$'s only).
\item[(3)] The map $n \mapsto \varphi_n$ is generable by a $\textrm{DSPACE}[f(n)]$ Turing machine.
\end{enumerate}
\end{definition}
The number of $\vee$-symbols in the formula $\varphi^0(\overline{x})$ is $0$, and their number in $\varphi^1(\overline{x})$ is $l+2m$ where $l$ is the number of $\vee$-symbols and $m$ is the number of occurrences of $X$ in $\varphi(X,\overline{x})$ (Note that the formula$\newline$ $\exists \overline{y} (\overline{y} = \overline{t} \wedge \exists \overline{x} (\overline{x} = \overline{y} \wedge \varphi^0(\overline{x})))$ with which the subformulas $X\overline{t}$ are replaced contain exactly two $\wedge$'s and no other binary connectives and hence, it contains exactly two $\vee$'s when it is written using only $\vee$'s and $\neg$'s). It can be easily seen that the number of $\vee$-symbols in $\varphi^2(\overline{x})$ is $l+m(2+l+2m)$ i.e. $l+(l+2)m+2m^2$, and in $\varphi^3(\overline{x})$ is $l+m(2+l+(l+2)m+2m^2)$ i.e. $l+(l+2)m+(l+2)m^2+2m^3$, and then by induction the number of $\vee$-symbols in $\varphi^i(\overline{x})$ is $\newline$ $l+(l+2)m+(l+2)m^2+\ldots +(l+2)m^{i-1}+2m^i$ for any $i\geq 4$.
 $\newline\newline$
 Let $h(i)$ denote the number of $\vee$-symbols in $\varphi^i(\overline{x})$. Then the number of $\vee$-symbols in the formula $$\varphi_n :=\overset{t(n)}{\underset{i = 0}{\bigvee}}(\forall \overline{x} (\varphi^i(\overline{x}) \leftrightarrow \varphi^{i+1}({\overline{x})}) \wedge \varphi^i(\overline{t}))$$ is $\overset{t(n)}{\underset{i=0}{\Sigma}}(4+2h(i)+h(i+1))$ i.e. $\newline$

  $ (4+l+(3l+4)t(n))+(8+l+(3l+6)(t(n)-1)) m + \newline (8+l+(3l+6)(t(n)-2)) m^2 + (8+l+(3l+6)(t(n)-3)) m^3 + \ldots + \newline (8+l+(3l+6)) m^{t(n)-1} + (8+l) m^{t(n)}+2 m^{t(n)+1}$,

  $\newline$
  Call this $f(n)$.
  (Note that $f(n)= \Theta(2^{ct(n)})$ for some constant number $c$ if $m\geq 2$). It is not hard to see from the construction of the formulas $\varphi_n$ that their generation does not need space more than $O(f(n))$, thus, the class of structures definable by $[\textrm{PFP}_{X,\overline{x}}\varphi]\overline{t}$ (under some interpretation of $\overline{t}$) is in $\textrm{FO}_{\bigvee}^{k+s}[f(n)]$. Hence $\textrm{FO}(\textrm{PFP})\subseteq \underset{j,k\geq 1}{\bigcup}FO_{\bigvee}^k[2^{O(2^{n^j})}]$ (Recall that if the relation variable through which the fixed-point is built is of arity $j$ then it takes at most $2^{n^j}$ stages to  come to the fixed-point if we do not restrict the formula to be positive in the relation variable, so the depth $t(n)$ of any formula is $O(2^{n^j})$ for some $j$).
$\newline$
 We think that the number of $\vee$-symbols is worth studying as a complexity measure. Also it is important to investigate the relationship between $\textrm{IND}[t(n)]$ (or $\textrm{FO}[t(n)]$) and $\textrm{FO}_{\bigvee}^{O(1)}[2^{O(t(n))}]$. $\newline$
Likewise we can also define the complexity classes $\textrm{FO}_{\bigvee}[f(n),g(n)]$ in which the number of variables depends on the size of the structure :

\begin{definition}
Fix a vocabulary $\sigma$. Assume that formulas are expressed using $\vee$'s and $\neg$'s only. We say that a class, $S$, of $\sigma$-structures is in $\textrm{FO}_{\bigvee}[f(n),g(n)]$ if there exists a sequence of sentences $\{\varphi_i \; | \; i= 1,2,\ldots \}$ from $\textrm{FO}[\sigma]$, and constant numbers, $k,l$, such that :
\begin{enumerate}
\item[(1)] For all $\sigma$-structures, $\mathcal{A}$, if $||\mathcal{A}||=n$, then :
                                      $$\mathcal{A} \in S \;\;\;\;\; \text{iff} \;\;\;\;\; \mathcal{A}\vDash \varphi_n.$$
\item[(2)] $\varphi_n$ has $\leq kf(n)$ $\vee$-symbols and uses $\leq l g(n)$ distinct variables (free and bound).
\item[(3)] The map $n \mapsto \varphi_n$ is generable by a $\textrm{DSPACE}[f(n)+g(n)]$ Turing machine.
\end{enumerate}
\end{definition}

 \begin{theorem} \label{3.3.10}
 For $f(n) \geq \log{n}$, (under Proviso \ref{2.2}), $$\textrm{NSPACE}[f(n)] \subseteq \textrm{FO}_{\bigvee}[f(n)+\left(\frac{f(n)}{\log{n}}\right)^2,\frac{f(n)}{\log{n}}]$$ In particular, $$\textrm{NL} \subseteq \textrm{FO}_{\bigvee}[\log{n},O(1)]$$
 \end{theorem}
 \begin{proof} (In this proof when we write an expression of functions, for example $\log{n}$ or $\frac{f(n)}{\log{n}}$, we mean its ceiling).
 Let $S\subseteq \textrm{STRUC}[\sigma]$ be a boolean query in $\textrm{NSPACE}[f(n)]$. Let $M$ be a nondeterministic $f(n)$-space machine that accepts $S$, i.e.,
 $$\mathcal{A}\in S  \;\;\;\;\; \text{iff} \;\;\;\;\; M(bin(\mathcal{A}))\downarrow.$$
  Let $m$ be such that $M$ uses at most $mf(n)$ bits of work-tape for inputs of size $n$. Let $\sigma = \{R_1,\ldots,R_r,c_1,\ldots,c_l\}$ where each $R_i$ is of arity $a_i$ and let $a=max\{a_i \; | \; 1\leq i \leq r\}$. Let $h(n)=m\left(\frac{f(n)}{\log{n}}\right)$, $t(n)= \frac{\log{(h(n))}}{\log{n}}$, and $g(n)=4+a+h(n)+t(n)$. Let $\mathcal{A}$ be a $\sigma$-structure of size $n$. A configuration, or an instantaneous description, of $M$'s computation on $\mathcal{A}$, can be coded as a $g(n)$-tuple of variables :
  $$(q,w_1,\ldots,w_{h(n)},s,r_1,\ldots,r_a,p,v_1,\ldots,v_{t(n)},v)$$
  The variable $q$ encodes the state of the machine.The variables $w_1,\ldots,w_{h(n)}$ encode the contents of $M$'s work-tape. Remember that each variable represents an element of $\mathcal{A}$'s $n$-element universe, so it corresponds to a $\log{n}$-bit number. The variable $s$ determines at which relation, constant, or otherwise, the input-head is looking. The variables $r_1,\ldots,r_a$ encode where in one of the input relations the input-head is looking, if it is looking at a relation. The variable $p$ determines the number of the bit being read in one of the input constants if the input-head is looking at a constant. The variables $v_1,\ldots,v_{t(n)}$ encode the index of the variable where the work-head is looking. (Note that $t(n)$ variables together can name $n^{t(n)}$, i.e. $h(n)$, things). The variable $v$ determines the number of the bit being read in a variable of $w_1,\ldots,w_{h(n)}$.
   $\newline$ We may assume without loss of generality that $(0,0,\ldots,0)$ is the unique start configuration of $M$ and $(max,max,\ldots,max)$ is its unique accept configuration. We denote the configuration $(0,\ldots,0)$ by $\overline{0}$ and $\newline(max,\ldots,max)$ by $\overline{max}$. $\newline$
 Thus the size of the configuration graph for an input structure of size $n$ is at most $n^{g(n)}$ ("at most" because not all tuples necessarily represent configurations). $\mathcal{A}\in S$ if and only if $M(bin(\mathcal{A}))\downarrow$ if and only if there is a path from $\overline{0}$ to $\overline{max}$ in the configuration graph of $M$'s computation on $\mathcal{A}$, i.e., if and only if the configuration graph satisfies the sentence $\varphi_{g(n) \log{n}}[\overline{0},\overline{max}]$ where the formulas $\varphi_i(x,y)$ are defined inductively as follows : $\newline$
  $\varphi_0(x,y) := E xy$ and, $\newline\varphi_{i+1}(x,y) := \exists z \forall u \forall v ( ((u=x \wedge v=z)\vee(u=z \wedge v=y)) \rightarrow \varphi_i(u,v))$. It can be easily proved by induction that the number of $\vee$-symbols in $\varphi_i$ is $4 i$ and that it expresses the existence of a path from $x$ to $y$ of length $\leq 2^i$. Hence, in particular, $\varphi_{g(n) \log{n} }[\overline{0},\overline{max}]$ has $4 g(n) \log{n} $ ($=\Theta(f(n))$ as a function of $n$) $\vee$-symbols and expresses the existence of a path from $\overline{0}$ to $\overline{max}$ of length $\leq 2^{g(n) \log{n} }$, i.e. $n^{g(n)}$, in the configuration graph.$\newline$ Since the configuration graph is of size at most $n^{g(n)}$, then $\varphi_{g(n) \log{n} }[\overline{0},\overline{max}]$ expresses the existence of a path in general from $\overline{0}$ to $\overline{max}$ in it. It remains now to retrieve from $\varphi_{g(n) \log{n} }[\overline{0},\overline{max}]$ the corresponding sentence in the vocabulary of $\mathcal{A}$ and show that the number of $\vee$-symbols in it is $O(f(n)+\left(\frac{f(n)}{\log{n}}\right)^2)$ and the number of variables is $O(\frac{f(n)}{\log{n}})$. (Note that any $\varphi_i$ contains only $5$ variables $x,y,z,u,v$). We should show that this sentence depends only on the size of $\mathcal{A}$ but not on $\mathcal{A}$ itself. Then we should show that the sequence of sentences we have is generable by a deterministic $f(n)+\left(\frac{f(n)}{\log{n}}\right)^2$-space machine. $\newline$ This sentence is obtained from $\varphi_{g(n) \log{n} }[\overline{0},\overline{max}]$ by replacing the unique occurrence of an atomic subformula of the relation $E$ with the formula $\psi_E(q,\overline{w},s,\overline{r},p,\overline{v},v,q^{\prime},\overline{w^{\prime}},s^{\prime},\overline{r^{\prime}},\overline{v^{\prime}},v^{\prime})$ expressing that there is an edge from $(q,\overline{w},s,\overline{r},p,\overline{v},v)$ to $(q^{\prime},\overline{w^{\prime}},s^{\prime},\overline{r^{\prime}},\overline{v^{\prime}},v^{\prime})$ in the configuration graph, and replacing each variable by a $g(n)$-tuple of variables. We do not need to restrict quantifiers so that their variables represent configurations because the unique occurrence of $\psi_E$ in the core of the sentence (note that $E$ occurs only once in the core of $\varphi_i$) is already restricting the vertices on the path claimed by $\varphi_{g(n) \log{n} }[\overline{0},\overline{max}]$ to represent configurations. This occurrence of $\psi_E$ ensure that the transition function of the machine can make it go in one step from configuration $\overline{0}$ to the situation represented by the next vertex on the path, and hence this vertex represents a configuration, and so on up to the end of the path.$\newline$
  There is an edge from $(q,\overline{w},s,\overline{r},p,\overline{v},v)$ to $(q^{\prime},\overline{w^{\prime}},s^{\prime},\overline{r^{\prime}},\overline{v^{\prime}},v^{\prime})$ in the configuration graph of $M$'s computation on $\mathcal{A}$ if and only if $M$ can go in one step from the configuration represented by $(q,\overline{w},s,\overline{r},p,\overline{v},v)$ to the configuration represented by $(q^{\prime},\overline{w^{\prime}},s^{\prime},\overline{r^{\prime}},\overline{v^{\prime}},v^{\prime})$. $\newline\psi_E$ is a disjunction over $M$'s finite transition table. A typical entry in the transition table is of the form $((x,b,w),(x^{\prime},i_d,w^{\prime},w_d))$. This says that in state $x$, looking at a bit that equals $b$ with the input-head and a bit that equals $w$ with the work-head, $M$ may go to state $x^{\prime}$, move its input-head one step in direction $i_d$, write $w^{\prime}$ in the cell where it is looking in the work-tape and move its work-head one step in direction $w_d$. Let us describe the disjunct corresponding to $((x,b,w),(x^{\prime},i_d,w^{\prime},w_d))$ if $b=1$, $w=0$, $w^{\prime}=1$, $i_d$ is right, and $w_d$ is left. This disjunct is the conjunction of :
   \begin{enumerate}
   \item[(1)] $\psi_1$ : A formula saying that $q$ equals $x$ and $q^{\prime}$ equals $x^{\prime}$.
   \item[(2)] $\psi_2$ : A formula saying that the input-head was reading a $1$ in some cell and is now looking at the next cell.
   \item[(3)] $\psi_3$ : A formula saying that the work-head was reading $0$ in some cell and now this cell contains $1$ and the contents of all the other cells are unchanged, and the work-head is now looking at the previous cell.
   \end{enumerate}
   $\psi_1$ is the conjunction of the formulas $\textrm{BIT}(q,i)$ for the bits $i$ which are $1$ in $x$, and the formulas $\neg \textrm{BIT}(q,i)$ for the bits $i$ which are $0$ in $x$, and the formulas $\textrm{BIT}(q^{\prime},i)$ for the bits $i$ which are $1$ in $x^{\prime}$, and the formulas $\neg \textrm{BIT}(q^{\prime},i)$ for the bits $i$ which are $0$ in $x^{\prime}$. Since the number of states of the machine is finite, $t$ say, then any state can be coded in $\log{t}$ bits, and hence this formula contains $2 \log{t}$ conjunctions ($\log{t}$ for $q$ and $\log{t}$ for $q^{\prime}$). Hence this part of $\psi_E$ contributes a constant number of $\vee$-symbols that is independent of the size of the structure. $\newline$
   $\psi_2$ is a finite disjunction (independent of the size of the structure), since the number of relation symbols and constant symbols is finite, considering the different cases for the position of the input head:
   \begin{enumerate}
   \item[(a)] It was looking at some relation at a tuple before the last and is now looking at the next tuple of the same relation. This is determined by the values of $s,s^{\prime},r_1,\ldots,r_a$, and $r_1^{\prime},\ldots,r_a^{\prime}$. (If it was looking at relation $R_i$ then that it was reading a $1$ is expressed by $R_i r_1\ldots r_{a_i}$).
   \item[(b)] It was looking at the last tuple of a relation and is now looking at the first tuple of the next relation or at the first bit of $c_1$. This is also determined by the values of $s,s^{\prime},r_1,\ldots,r_a$, and $r_1^{\prime},\ldots,r_a^{\prime}$, plus possibly $p$.
   \item[(c)] It was looking at a constant at a bit before the last and is now looking at the next bit of the same constant, or it was looking at the last bit of a constant and is now looking at the first bit of the next constant, or is now outside the range of bits of the input from the right. This is determined by the values of $s,s^{\prime},p,$ and $p^{\prime}$. (If it was looking at the constant $c_i$ then that it was reading a $1$ is expressed by $\textrm{BIT}(c_i,p)$).
   \end{enumerate}
   $\psi_3$ is a disjunction depending on the size of the structure : $\newline$
   $$\underset{i=1}{\overset{h(n)}{\bigvee}} (\neg \textrm{BIT}(w_i,v) \wedge \textrm{BIT}(w_i^{\prime},v) \wedge v^{\prime} = v-1 \wedge (\underset{j;j\neq i}{\bigwedge} w_j =w_j^{\prime}) \wedge$$ $$\forall x (x \neq v \rightarrow (\textrm{BIT}(w_i,x) \leftrightarrow \textrm{BIT}(w_i^{\prime},x))) \wedge (\overset{t(n)}{\underset{j=1}{\bigwedge}} v_j=v^{\prime}_j)) $$
   This is the disjunction for the case when the work-head is looking in a variable of $w_1,\ldots,w_{h(n)}$ at a bit other than the first bit of it, the other cases can be treated similarly. This contributes a number of $\Theta((h(n))^2)$ $\vee$-symbols to $\psi_E$.
 $\newline$  Noting that $E$ occurs only once in any of the formulas $\varphi_i$, then the number of $\vee$-symbols in the $n$-th sentence is $\Theta(g(n)\log{n}+(h(n))^2)$, i.e., $\newline\Theta(f(n)+\left(\frac{f(n)}{\log{n}}\right)^2)$. The variables used in the formula $\psi_E$ are $x,q,w_1,\ldots,\newline w_{h(n)},s,r_1,\ldots,r_a,p,v_1,\ldots,v_{t(n)},v, q^{\prime},w_1^{\prime},\ldots, w_{h(n)}^{\prime},s^{\prime},r_1^{\prime},
        \ldots,r_a^{\prime},p^{\prime},\newline v_1^{\prime},\ldots,v_{t(n)}^{\prime},v^{\prime}$, i.e., $\Theta(\frac{f(n)}{\log{n}})$ variable, and the formulas $\varphi_i$ use only a finite number of variables, five, and $E$ occurs only once in it, hence the number of variables in the resulting $n$-th sentence is $\Theta(\frac{f(n)}{\log{n}})$.$\newline$ It can be easily seen that this $n$-th formula depends only on the size of $\mathcal{A}$ and not on $\mathcal{A}$ itself.
        Again from the construction of the formulas $\varphi_i$ and the fact that $E$ occurs once in any of them, it can be easily seen that it takes $\Theta(f(n)+\left(\frac{f(n)}{\log{n}}\right)^2)$-space to generate the $n$-th sentence.
 \end{proof}

 \begin{theorem}\label{3.3.11} For any functions $f(n),g(n)$, (under Proviso \ref{2.2}),
 $$\textrm{FO}_{\bigvee}[f(n),g(n)]\subseteq \textrm{DSPACE}[f(n)g(n)\log{n}]$$
 \end{theorem}
 For a proof, we need first to prove the following lemma :

 \begin{lemma}
 Every first-order formula with $k$ $\vee$-symbols and using $m$ variables is equivalent to a first-order formula in which the number of quantifiers plus the number of free variables is at most $m(k+1)$.
 \end{lemma}
 \begin{proof}
 By induction on $k$. When the number of $\vee$-symbols is zero, the formula is something of the form : $\newline$
     $Q_{i_1}x_{i_1}\ldots Q_{i_l}x_{i_l} Rx_1\ldots x_m$, or $Q_{i_1}x_{i_1}\ldots Q_{i_l}x_{i_l} \neg Rx_1\ldots x_m$. In each one of these formulas the number of quantifiers plus the number of free variables is $m$. $\newline$
      For some $k\geq 0$, suppose the statement is true for $\textrm{FO}^m$-fromulas with $\leq k$ $\vee$-symbols, and consider an $\textrm{FO}^m$-formula with $k+1$ $\vee$-symbols. Such a formula is equivalent to a formula $\theta$ of the form $Q_{i_1}x_{i_1}\ldots Q_{i_l} x_{i_l} (\neg)(\varphi \vee \psi)$, where $(\neg)$ means that there may be or may be not a negation, and where $\varphi$ is a formula with $r$ $\vee$-symbols and $r^{\prime}$ free variables, and $\psi$ is a formula with $s$ $\vee$-symbols and $s^{\prime}$ free variables, and $r+s=k$. By the induction hypothesis, the number of quantifiers plus the number of free variables in $\varphi$ is $\leq m(r+1)$, and the number of quantifiers plus the number of free variables in $\psi$ is $\leq m(s+1)$. Thus, the number of quantifiers in $\varphi$ is $\leq m(r+1)-r^{\prime}$ and the number of quantifiers in $\psi$ is $\leq m(s+1)-s^{\prime}$. Hence the number of quantifiers in $\theta$ is $\leq l+m(r+1)-r^{\prime}+m(s+1)-s^{\prime}$. The number of free variables in $\theta$ is $\leq r^{\prime}+s^{\prime}-l$. Thus, the number of quantifiers plus the number of free variables in $\theta$ is at most $l+m(r+1)-r^{\prime}+m(s+1)-s^{\prime}+r^{\prime}+s^{\prime}-l$, i.e., $m (r+s+2)$, i.e., $m(k+2)$.
 \end{proof}

 Now we prove Theorem \ref{3.3.11}:
 \begin{proof}
  let $S$ be in $\textrm{FO}_{\bigvee}[f(n),g(n)]$, then $S$ is definable by a sequence of sentences $\varphi_1,\varphi_2,\ldots$ generable by a $\textrm{DSPACE}[f(n)+g(n)]$ machine such that $\varphi_n$ has $O(f(n))$ $\vee$-symbols and uses $O(g(n))$ variables. Hence, by the previous lemma, $\varphi_n$ has $O(f(n)g(n))$ quantifiers, i.e. $S$ is in $\textrm{QN}[f(n)g(n)]$ ($\textrm{QN}$ denotes Number of Quantifiers, cf. Immerman's paper "Number of Quantifiers is Better than Number of Tape Cells" \cite{NumQuan}). The proof of the second inclusion of Theorem 2 in that paper implies that $S$ is in $\textrm{DSPACE}[f(n)g(n)\log{n}]$.
 \end{proof}

 As a direct corollary of these two theorems, \ref{3.3.10} and \ref{3.3.11}, we have,
 \begin{corollary}\label{3.3.13}
  $$\underset{k\geq 1}{\bigcup} \textrm{FO}_{\bigvee}[n^k+\left(\frac{n^k}{\log{n}}\right)^2,\frac{n^k}{\log{n}}] = \textrm{FO}[2^{n^{O(1)}}] = \textrm{PSPACE} = \textrm{FO}(\textrm{PFP})$$
 \end{corollary}
 The last two equalities were proved by Immerman and Vardi (Theorem \ref{2.2.11}).

\section{More on Pebble Games}
In this section we talk more about pebble games and depth. We ask a question about the depth of the formulas $\varphi_{s,\nsim}$ whose fixed-points in a structure $\mathcal{A}$ are the sets of winning positions $(\overline{a},\overline{b})$ for Spoiler in the games $G_{\infty}^s(\mathcal{A},\overline{a},\mathcal{A},\overline{b})$. We conclude the section with a conjecture relating the existence of an upper bound (a function) on the depths (in the structures of a class of finite structures) of positive first-order formulas to the bounded-ness with the same bound, over the same class, of another characteristic of finite structures called the $s$-rank, and relating the existence of both bounds to the expressibility of infinitary finite variable formulas in $\textrm{IND}[t(n)]$, where $t(n)$ is the claimed bound.$\newline\newline$
Obviously, for any structures $\mathcal{A}$ and $\mathcal{B}$,
$$W_0^s(\mathcal{A},\mathcal{B}) \supseteq W_1^s(\mathcal{A},\mathcal{B}) \supseteq \ldots$$
and since there are at most $(||\mathcal{A}||+1)^s.(||\mathcal{B}||+1)^s$ $s$-partial isomorphisms from $\mathcal{A}$ to $\mathcal{B}$, then there is an $m\leq (||\mathcal{A}||+1)^s.(||\mathcal{B}||+1)^s$ such that $W_m^s(\mathcal{A},\mathcal{B})=W_{m+1}^s(\mathcal{A},\mathcal{B})$.$\newline$
For $\mathcal{A}=\mathcal{B}$ the minimum such $m$ has a name :
\begin{definition} $\newline$
The minimum $m$ such that $W_m^s(\mathcal{A},\mathcal{A})=W_{m+1}^s(\mathcal{A},\mathcal{A})$
is called the $s$-\emph{rank}  of $\mathcal{A}$, and is denoted by $r(s,\mathcal{A})$, or $r(\mathcal{A})$ for short.
\end{definition}

\begin{definition} $\newline$
Let $K$ be a class of structures. We say that $K$ is $s$-\emph{bounded} if the set $\{r(\mathcal{A})\; |\; \mathcal{A} \in K\}$ of $s$-ranks of structures in $K$ is bounded. The class $K$ is \emph{bounded} if it is $s$-bounded for every $s\geq 1$.
\end{definition}

Fix $s$, and let $$\overline{a} \sim \overline{b} \;\;\;\;\; \text{iff} \;\;\;\;\; \overline{a} \; \text{and} \; \overline{b} \; \text{satisfy the same} \; \textrm{L}_{\infty\omega}^s\text{-formulas in}\; \mathcal{A}$$
 We show that $\sim$ is definable in $\textrm{FO}(\textrm{LFP})$. Let $\overline{a}$ and $\overline{b}$ range over $A^s$. For $j\geq 0$ define $\sim_j$ on $A^s$ by induction :
$$\overline{a} \sim_0 \overline{b} \;\;\;\;\; \text{iff} \;\;\;\;\; \overline{a} \; \text{and} \; \overline{b} \; \text{satisfy the same atomic formulas in}\; \mathcal{A}$$
$$\overline{a} \sim_{j+1} \overline{b} \;\;\;\;\; \text{iff} \;\;\;\;\; \overline{a}\sim_0 \overline{b} \; \text{and for all} \; i=1,\ldots,s \; \text{and all} \; a\in A\;(b\in A)$$ $$\text{there is} \; b\in A\; (a\in A) \; \text{such that} \; \overline{a}\frac{a}{i} \sim_j \overline{b}\frac{b}{i}$$

Obviously, $\overline{a} \sim_j \overline{b}$ iff Duplicator has a winning strategy in the pebble game $G^s_j(\mathcal{A},\overline{a},\mathcal{A},\overline{b})$ with $s$ pebbles and $j$ moves. Clearly, $\sim_0\supseteq \sim_1 \supseteq \ldots$, so that $\sim_l=\sim_{l+1}$ for some $l$ ($l$ is the $s$-rank $r(\mathcal{A})$). For such an $l$ we have $\sim_l=\sim$. For the complements $\nsim_j$ of $\sim_j$ we have $$\nsim_0\subseteq \nsim_1 \subseteq \ldots$$ They are the stages $F_1^{\varphi_{s,\nsim}},F_2^{\varphi_{s,\nsim}},\ldots$, where $\varphi_{s,\nsim}(Z,x_1,\ldots,x_s,y_1,\ldots,y_s)$ is the following formula positive in $Z$ :
$$\underset{\underset{\psi \; \text{atomic}}{\psi\in \textrm{FO}^s}}{\bigvee}(\psi(\overline{x}) \leftrightarrow \neg \psi(\overline{y})) \vee \underset{1\leq i \leq s}{\bigvee} (\exists x_i \forall y_i Z\overline{x} \overline{y} \vee \exists y_i \forall x_i Z\overline{x} \overline{y}).$$ $\newline$
Then the $\textrm{FO}(\textrm{LFP})$ formula $\neg [\textrm{LFP}_{Z,\overline{x},\overline{y}}\varphi_{s,\nsim}(Z,\overline{x},\overline{y})]\overline{x}\overline{y}$ expresses $\overline{x} \sim \overline{y}$.

\textbf{Open Question:} The fixed-point of the formula $\newline$ $\small{\varphi_{s,\nsim}(x_1,\ldots,x_s,y_1,\ldots,y_s,Z)}$ is in $\textrm{IND}[n^{2s}]$; is this fixed-point expressible in $\textrm{IND}[n^r]$ for some $r <2s$ ? $\newline$
If not then we have the strict hierarchy $$\textrm{IND}[n^2]\subsetneq \textrm{IND}[n^4] \subsetneq \textrm{IND}[n^6] \subsetneq \ldots$$
\textbf{Open Question :} Is there a structure $\mathcal{A}$ for which $r(s,\mathcal{A})$ is $(||\mathcal{A}||+1)^{2s}$?

\begin{definition}\cite{Flum} (Scott Formulas) $\newline$
For given $\overline{a}$, the formula $$\sigma_{\overline{a}} := \psi^{r(\mathcal{A})}_{\overline{a}} \wedge \underset{\overline{b}\in (A\cup\{*\})^s}{\bigwedge} \forall v_1 \ldots \forall v_s (\psi^{r(\mathcal{A})}_{\overline{b}} \rightarrow \psi^{r(\mathcal{A})+1}_{\overline{b}})$$
(more exactly, $\sigma_{\overline{a}} =$ $^s\sigma_{\mathcal{A},\overline{a}}$) is called the $s$-\emph{Scott formula} of $\overline{a}$ in $\mathcal{A}$. It is an $\textrm{FO}^s$-formula of quantifier rank $r(\mathcal{A})+1+s$. In particular, $\sigma_{\mathcal{A}} := \sigma_{*...*}$ is an $\textrm{FO}^s$-sentence.
\end{definition}
These formulas capture the whole $\textrm{L}^s_{\infty\omega}$-theory of $\mathcal{A}$ :

\begin{theorem}\cite{Flum}$\newline$
Let $\mathcal{A}$ be a structure.
\begin{enumerate}
\item[(a)] For any structure $\mathcal{B}$,
                       $$\mathcal{B} \vDash \sigma_{\mathcal{A}} \;\;\;\;\; \text{iff} \;\;\;\;\; \mathcal{A}\equiv^{\textrm{L}^s_{\infty\omega}} \mathcal{B}.$$
\item[(b)] For $\overline{a}\in (A\cup\{*\})^s$, any structure $\mathcal{B}$ and $\overline{b}\in (B\cup\{*\})^s$ with $supp(\overline{a})=supp(\overline{b})$,
            $$\mathcal{B}\vDash \sigma_{\overline{a}}[\overline{b}] \;\; \text{iff} \;\; \overline{a} \; \text{satisfies in} \; \mathcal{A} \; \text{the same}\; \textrm{L}^s_{\infty\omega}\text{-formulas as} \; \overline{b} \; \text{in} \; \mathcal{B}. $$
\end{enumerate}
\end{theorem}

\begin{corollary}\cite{Flum} $\newline$
Each $L_{\infty\omega}^s$-formula $\varphi$ is equivalent to a countable disjunction of $\textrm{FO}^s$-formulas. In fact, $\varphi$ is equivalent to the $\textrm{L}_{\infty\omega}^s$-formula $\newline$ $\bigvee\{\sigma_{\overline{a}}\; | \;\mathcal{A} \; \text{is a structure},\; \overline{a}\in A^s,\; \mathcal{A}\vDash \varphi[\overline{a}]\}$. Moreover, if $K$ is any class of structures, then $\varphi$ and  $\bigvee\{\sigma_{\overline{a}}\; | \;\mathcal{A}\in K ,\; \overline{a}\in A,\; \mathcal{A}\vDash \varphi[\overline{a}]\}$ are equivalent in all structures of $K$.
\end{corollary}
As an application of the Scott formulas we present a condition for $\textrm{L}^s_{\infty\omega}$ and $\textrm{FO}^s$ to coincide in expressive power, but first we need the following lemma :
\begin{lemma} \cite{Flum}
For structures $\mathcal{A}$ and $\mathcal{B}$, if $W_j^s(\mathcal{A},\mathcal{A})=W_{j+1}^s(\mathcal{A},\mathcal{A})$ and $\mathcal{A} \equiv^{\textrm{L}_{\infty\omega}^s} \mathcal{B}$, then $W_j^s(\mathcal{B},\mathcal{B})=W_{j+1}^s(\mathcal{B},\mathcal{B})$. Hence : if $\mathcal{A} \equiv^{\textrm{L}_{\infty\omega}^s} \mathcal{B}$ then $r(\mathcal{A})=r(\mathcal{B})$.
\end{lemma}
\begin{proof}  Suppose that $W_j^s(\mathcal{A},\mathcal{A})=W_{j+1}^s(\mathcal{A},\mathcal{A})$ and $\mathcal{A} \equiv^{\textrm{L}_{\infty\omega}^s} \mathcal{B}$. $\newline$
Put $\varphi := \psi_{\mathcal{A}}^j \wedge \underset{\overline{a}\in(A\cup\{*\})^s}{\bigwedge} \forall v_1 \ldots \forall v_s (\psi^j_{\overline{a}} \rightarrow \psi^{j+1}_{\overline{a}})$. $\newline$ Since $W_j^s(\mathcal{A},\mathcal{A})=W_{j+1}^s(\mathcal{A},\mathcal{A})$, then $\mathcal{A}\vDash \varphi$, and since $\mathcal{A} \equiv^{\textrm{L}_{\infty\omega}^s} \mathcal{B}$, then also $\mathcal{B}\vDash \varphi$.
$\mathcal{B} \vDash \psi^j_{\mathcal{A}}$ implies : $$*\ldots* \mapsto *\ldots* \in W_j^s(\mathcal{A},\mathcal{B})\;\;\;\;\; (1)$$
and $\mathcal{B} \vDash \underset{\overline{a}\in(A\cup\{*\})^s}{\bigwedge} \forall v_1 \ldots \forall v_s (\psi^j_{\overline{a}} \rightarrow \psi^{j+1}_{\overline{a}})$ implies :
 $$W_j^s(\mathcal{A},\mathcal{B}) = W_{j+1}^s(\mathcal{A},\mathcal{B}) \;\;\;\;\;\;\;\;\;\;\;\; (2)$$
From $(1)$ and $(2)$ it follows that $$W_j^s(\mathcal{A},\mathcal{B}) : \mathcal{A} \cong^s_{part} \mathcal{B} \;\;\;\;\;\;\;\;\;\;\;\;\;\;\;\; (3)$$
From $(1)$ and $(3)$ it follows that for any $\overline{b} \in (B\cup\{*\})^s$ there is $\newline\overline{a} \in (A\cup \{*\})^s$ such that $\overline{a}\mapsto \overline{b} \in W^s_j(\mathcal{A},\mathcal{B})$.$\newline$
Let $\overline{b}\mapsto\overline{b^{\prime}} \in W_j^s(\mathcal{B},\mathcal{B})$, then $\overline{b}$ and $\overline{b^{\prime}}$ satisfy in $\mathcal{B}$ the same formulas of $\textrm{FO}^s$ of quantifier rank $\leq j$. There are $\overline{a}$ and $\overline{a^{\prime}}$ such that $\overline{a} \mapsto \overline{b}$ and $\overline{a^{\prime}} \mapsto \overline{b^{\prime}}$ are in $W_j^s(\mathcal{A},\mathcal{B})$. Then $\overline{a}$ and $\overline{a^{\prime}}$ satisfy  in $\mathcal{A}$ the same formulas of $\textrm{FO}^s$ of quantifier rank $\leq  j$. Thus $\overline{a} \mapsto \overline{a^{\prime}}$ is in $W_j^s(\mathcal{A},\mathcal{A})$ and hence in $W_{j+1}^s(\mathcal{A},\mathcal{A})$, so $\overline{a}$ and $\overline{a^{\prime}}$ satisfy  in $\mathcal{A}$ the same formulas of $\textrm{FO}^s$ of quantifier rank $\leq {j+1}$. Since $W_j^s(\mathcal{A},\mathcal{B}) = W_{j+1}^s(\mathcal{A},\mathcal{B})$ then $\overline{b}$ satisfies in $\mathcal{B}$ the same $\textrm{FO}^s$-formulas of quantifier rank $\leq j+1$ as $\overline{a}$ in $\mathcal{A}$, and $\overline{b^{\prime}}$ satisfies in $\mathcal{B}$ the same $\textrm{FO}^s$-formulas of quantifier rank $\leq j+1$ as $\overline{a^{\prime}}$ in $\mathcal{A}$. Hence $\overline{b}$ and $\overline{b^{\prime}}$ satisfy in $\mathcal{B}$ the same $\textrm{FO}^s$-formulas of quantifier rank $\leq j+1$. Thus $\overline{b} \mapsto \overline{b^{\prime}} \in W^s_{j+1}(\mathcal{B},\mathcal{B})$.
\end{proof}

\begin{theorem}\cite{Flum}\label{3.3.12}
\begin{enumerate}
\item[(a)] For $s\geq 1$ the following are equivalent:
            \begin{enumerate}
                   \item[(i)] $K$ is $s$-bounded.
                   \item[(ii)] On $K$, every $\textrm{L}_{\infty\omega}^s$-formula is equivalent to an $\textrm{FO}^s$-formula.
                   \item[(iii)] On $K$, every $\textrm{L}_{\infty\omega}^s$-formula is equivalent to an $\textrm{FO}$-formula.
            \end{enumerate}
\item[(b)] $K$ is bounded iff $\textrm{FO}$ and $\textrm{L}_{\infty\omega}^{\omega}$ have the same expressive power on $K$.
\end{enumerate}
\end{theorem}
\begin{proof} $\newline$
As (b) is a consequence of (a), it suffices to prove (a). First suppose that $K$ is $s$-bounded and set $m := sup \{r(s,\mathcal{A})\; | \; \mathcal{A} \in K\} < \infty$. Thus, for $\mathcal{A} \in K$ and $\overline{a}$ in $\mathcal{A}$, the quantifier rank of $\sigma_{\overline{a}}$ is $\leq m+s+1$. Let $\varphi$ be any $\textrm{L}^s_{\infty\omega}$-formula. Then the disjunction in the preceding corollary is a disjunction of formulas of quantifier rank $\leq m+s+1$ and hence, it is a finite one. This shows that (i) implies (ii). The implication from (ii) to (iii) is trivial. To show that (iii) implies (i) assume, by contradiction, that K is not $s$-bounded. Let $\mathcal{A}_0,\mathcal{A}_1,...$ be structures in $K$ of pairwise distinct $s$-rank. For $M\subseteq \mathbb{N}$ let
$$\varphi_M := \bigvee \{\sigma_{\mathcal{A}_i}\;|\; i\in M\}.$$
By the previous lemma, if $L,M\subseteq \mathbb{N}$ and $L\neq M$ then $K\nvDash \varphi_L \leftrightarrow \varphi_M$. Hence on $K$, $\textrm{L}^s_{\infty\omega}$ contains uncountably many pairwise nonequivalent sentences and is therefore more expressive than $\textrm{FO}$.
\end{proof}

\begin{theorem}\label{t3}\cite{Flum} $\newline$
Let $K$ be a class of structures. Let $K$ be \emph{fixed-point bounded}, i.e., for any first-order formula $\varphi(X,\overline{x})$ positive in $X$ with free variables among $\overline{x}$, $X$, there is an $m_0$ such that
                                       $$K \vDash \forall \overline{x} (\varphi^{m_0 +1} (\overline{x})\rightarrow \varphi^{m_0}(\overline{x}))$$
Where $\varphi^m(\overline{x})$ is a first-order formula defining $F_m^{\varphi}$. $\newline$
Then the following are equivalent :
\begin{enumerate}
\item[(i)] $K$ is fixed-point bounded.
\item[(ii)]$K$ is bounded.
\item[(iii)] On $K$, every $\textrm{L}_{\infty\omega}^{\omega}$-formula is equivalent to an $\textrm{FO}$-formula.
\end{enumerate}
\end{theorem}
\begin{proof} $\newline$
$(i) \Rightarrow (ii)$ : $\newline$
Assume that $K$ is fixed-point-bounded. Recall the formulas :
$$\varphi_{s,\nsim}(Z,\overline{x},\overline{y}) := \underset{\underset{\psi \; \text{atomic}}{\psi\in \textrm{FO}^s}}{\bigvee} (\psi(\overline{x})\leftrightarrow \neg\psi(\overline{y})) \vee \underset{1\leq i \leq s}{\bigvee} (\exists x_i \forall y_i Z \overline{x}\overline{y} \vee \exists y_i \forall x_i Z \overline{x}\overline{y})$$
where $\overline{x} := x_1 \ldots x_s,\; \overline{y} := y_1 \ldots y_s$. $\newline$
Since $K$ is fixed-point bounded then there is $m_0$ such that
                    $$K \vDash \forall \overline{x}\forall \overline{y} (\varphi_{s,\nsim}^{m_0 +1}(\overline{x},\overline{y}) \rightarrow \varphi_{s,\nsim}^{m_0}(\overline{x},\overline{y}) )$$
this is equivalent to $K \vDash \forall \overline{x}\forall\overline{y} (\neg \varphi_{s,\nsim}^{m_0}(\overline{x},\overline{y}) \rightarrow \neg \varphi_{s,\nsim}^{m_0 +1}(\overline{x},\overline{y}))$.
Thus for every $\mathcal{A} \in K$, every $\overline{a},\overline{b}\in A^s$, if Duplicator wins $G_{m_0}^s(\mathcal{A},\overline{a},\mathcal{A},\overline{b})$ then it wins $G_{m_0 +1}^s(\mathcal{A},\overline{a},\mathcal{A},\overline{b})$.
Hence for every $\mathcal{A} \in K$, $W_{m_0}^s(\mathcal{A},\mathcal{A}) = W_{m_0 +1}^s(\mathcal{A},\mathcal{A})$. So $K$ is $s$-bounded. $s$ was arbitrary, so $K$ is $s$-bounded for every $s$, i.e., $K$ is bounded.$\newline$
$(ii)\Leftrightarrow(iii)$ : was proved in Theorem \ref{3.3.12}. $\\$
$(iii)\Rightarrow(i)$ : We prove ($\neg (i) \Rightarrow \neg (iii)$)

Let $\varphi(X,\overline{x})$ be a first-order formula positive in $X$ with free variables among $\overline{x}$, $X$. Assume that there is no $m$ such that
                                          $$K\vDash \forall \overline{x} (\varphi^{m+1} (\overline{x})\rightarrow \varphi^{m}(\overline{x}))$$
i.e. for every $m$, there is $\mathcal{A}_m\in K$ such that
                  $$\mathcal{A}_m \vDash \exists \overline{x} (\varphi^{m+1} (\overline{x}) \wedge \neg \varphi^{m}(\overline{x})) \;\;\;\;\;\;\;\;\; (*)$$
For any $M \subseteq \mathbb{N}$, set $\varphi_M := \underset{m\in M}{\bigvee} (\varphi^{m+1} (\overline{x}) \wedge \neg \varphi^{m}(\overline{x}))$. If $M \neq L$, we may assume without loss of generality that there is $m_0 \in M \backslash L$. From $(*)$ $\mathcal{A}_{m_0} \vDash \varphi^{m_0 +1} (\overline{a}) \wedge \neg \varphi^{m_0}(\overline{a})$ for some $\overline{a}$ from $A_{m_0}$. Then $\mathcal{A}_{m_0} \vDash \varphi_{M}(\overline{a})$, and $\mathcal{A}_{m_0} \vDash \varphi^l(\overline{a})$ for all $l \geq m_0 +1$, and $\mathcal{A}_{m_0} \vDash \neg\varphi^l(\overline{a})$ for all $l \leq m_0$. $\newline$
For any $l\in L$ either $l > m_0$ or $l < m_0$, thus $\mathcal{A}_{m_0} \nvDash \varphi^{l+1}(\overline{a}) \wedge \neg \varphi^l (\overline{a})$ for all $l \in L$. Hence $\mathcal{A}_{m_0} \nvDash \varphi_L$. $\newline$
Thus for $L\neq M$, $\varphi_M$ and $\varphi_L$ are not equivalent on $K$. So on $K$ there are uncountably many pairwise nonequivalent $\textrm{L}_{\infty\omega}^{\omega}$-sentences, hence the expressive power of $\textrm{L}_{\infty\omega}^{\omega}$ on $K$ is greater than that of $\textrm{FO}$.
\end{proof}

Inspired by this we introduce the following more general definition for fixed-point bounded-ness :
\begin{definition} $\newline$
For a class $K$ of structures and polynomially bounded $t(n)$, we say that $K$ is $t(n)$-fixed-point bounded if for any first-order formula $\varphi(X,\overline{x})$ positive in $X$ with free variables among $\overline{x}$, $X$, there is a constant number $c$ such that for every positive integer $n$, and every structure $\mathcal{A}$ of size $n$ in $K$, the depth of $\varphi$ in $\mathcal{A}$ is $\leq ct(n)$.
\end{definition}
and then we suggest the following conjecture which is a generalization of the previous theorem : $\newline$

\begin{conjecture}\label{c3}
For every class $K$ of finite structures, every polynomially bounded $t(n)$, the following are equivalent:
\begin{enumerate}
\item[(i)] $K$ is $t(n)$-fixed-point bounded.
\item[(ii)] For every $s$, there is a constant number $c$ such that for every positive integer $n$, every structure $\mathcal{A}$ of size $n$ in $K$, $$r(s,\mathcal{A})\leq ct(n).$$
\item[(iii)] On $K$, every $\textrm{L}_{\infty\omega}^{\omega}$-formula is equivalent to an $\textrm{IND}[t(n)]$-formula.
\end{enumerate}
\end{conjecture}

\newpage
\thispagestyle{empty}
\mbox{}
.

\chapter{Conclusion and Future Research}
\section{Conclusion}
Depth is an important complexity measure that has several characterizations :
\begin{enumerate}
\item[(1)] It equals the number of iterations of a first-order quantifier block when the inductive definition is positive, Theorem \ref{2.2.14}.
(If the inductive definition is not positive, and hence its depth may be exponential, then the inductive definition is expressible by a first-order quantifier block iterated exponentially, Theorem \ref{2.2.11}).
\item[(2)] The class of boolean queries definable by positive inductive definitions (which are of polynomial depth) is exactly the class of boolean queries accepted by a deterministic polynomial-time Turing machine, Theorem \ref{2.2.1}.
\item[(3)] It equals parellel-time when the inductive definition is positive, Theorem \ref{2.2.14}.
\item[(4)] It equals circuit depth in Circuit Complexity, Theorem \ref{2.2.14}.
\end{enumerate}

Also number of variables, quantifier rank, number of quantifiers, number of $\vee$-symbols, and size of formulas in general, are all important complexity measures that are closely related to inductive definitions. We have seen in Theorem \ref{2.2.17} that the number of variables in an inductive definition is related to the number of processors in the corresponding $\textrm{CRAM}$, and that quantifier rank and number of quantifiers increase with iterations in an inductive definition. In the discussion after Definition \ref{3.3.8} we have seen that the class of structures definable by the fixed-point of a first-order formula of depth $t(n)$ is in $\textrm{FO}_{\bigvee}^{k}[2^{ct(n)}]$ for some constant number $c$. This shows roughly how many $\vee$-symbols an inductive definition of a certain depth needs. We have also seen that

$$\underset{k\geq 1}{\bigcup} \textrm{FO}_{\bigvee}[n^k+\left(\frac{n^k}{\log{n}}\right)^2,\frac{n^k}{\log{n}}] = \textrm{FO}[2^{n^{O(1)}}] = \textrm{PSPACE} = \textrm{FO}(\textrm{PFP})$$
in Corollary \ref{3.3.13}. In his paper "Number of Quantifier is Better than Number of Tape cells"\cite{NumQuan}, Immerman showed that $$\textrm{NSPACE}[f(n)] \subseteq \textrm{QN}[\frac{(f(n))^2}{\log{n}}]\subseteq \textrm{DSPACE}[(f(n))^2]$$
and hence as a corollary we have
$$\textrm{QN}[n^{O(1)}]=\textrm{FO}[2^{n^{O(1)}}] = \textrm{PSPACE} = \textrm{FO}(\textrm{PFP}) $$
this relates roughly number of quantifiers to depth.

\section{Open Problems}
We have left open some questions and a conjecture :
\begin{enumerate}
\item[(1)] In Section 2.3, we asked whether it is possible for two formulas, positive in the relation variables through which the induction is made, that the computation of their simultaneous fixed-point takes more time or more iterations than the computation of one of their nested fixed-points.
\item[(2)] In Section 3.4, we asked whether the fixed-point of the formula $\small{\varphi_{s,\nsim}(x_1,\ldots,x_s,y_1,\ldots,y_s,Z)}$ can be expressed as the fixed-point of a formula of a smaller depth than $n^{2s}$. If not, then we have the strict hierarchy :
$$\textrm{IND}[n^2]\subsetneq \textrm{IND}[n^4] \subsetneq \textrm{IND}[n^6] \subsetneq \ldots.$$
\item[(3)] Directly after the previous question we asked the question : Is there a structure $\mathcal{A}$ for which $r(s,\mathcal{A})$ is $(||\mathcal{A}||+1)^{2s}$? $\newline$
The answer of this question should help answer the previous question.
\item[(4)] Conjecture \ref{c3}, which is a generalization of Theorem \ref{t3}, $\newline$
For every class $K$ of finite structures, every polynomially bounded $t(n)$, the following are equivalent:
\begin{enumerate}
\item[(i)] $K$ is $t(n)$-fixed-point bounded.
\item[(ii)] For every $s$, there is a constant number $c$ such that for every positive integer $n$, every structure $\mathcal{A}$ of size $n$ in $K$, $$r(s,\mathcal{A})\leq ct(n).$$
\item[(iii)] On $K$, every $\textrm{L}_{\infty\omega}^{\omega}$-formula is equivalent to an $\textrm{IND}[t(n)]$-formula.
\end{enumerate}

\end{enumerate}

\newpage
\thispagestyle{empty}
\mbox{}
.

\end{document}